\documentclass[12pt]{amsart}
\usepackage[a4paper, left=28mm, right=28mm, top=28mm, bottom=34mm]{geometry}
\usepackage[all]{xy}
\usepackage{amsmath, amssymb, amsthm}
\usepackage{mathrsfs}
\usepackage{color}
\usepackage{aliascnt}
\usepackage{braket}
\usepackage{mathtools}
\usepackage{enumerate}
\usepackage{enumitem}
\usepackage{tikz-cd}

\mathtoolsset{showonlyrefs=true} 

\usepackage[colorlinks=true]{hyperref}
\usepackage{enumitem}
\setlist[enumerate]{itemsep=0pt,label=$(\mathrm{\arabic*})$,leftmargin=*}
\setlist[itemize]{itemsep=0pt, topsep=5pt, labelindent=\parindent,leftmargin=27pt}

\usepackage[OT2,T1]{fontenc}
\DeclareSymbolFont{cyrletters}{OT2}{wncyr}{m}{n}
\DeclareMathSymbol{\Sha}{\mathalpha}{cyrletters}{"58}

\newtheorem{thm}{Theorem}[section]

\newaliascnt{lem}{thm}
\newtheorem{lem}[lem]{Lemma}
\aliascntresetthe{lem}

\newaliascnt{cor}{thm}
\newtheorem{cor}[cor]{Corollary}
\aliascntresetthe{cor}

\newaliascnt{prop}{thm}
\newtheorem{prop}[prop]{Proposition}
\aliascntresetthe{prop}

\newtheorem{claim}{Claim}

\newtheorem*{oneclaim}{Claim}

\theoremstyle{definition}

\newaliascnt{rem}{thm}
\newtheorem{rem}[rem]{Remark}
\aliascntresetthe{rem}

\theoremstyle{definition}
\newaliascnt{defn}{thm}
\newtheorem{defn}[defn]{Definition}
\aliascntresetthe{defn}

\newaliascnt{ex}{thm}
\newtheorem{ex}[ex]{Example}
\aliascntresetthe{ex}

\numberwithin{equation}{section}

\def\F{{\mathbb F}}
\newcommand{\Fp}{\mathbb{F}_p}

\def\Q{{\mathbb Q}}
\def\Qp{\Q_p}

\def\Z{{\mathbb Z}}
\newcommand{\Zp}{\mathbb{Z}_p}
\def\C{{\mathbb C}}

\def\cF{{\mathcal F}}

\def\cH{{\mathcal H}}

\def\aa{{\mathfrak a}}

\def\oo{{\mathfrak o}}

\def\cO{{\mathcal O}}
\def\cA{{\mathcal A}}

\def\End{\mathop{\mathrm{End}}\nolimits}
\def\Cl{\mathop{\mathrm{Cl}}\nolimits}
\def\Sel{\mathop{\mathrm{Sel}}\nolimits}

\def\Aut{\mathop{\mathrm{Aut}}\nolimits}
\def\Ann{\mathop{\mathrm{Ann}}\nolimits}
\newcommand{\Coker}{\operatorname{Coker}}

\def\Frob{\mathop{\mathrm{Frob}}\nolimits}

\def\Gal{\mathop{\mathrm{Gal}}\nolimits}

\def\Hom{\mathop{\mathrm{Hom}}\nolimits}
\def\Ind{\mathop{\mathrm{Ind}}\nolimits}

\def\Fitt{\mathop{\mathrm{Fitt}}\nolimits}
\def\corank{\mathop{\mathrm{corank}}\nolimits}

\def\Im{\mathop{\mathrm{Im}}\nolimits}

\newcommand{\Ker}{\operatorname{Ker}}

\def\Tr{\mathop{\text{\rm Tr}}\nolimits}

\def\rank{\mathop{\text{\rm rank}}\nolimits}

\def\det{\mathop{\mathrm{det}}\nolimits}

\def\ord{\mathop{\mathrm{ord}}\nolimits}

\newcommand{\ab}{\mathrm{ab}}
\newcommand{\ur}{\mathrm{ur}}
\renewcommand{\div}{\mathrm{div}}
\newcommand{\res}{\mathrm{res}}
\newcommand{\loc}{\mathrm{loc}}
\newcommand{\ds}{\displaystyle}
\newcommand{\Leg}[2]{\left(\dfrac{#1}{#2}\right)}

\begin{document}

\title[Class groups and 
Iwasawa theory of elliptic curves]{
Asymptotic behavior of class groups and 
cyclotomic Iwasawa theory of elliptic curves}

\author{Toshiro Hiranouchi}
\thanks{The work of the first author 
 is supported by JSPS KAKENHI 20K03536.}
\address{Department of Basic Sciences, Graduate School of Engineering, 
Kyushu Institute of Technology, 
1-1 Sensui-cho, Tobata-ku, Kitakyushu-shi, 
Fukuoka 804-8550 JAPAN}
\email{hira@mns.kyutech.ac.jp}

\author{Tatsuya Ohshita}
\thanks{
The work of the second author 
is supported by JSPS KAKENHI 18H05233, 20K14295 and 21K18577.}

\address{Department of Mathematics, 
Cooperative Faculty of Education, Gunma University,
Maebashi, Gunma 371-8510, Japan}
\email{ohshita@gunma-u.ac.jp}

\date{\today}

\begin{abstract}
In this article, we study a relation between 
certain quotients of 
ideal class groups 
and the cyclotomic Iwasawa module $X_\infty$  
of the Pontrjagin dual of the fine Selmer group of 
an elliptic curve $E$ defined over $\mathbb{Q}$.
We consider the Galois extension field $K^E_n$ 
of $\mathbb{Q}$ generated by coordinates of 
all $p^n$-torsion points of $E$, and introduce a quotient 
$A^E_n$ of the $p$-sylow subgroup of the ideal class group of $K^E_n$
cut out by the modulo $p^n$ Galois representation $E[p^n]$. 
We describe the asymptotic behavior of $A^E_n$ 
by using the Iwasawa module $X_\infty$. 
In particular, under certain conditions, 
we obtain an asymptotic formula as Iwasawa's class number formula 
on the order of $A^E_n$ by using Iwasawa's invariants of $X_\infty$. 
\end{abstract}

\maketitle

\section{Introduction}\label{secintro}

Let $E$ be an elliptic curve over $\Q$.
For each $N \in \Z_{>0}$, we denote 
by $E[N]$ the subgroup of $E(\overline{\Q})$
consisting of elements annihilated by $N$.
Fix an odd prime number $p$ at which $E$ has good reduction. 
For each $n \in \Z_{> 0}$, 
we put $K^E_n:=\Q (E[p^n])$, and 
$h_n:= \ord_p \# (\Cl(\cO_{K^E_n}) \otimes_\Z \Z_p)$, 
where $\ord_p$ denotes the additive $p$-adic valuation 
normalized by $\ord_p(p)=1$ and $\Cl(\cO_{K^E_n})$ is the ideal class group of 
the ring of integers $\cO_{K^E_n}$. 
In recent papers \cite{SY1}, \cite{SY2}, and \cite{Hi}, 
there has been renewal of interest in 
an asymptotic behavior of the class numbers 
$\set{h_n}_{n\ge 0}$ along the tower of number fields $K_n^E$. 
It has been shown that 
an asymptotic inequality which gives a \emph{lower bound} of $\set{h_n}_{n\ge 0}$ 
in terms of the Mordell-Weil rank $\rank_{\Z} E(\Q)$ of $E$ (cf.\ \autoref{rem_comparison}). 
For some generalizations of these results including abelian varieties over a number field, 
see \cite{Ga} and \cite{Oh}. 
In these works, the divisible part of 
\emph{the fine Selmer group} $\Sel_p (\Q, E[p^\infty])$ (cf.~\autoref{defnSel}) 
plays important roles.

We define a quotient 
$A^E_n$ of  
$\Cl(\cO_{K^E_n}) \otimes_\Z \Z_p$,  
which is cut out by the Galois representation $E[p^n]$ (see \eqref{def:AnE} below). 
In this paper, 
we shall 
describe the asymptotic behavior of $A^E_n$ 
by using the fine Selmer group  $\Sel_{p}(K_n,E[p^n])$, 
where we put $K_n:=\Q(\mu_{p^n})$. 
As an application of our result, 
we shall show an asymptotic formula on the order of $A^E_n$ 
using Iwasawa's $\mu$ and $\lambda$-invariants of 
the cyclotomic Iwasawa module 
associated with the fine Selmer group of the elliptic curve $E$, 
as ``Iwasawa's class number formula'' (\cite{Iw}).

\subsection{The statements of the main results}

In order to state our main results, 
let us introduce some notation. 
For each $N\in \Z_{> 0}$, we denote by 
$\mu_N:=\boldsymbol{\mu}_N(\overline{\Q})$ 
the group of $N$-th roots of unity.
For each $m \in \Z_{\ge 0}$,
we define $K_m := \Q(\mu_{p^m})$ (in particular, we put $K_0:=\Q$), 
and set $K_{\infty}:=\bigcup_{m \ge 0} K_m$.
For each $m_1,m_2 \in \Z_{\ge 0} \cup \{ \infty \}$ with 
$m_2>m_1$, we set 
$\mathcal{G}_{m_2,m_1}:=\Gal(K_{m_2}/K_{m_1})$, 
and put  $\Delta:=\mathcal{G}_{1,0} \simeq (\Z/p\Z)^\times$.
For any $m \ge 1$, 
we have $\mathcal{G}_{m,0}=
\Delta \times \mathcal{G}_{m,1}$.
We can regard $\Zp[\Delta]$ as 
a subring of $\Z_p[\mathcal{G}_{m,0}]$. 
We put $\widehat{\Delta}:=
\Hom (\Delta, \Zp^\times)$. 
For each $\chi \in \widehat{\Delta}$, 
we define $\Zp(\chi):=\Zp$ to be 
the $\Zp[\Delta]$-algebra 
where $\Delta$ acts via $\chi$,  
and for a $\Zp[\Delta]$-module $M$,  
we set $M_\chi:=
M\otimes_{\Zp[\Delta]}\Zp(\chi)$. 
We have 
$M=\bigoplus_{\chi \in \widehat{\Delta}} M_\chi$ 
because $p$ is odd. 
For each $m,n \in \Z_{\ge 0}$, we define 
\[
R_{m,n}:=\Z/p^n\Z [\mathcal{G}_{m,0}]
=\Z_p/p^n\Z_p[\Gal(K_m/\Q)], 
\]
and put $R_n:=R_{n,n}$.
For each number field $L$, that is, a finite extension of $\Q$, 
and each $n\in \Z_{\ge 0} \cup \{ \infty \}$, 
let $\Sel (L,E[p^n])$ be the Selmer group in the classical sense, 
and $\Sel_p (L,E[p^n])$ the kernel of the localization map
\[
\Sel (L,E[p^n]) \longrightarrow \prod_{v \mid p}
H^1(L_v, E[p^n])
\]
which is called \emph{the fine Selmer group} 
(for details, see \autoref{defnSel} and
\autoref{remSelcl}, later). 
For each $m,n \in \Z_{\ge 0}$, 
the fine Selmer group 
$\Sel_p (K_m,
E[p^n])$ becomes an $R_{m,n}$-module.
For any $n \in \Z_{\ge 0}$, 
the field $K^E_n=\Q(E[p^n])$ 
contains $\mu_{p^n}$ and hence $K^E_n \supseteq K_n = \Q(\mu_{p^n})$
because of the Weil pairing 
$E[p^n] \times E[p^n] \longrightarrow 
\mu_{p^n}$ (\cite[Chapter~III, Corollary~8.1.1]{Si1}).
Let 
\begin{equation}
    \label{eq:rhoev}
(\rho_n^E)^\vee \colon \Gal(K^E_n/\Q)^{\rm op} 
\longrightarrow  \Aut_{\Z_p} (E[p^n]^{\vee})
=GL_2(\Z/p^n\Z)
\end{equation}
be the right action of 
$G_{\Q}$ on 
the Pontrjagin dual 
$E[p^n]^{\vee} := \Hom_{\Z_p}(E[p^n], \Z/p^n\Z)$ 
of $E[p^n]$.
We define a $\Z_p$-module  
$A_n^E$ by 
\begin{equation}
\label{def:AnE}
A_n^E:= (M_2(\Z/p^n\Z), (\rho_n^E)^\vee)
\otimes_{\Z[\Gal(K^E_n/K_n)]} 
\Cl(\cO_{K^E_n}[1/p]),
\end{equation}
where $M_2(\Z/p^n\Z)$ denotes the matrix algebra of degree two 
over $\Z/p^n\Z$. 
We can define a $\Z_p$-linear left action 
of $G_\Q := \Gal(\overline{\Q}/\Q)$ on $A_n^E$ by 
\[
\sigma (A \otimes [\aa]):=A  (\rho_n^E)^\vee(\sigma^{-1}) 
\otimes [\sigma \aa]
\]
for each $\sigma \in G_\Q$, 
$A\in M_2(\Z/p^n\Z)$ and 
$[\aa] \in \Cl(\cO_{K^E_n}[1/p])$.
Since every $\sigma \in G_{K_n}$ 
acts trivially on $A_n^E$, 
we may regard $A_n^E$ as an $R_n$-module. 
We denote by $(A_n^E)^{\vee} = \Hom_{\Zp}(A_n^E,\Z/p^n\Z)$ 
the Pontrjagin dual of $A_n^E$.
The following theorem is the main result of our paper.

\begin{thm}[\autoref{thmmainbody}]
\label{thmmain}
Let $E$ be an elliptic curve over $\Q$, and $p$ an odd prime number 
where $E$ has good reduction. 
Suppose that $E$ satisfies the following conditions 
{\rm (C1)}, {\rm (C2)} and {\rm (C3)}.
\begin{itemize}
\item[{\rm (C1)}] The Galois representation
\[
\rho^E_1  
\colon G_{K_{\infty}} := \Gal(\overline{\Q}/K_\infty) 
\longrightarrow \Aut_{\F_p}(E[p])
\simeq GL_2(\F_p)
\]
is absolutely irreducible over $\F_p$.
\item[{\rm (C2)}] 
For any $n \in \Z_{\ge 1}$ and 
any place $v$ of $K_n$ where the base change $E_{K_{n,v}}$ of $E$
has potentially multiplicative reduction, 
we have $E(K_{n,v})[p]=0$.
\item[{\rm (C3)}] 
If $E$ has complex multiplication, 
the ring $\End (E)$ of endomorphisms of $E$ 
defined over $\overline{\mathbb{Q}}$ is 
the maximal order of an imaginary quadratic field.
\end{itemize}
Then, there exists a family of $R_n$-homomorphisms 
\[
r_n:\Sel_p (K_n,
E[p^n])^{\oplus 2}\longrightarrow (A_n^E)^{\vee}
\]
such that 
the kernel $\Ker(r_n)$ and the cokernel $\Coker(r_n)$ are finite 
with order bounded independently of $n$. 
\end{thm}

\begin{rem}\label{rem:C1andC1str}
As we see \autoref{lemC1} below, 
the condition (C1) is satisfied if 
the following condition 
$\mathrm{(C1)}_{\mathrm{str}}$ holds: 
\begin{itemize}\setlength{\leftskip}{5mm}
    \item[$\mathrm{(C1)}_{\mathrm{str}}$] The Galois representation 
    \[
        \rho^E=\rho^{E,p} \colon G_\Q 
        \longrightarrow \Aut_{\Z_p}(T_p (E)) \simeq 
        GL_2(\Z_p)
    \] 
is surjective.
\end{itemize}
Note that if $E$ does not have complex multiplication, 
then 
the map $\rho^E$ is surjective for all but finitely many 
prime number $p$ 
by Serre's open image theorem \cite{Se3}. 
\end{rem}

\begin{rem}\label{rem:C2twist}
In \autoref{secC}, we show that 
for any elliptic curve $E$ over $\Q$, 
there exists a quadratic twist $E'/\Q$  of $E$ 
which satisfies the condition (C2) (\autoref{lemTwist}). 
\end{rem}

%

\begin{rem}\label{remendsurj}
If the condition (C1) for $E$ is satisfied, then 
the ring homomorphism
\(
    \Z_p[G_{K_\infty}] \longrightarrow 
    M_2(\F_p)
\)
induced by $\rho^E_1=(\rho^E \bmod p)$ 
is surjective, 
where $M_2(\Fp)$ is the matrix algebra of degree two 
over $\Fp$. 
Hence, with the aid of Nakayama's lemma for 
finitely generated $\Z_p$-modules, 
the condition (C1) for $E$ implies that 
the map 
\[ 
    (\rho_n^E)^\vee \colon 
    \Z_p[G_{K_n}] \longrightarrow 
    M_2(\Z/p^n\Z)
\]
induced by \eqref{eq:rhoev} 
is surjective. 
Under the assumption of (C1), we can regard 
$A^E_{n}$ as a quotient of 
$\Cl (\cO_{K^E_n})$.
\end{rem}

\begin{rem}\label{remAS}
For each $n\in \Z_{\ge 1}$, 
we define  an $R_n$-module
\[
S_{n}:=
\Hom_{\Z_p[\Gal(K^E_n/K_n)]}
(\Cl(\cO_{K^E_n}[1/p]) \otimes_{\Z} \Z_p, E[p^n]  ).
\]
In \autoref{secpfmain}, we shall prove 
\autoref{thmmain} by constructing $\Gal(K_n/\Q)$-homomorphisms 
\[
\Sel_p (K_n,
E[p^n])^{\oplus 2} \longrightarrow  S_{n}^{\oplus 2} \xleftarrow{\ \simeq\ } (A_n^E)^{\vee}, 
\]
where the orders of the kernel and the cokernel of the former map are bounded and the latter is an isomorphism. 
\end{rem}

\begin{rem}
In \cite{PS}, 
under certain assumptions on $(E,p)$, 
Prasad and Shekhar studied a relation between  
$\Sel_{p}(\Q ,E[p])$ and 
\(
\widetilde{S}:=
\Hom_{\Z_p}
(\Cl(\cO_{K^E_1}) \otimes_{\Z} \F_p, E[p]  ).
\)
Here, we give a remark on 
a relation between $\widetilde{S}$ and our $A^E_1$. 
Let $\boldsymbol{1} \in \widehat{\Delta}$ be the trivial character. 
Note that $S_{1,\boldsymbol{1}}$
in the sense of \autoref{remAS} is 
an $\F_p$-subspace of  
$\widetilde{S}$. Moreover, if $E(\Qp)[p]=\{ 0 \}$, 
then the natural injection 
$S_{1,\boldsymbol{1}} \hookrightarrow \widetilde{S}$
becomes an isomorphism. 
Indeed, in such case, 
for any $f \in \widetilde{S}$ and 
any prime ideal $\mathfrak{p}$ of $K^E_1$, 
it follows from the comparison of the action of the 
decomposition group at $\mathfrak{p}$ in $\Gal(K^E_1/\Q)$ 
that we have $f([\mathfrak{p}] \otimes 1)=0$.
Hence by \autoref{remAS}, we deduce that 
if $E(\Qp)[p]=\{ 0 \}$, then we have 
$A^E_{1,\boldsymbol{1}} \simeq \widetilde{S}^{\oplus 2}$.
\end{rem}

Here, we shall note that 
\autoref{thmmain} gives a description of 
the asymptotic behavior of 
the higher Fitting ideals of the $\Zp$-modules $A_n^E$.
Let $M$ be a finitely generated $\Z_p$-module.
For each $i \in \Z_{\ge 0}$, we denote by 
$\Fitt_{\Z_p,i}(M)$ the $i$-th Fitting ideal of $M$ (cf.~\autoref{defnFitt}), 
and put
\[
\Phi_i(M):=\ord_p (\Fitt_{\Z_p,i}(M)) 
\in \Z_{\ge 0} \cup \set{ \infty }.
\]
The sequence  $\set{ \Phi_i(M) }_{i \ge 0}$
determines the isomorphism class of 
the $\Z_p$-module $M$ (see \autoref{remPhiisom}). 
There is an equality 
$\Phi_i(A_{n,\chi}^E)= \Phi_i((A^E_{n,\chi^{-1}})^{\vee})$ for 
any $\chi \in \widehat{\Delta}$ because 
$A_{n,\chi}^E$ is non-canonically isomorphic to 
$(A^E_{n,\chi^{-1}})^{\vee} = 
\Hom_{\Z_p}(A_{n, \chi^{-1}}^E, \Z/p^n\Z)$ 
as a $\mathbb{Z}_p$-module.  
Similarly, we have $\Phi_i(A_n^E)= \Phi_i((A_n^E)^{\vee})$.   
%
Let $\set{ a_n}_n$ and $\set{ b_n }_n$ be sequences of real numbers.
we write $a_n \succ b_n$ if it holds that
$\liminf_{n \to \infty} (a_n-b_n) > -\infty$, 
namely, if the sequence $\set{a_n - b_n}_{n}$ is bounded below.
If $a_n \succ b_n$ and $b_n \succ a_n$, then 
we write $a_n \sim b_n$. 
For a family of homomorphisms 
$f_n\colon M_n\longrightarrow M_n'$ 
of finitely generated torsion $\Zp$-modules 
if the order of $\Ker(f_n)$ and that of $\Coker(f_n)$ 
are bounded independently of $n$, 
then we have $\Phi_i(M_n)\sim \Phi_i(M_n')$ for any $i\in \Z_{\ge 0}$ (\autoref{lemasymPhi}). 
\autoref{thmmain} implies the following corollary:

\begin{cor}
\label{cormain}
Let $E$ be an elliptic curve over $\Q$, and $p$ an odd prime number 
where $E$ has good reduction. 
Suppose that $E$ satisfies the conditions 
{\rm (C1)}, {\rm (C2)} and {\rm (C3)}.
Then, for any $i \in \Z_{\ge 0}$ and $\chi \in \widehat{\Delta}$, 
it holds 
\[
\Phi_i(A_{n,\chi}^E) = \Phi_i((A_{n,\chi^{-1}}^{E})^{\vee}) 
\sim \Phi_i \left(\Sel_p (K_n,
E[p^n])^{\oplus 2}_{\chi^{-1}} \right), 
\]
and moreover, we have
$\Phi_i(A_{n}^E) = \Phi_i((A_{n}^{E})^{\vee}) 
\sim \Phi_i \left(\Sel_p (K_n,
E[p^n])^{\oplus 2} \right)$.
\end{cor}

\subsection{Asymptotic formulas as Iwasawa's class number formula}

For each $\chi\in \widehat{\Delta}$, 
we put $h_{n,\chi}^E:=\ord_p(\# A^E_{n,\chi})$, 
and $h_{n}^E:=\ord_p(\# A^E_{n})=\sum_{\chi \in \widehat{\Delta}} 
h^E_{n,\chi}$. 
Since  $A^E_{n}$ is a quotient of $\Cl(\cO_{K^E_n})$ 
as noted in \autoref{remendsurj}, 
we have 
\[
h_n:= \ord_p(\# \Cl(\cO_{K^E_n}) \otimes \Zp) 
\ge h^E_n. 
\] 
As we shall see below, 
\autoref{cormain} for $i=0$ gives 
a description of asymptotic behavior of $h^E_n$
like ``Iwasawa's class number formula''. 
Let us introduce Iwasawa theoretic notation. 
We put $\Gamma:=\mathcal{G}_{\infty,1}= \Gal(K_\infty/K_1)$.
There is a non-canonical isomorphism 
$\Gamma \simeq \Z_p$ 
and fix a toplogical generator $\gamma \in \Gamma$.
We set $\Lambda:=\Z_p[\![\Gamma]\!]$.
There exists an isomorphism 
$\Lambda \xrightarrow{\ \simeq \ } \Z_p[\![T]\!]$
of $\Z_p$-algebras sending $\gamma$ to $1+T$. 
For each $m,n \in \mathbb{Z}_{> 0}$, we define 
\[
    \Lambda_{m,n}:=\Z/p^n\Z [\mathcal{G}_{m,1}] \simeq 
    \Lambda/(p^n, \gamma^{p^{m-1}}-1), 
\]
and put $\Lambda_n:=\Lambda_{n,n}$. 
Since we have $\mathcal{G}_{m,0}=\Delta \times \mathcal{G}_{m,1}$, 
the equality $R_{m,n}=\Lambda_{m,n}[\Delta]$ holds.
In the following, we introduce the Iwasawa module of 
the Pontrjagin dual of the fine Selmer groups.
Write  
\[
    \Sel_p (K_{\infty},E[p^\infty])
    :=\varinjlim_{m} 
    \Sel_p (K_m,E[p^{\infty}]). 
\] 
For any $m,n \in \Z_{\ge 0} \cup \{ \infty \}$, 
define  
\[
    X_{m,n} := \Sel_p (K_{m},E[p^n])^\vee:=\Hom_{\Zp}\left( 
    \Sel_p (K_{m},E[p^n]), 
    \Qp/\Zp 
    \right),  
\]
and put $X_{n}:=X_{n,n}$. 
It is known that 
the $\Lambda$-module $X_{\infty}$ 
is finitely generated and torsion (\cite{Ka}). 
Take any $\chi \in \widehat{\Delta}$. 
The control theorem of 
the fine Selmer groups (\autoref{corpropcontrthm}) 
implies that  
\begin{equation}
    \label{eq:Phi0}
    \Phi_0 (X_{\infty,\chi} \otimes_{\Lambda} \Lambda_{n}) 
    \sim \Phi_0 (X_{n,\chi}) 
    \sim \Phi_0 \left(\Sel_p (K_n,E[p^n])^{\oplus 2}_{\chi^{-1}} \right). 
\end{equation}
Since $X_{\infty,\chi}$ 
is a finitely generated torsion $\Lambda$-module, 
we can define 
Iwasawa's $\mu$ and $\lambda$-invariants 
$\mu(X_{\infty,\chi})$ and 
$\lambda(X_{\infty, \chi})$
of the $\Lambda$-module $X_{\infty, \chi}$:  
for any finitely generated torsion $\Lambda$-module $M$, 
the characteristic ideal 
$\mathrm{char}_{\Lambda}(M)$ 
of the $\Lambda$-module $M$ is 
generated by an element 
$p^{\mu(M)} f (\gamma -1) \in \Lambda$ 
for some distinguished polynomial $f(T) \in \Z_p[T]$
of degree $\lambda(M)$.
By \eqref{eq:Phi0}, we have 
\[
\Phi_0 (X_{\infty,\chi} \otimes_{\Lambda} \Lambda_{n})
\sim 
\mu(X_{\infty,\chi}) p^n + \lambda(X_{\infty,\chi}) n.
\]
\autoref{cormain} for $i=0$ 
combined with the additivity of $\Phi_0$,  
$\mu$ and $\lambda$,  
implies the following.

\begin{cor}\label{cor:maintheoremcontrol}
Let $E$ be an elliptic curve over $\Q$, and $p$ an odd prime number 
where $E$ has good reduction. 
Suppose that $E$ satisfies 
the conditions {\rm (C1)}, 
{\rm (C2)}  and {\rm (C3)}.
Then, for any $\chi \in \widehat{\Delta}$, we have 
\[
h_{n,\chi}^E \sim 2\left(
\mu(X_{\infty,\chi}) p^n 
+ \lambda(X_{\infty, \chi }) n
\right), 
\]
and moreover, 
\(
h_{n}^E \sim 2\left(
\mu(X_{\infty}) p^n + \lambda(X_{\infty}) n
\right).
\)
\end{cor}

As we note below, by assuming the Iwasawa main 
conjecture for elliptic curves, the constants 
$\mu(X_{\infty})$ and $\lambda(X_{\infty})$ 
are described in terms of Kato's Euler systems.
Let us recall  
the Iwasawa main conjecture 
(in the formulation 
using Kato's Euler systems).
By using Euler systems of Beilinson--Kato elements, 
Kato constructed a $\Lambda$-submodule $Z$ of 
$\mathbf{H}^1$, where we set 
$$
\mathbf{H}^q
=\mathbf{H}^q(T_p(E)):=
\varprojlim_{m}H^1 \left(
K_m, T_p(E) 
\right) 
$$
for each $q \in \mathbb{Z}_{\ge 0}$
(or the construction of $Z$, 
see \cite[Theorem~12.6]{Ka} for the  Galois representation 
$T=T_p(E) \subseteq V_{\Q_p}(f_E)$,  
where $f_E$ is the cuspform attached to $E$). 
The Iwasawa main conjecture for $(f_E,p, \chi)$ 
with $\chi \in \widehat{\Delta}$ 
in the sense of \cite[Conjecture~12.10]{Ka} 
(combined with \cite[Theorem~12.6]{Ka})
predicts the equality
\begin{equation}
    \label{eq:IMC}
    \mathrm{char}_{\Lambda}(
    \mathbf{H}^2_\chi)   
    =\mathrm{char}_{\Lambda}(
    \mathbf{H}^1_\chi /Z_\chi).   
\end{equation}
Since $E$ has good reduction at $p$,
for the left hand side of \eqref{eq:IMC}, 
we have 
$\mathrm{char}_{\Lambda}(X_{\infty,\chi})
=\mathrm{char}_{\Lambda}(
\mathbf{H}^2_\chi) $  
because of the following: 
\begin{itemize}
    \item By the limit of the 
    Poitou-Tate exact sequence, our $X_\infty$ coincides with 
    \[\mathbf{H}^2(T_p(E))_0:=
    \Ker\left( \mathbf{H}^2 \longrightarrow 
    \mathbf{H}^2_{\mathrm{loc}}:=\varprojlim_m H^1(\Qp (\mu_{p^m}), T_p(E))\right)
    \]  
    (see, for instance, the proof of \cite[Proposition 3.17]{Oh2}).
    \item When $E$ has good reduction at $p$, 
    the local duality of the Galois cohomology 
    and Imai's result \cite{Im} imply that the order of 
    $\mathbf{H}^2_{\mathrm{loc}}$ is finite, and hence 
    the index of $\mathbf{H}^2(T_p(E))_0 $ in $\mathbf{H}^2(T_p(E))$ 
    is finite.
\end{itemize}
By using the Euler systems, 
Kato proved that  
the half side of \eqref{eq:IMC}, that is, the inclusion 
\[
    \mathrm{char}_{\Lambda}(
    \mathbf{H}^2_{0,\chi} )
    \supseteq 
    \mathrm{char}_{\Lambda}(
    \mathbf{H}^1_\chi /Z_\chi)   
\]
holds for any $\chi \in \widehat{\Delta}$ 
under the following condition 
which is satisfied when $\mathrm{(C1)_{str}}$ holds:
\begin{itemize}
    \item[] The image of 
    the Galois representation 
    \[
        \rho^E\vert_{G_{K_\infty}} \colon G_{K_\infty} 
        \longrightarrow \Aut_{\Z_p}(T_p (E)) \simeq GL_2(\Z_p)
    \] 
    contains $SL_2(\Zp)$
\end{itemize}
(See \cite[Theorem~13.4]{Ka}. Note that 
$\mathrm{(C1)_{str}}$ implies the assumption (3) in 
\cite[Theorem~13.4]{Ka}). 
By summarizing all $\chi$-parts, 
the following corollary follows from 
\autoref{cor:maintheoremcontrol}.

\begin{cor}
Let $E$ be an elliptic curve over $\Q$, and $p$ an odd prime number 
where $E$ has good reduction. 
\begin{enumerate}
    \item Suppose that $E$ satisfies 
    the conditions $\mathrm{(C1)_{str}}$ and {\rm (C2)}.
    Then, we have 
    \[
        h_{n}^E \prec 2\left(
        \mu(\mathbf{H}^1/Z) p^n + 
        \lambda(\mathbf{H}^1/Z) n
        \right).
    \]
    \item Suppose that $E$ satisfies 
    the conditions $\mathrm{(C1)}$, 
    {\rm (C2)}  and {\rm (C3)}. 
    Let $\chi_0 \in \widehat{\Delta}$. 
    Then, if the Iwasawa main conjecture for $(f_E,p,\chi_0)$ holds, 
    we have 
    \[
    h^E_{n,\chi_0} \sim 2\left(
    \mu(\mathbf{H}^1_{\chi_0}/Z_{\chi_0}) p^n + 
    \lambda(\mathbf{H}^1_{\chi_0}/Z_{\chi_0}) n
    \right).
    \]
    In particular, if the Iwasawa main conjecture for $(f_E,p,\chi)$ holds for every $\chi \in \widehat{\Delta}$, then we have    
    \[
    h_{n}^E \sim 2\left(
    \mu(\mathbf{H}^1/Z) p^n + 
    \lambda(\mathbf{H}^1/Z) n
    \right).
    \]
\end{enumerate}
\end{cor}

Let $\boldsymbol{1} \in \widehat{\Delta}$ be 
the trivial character. 
In \cite{SU}, 
Skinner and Urban proved the Iwasawa main conjecture for $(f_E,p,\boldsymbol{1})$ 
with the following conditions: 
\begin{itemize}
    \item The pair $(E,p)$ satisfies $\mathrm{(C1)_{str}}$. 
    \item The elliptic curve $E$ has good ordinary reduction at $p$. 
    \item There exists a prime number $\ell_0$ where $E$ has multiplicative reduction. 
\end{itemize}
(See \cite[Theorem~3.33]{SU}.) 
These conditions are satisfied 
when $E$ is semistable, and $p$ is a prime number  of good ordinary reduction 
satisfying $p \ge 11$ 
(see \cite[Theorem~3.34]{SU}).  
We obtain the following corollary. 

\begin{cor}
Suppose that $E$ is semistable, and 
let $p$ be a prime number with $p \ge 11$ 
where $E$ has good ordinary reduction. 
If $E$ satisfies the condition {\rm (C2)}, 
then we have
    \[
    h^E_{n,\boldsymbol{1}} \sim 2\left(
    \mu(\mathbf{H}^1_{\boldsymbol{1}}/Z_{\boldsymbol{1}}) p^n + 
    \lambda(\mathbf{H}^1_{\boldsymbol{1}}/Z_{\boldsymbol{1}}) n
    \right).
    \]
\end{cor}

Let us see the relation between 
our results and previous works on 
the asymptotic behavior of $h_n$.
By the arguments in \cite[\S 4.1]{Oh}, 
for any number field $L$, 
we have   
$$ 
\corank_{\Zp} \Sel_p (L, E[p^\infty])
\ge \rank_{\Z} E(L) - [L: \Q ].
$$
(Indeed, the fine Selmer group 
$\Sel_p (L, E[p^\infty])$ contains 
the kernel of 
$$
E(L) \otimes_{\mathbb{Z}}
\mathbb{Q}_p/\mathbb{Z}_p
\longrightarrow 
E(L \otimes_{\mathbb{Q}} \mathbb{Q}_p ) \otimes_{\mathbb{Z}}
\mathbb{Q}_p/\mathbb{Z}_p=
\prod_{v \mid p}E(L_v) \otimes_{\mathbb{Z}}
\mathbb{Q}_p/\mathbb{Z}_p,
$$
and we have  
$\corank_{\Zp} \left(
\prod_{v \mid p}E(L_v) \otimes_{\mathbb{Z}}
\mathbb{Q}_p/\mathbb{Z}_p \right)=\sum_{v \mid p} [L_v: \Qp]=[L:\Q]$.)  
By the control theorem of 
fine Selmer groups (\autoref{corpropcontrthm} and 
\autoref{remcontrolthmR}), 
we deduce that 
$$
\lambda(X_\infty) \ge \rank_{\mathbb{Z}} E(K_m) - 
\varphi(p^m)
$$ 
for any $m \in \mathbb{Z}_{\ge 0}$, where 
$\varphi$ denotes Euler's totient function.
Thus, \autoref{cor:maintheoremcontrol} implies the following. 

\begin{cor}\label{cor:mordell-weilranklowerbound}
Let $E$ be an elliptic curve over $\Q$ which has good reduction at an odd prime $p$. 
Suppose that $E$ satisfies the conditions 
{\rm (C1)}, {\rm (C2)} and {\rm (C3)}.
Then, for any fixed $m \in \mathbb{Z}_{\ge 0}$, we have 
\[
h_n  \ge h_{n}^E \succ 2(r_m-\varphi(p^m))n
\] 
as $n \to \infty$, 
where we put $r_m:=\rank_{\mathbb{Z}} E(K_m)$. 
\end{cor}

\begin{rem}\label{rem_comparison}
The assertion of 
\autoref{cor:mordell-weilranklowerbound} for $m=0$ 
implies 
the ``asymptotic parts'' 
of the results   
by \cite{SY1}, \cite{SY2} and \cite{Hi},   
and  that for general $m \ge 0$ implies 
\cite{Oh} for the $p$-adic representation 
$T_p(E)=\varprojlim_n E[p^n]$ of $G_{K_m}$.
(Here,  the ``asymptotic parts'' means 
the assertions without description of constant error factors.) 
Our results, in particular \autoref{thmmain} and 
\autoref{cor:maintheoremcontrol}, 
can be regarded as a refinement of them in the following senses.
\begin{itemize}
    \item \autoref{cor:maintheoremcontrol} determines 
    the quotient $A^E_n$ of the ideal class group 
    $\mathrm{Cl}(\cO_{K^E_n})$, whose growth is described by the fine Selmer groups.
    \item \autoref{thmmain} describes not only the asymptotic behavior of the order of $A^E_n$ but also asymptotic behavior of the $R_n$-module  (and in particular, $\Zp$-module) structure.
\end{itemize}
\end{rem}

\begin{ex}\label{ex:5077p=7}
Let $E$ be the elliptic curve over $\Q$ of 
the LMFDB label 5077.a1 (the Cremona label 5077a1), which 
is defined by the equation 
\[
y^2+y=x^3-7x+6,  
\]
and set $p:=7$.
It is known the following (\cite{LMFDB}): 
\begin{enumerate}[label=(\roman*)]
    \item The elliptic curve $E$ does not have CM, and $(E,p)$ satisfies 
    $\mathrm{(C1)}_{\mathrm{str}}$. 
    \item The conductor of $E$ is 5077, which is a prime number, and $E$ has non-split multiplicative reduction at $5077$.
    \item The rank of $E(\Q)$ is $3$.
    \item Let $\widetilde{X}:=\Sel (\Q_\infty, E[7^\infty])^\vee$ 
    be the Iwasawa module of the Pontrjagin dual of the classical Selmer group of $E$ 
    over the cyclotomic $\Z_7$-extension field $\Q_\infty$ of $\Q$.
    We have $\mu (\widetilde{X})=0$, and $\lambda (\widetilde{X})=3$.
\end{enumerate}
The properties (iii) and (iv) imply that we have 
$\mathrm{char}_{\Lambda}(\widetilde{X})=(\gamma-1)^3\Lambda$. 
We further obtain 
$\mathrm{char}_{\Lambda}(X_{\infty, \boldsymbol{1}})=
(\gamma-1)^2\Lambda$ (see, for instance, \cite[VI.10]{Wu}). 
This implies that $\mu(X_{\infty, \boldsymbol{1}})=0$ and 
$\lambda (X_{\infty, \boldsymbol{1}})=2$. 
Moreover, we can show that the pair $(E,p)$ satisfies the 
condition (C2) (see \autoref{ex:C2suff} in \autoref{ss:C2ex}).
Therefore,  we obtain 
\[
h^E_{n,\boldsymbol{1}} \sim 2 n.
\]
\end{ex}

\subsection*{Notation}
Let $L/F$ be a Galois extension with $G = \Gal(L/F)$, 
and $M$ a topological abelian group equipped with a $\Z$-linear action of $G$. 
For each $i\in\Z_{\ge 0}$, we denote by 
$H^i(L/F, M) = H^i_{\mathrm{cont}}(G,M)$ 
the $i$-th continuous Galois cohomology group. 
If $L$ is a separable closure of $F$, then we write $H^i(F,M) = H^i(L/F,M)$. 
When $F$ is a non-archmedean local field, 
we denote by $F^{\ur}$ the 
maximal unramified extension of $F$. 
We define 
$H^1_{\ur}(F,M) = \Ker(H^1(F,M) \longrightarrow H^1(F^{\ur}, M))$ 
(cf.~\cite[Definition 1.3.1]{Ru}). 

For a $\Z_p$-module $A$, let $A_{\mathrm{div}}$ denote its maximal divisible subgroup. 
For an abelian group $M$ and an endomorphism $f$ of $M$, 
we put $M[f]:=\Ker(f)$. 
In particular, if $M$ is a module over a ring $R$, 
then, for each $a \in R$, we set
$M[a]:=\set{ x \in M | ax=0 }$. 
For an elliptic curve $E$ over a field $K$ and a field extension $L/K$, 
we will denote by $E_L:=E\otimes_K L$ the base change to $L$.

\section{The higher Fitting ideals}

\begin{defn}[{cf.~\cite[Section 20.2]{Ei}}]\label{defnFitt}
Let $R$ be a commutative ring, and $M$ 
a finitely presented $R$-module 
given by a presentation
\begin{equation}
\label{eqpresFitt}
R^m \xrightarrow{\ A \ } R^n 
\longrightarrow M \longrightarrow 0
\end{equation}
with $m \ge n$.
We define {\em the $i$-th Fitting ideal}
$\Fitt_{R,i}(M)$ of the $R$-module $M$ to 
be the ideal of $R$ generated by
$(n-i)\times (n-i)$ minors (that is, the determinants of the submatrices) of the matrix $A$.
When $i \ge n$, we define
$\Fitt_{R,i}(M):=R$. 
\end{defn}

\begin{rem}
The ideal $\Fitt_{R,i}(M)$ in \autoref{defnFitt}
does not depend on the choice of the presentation 
\eqref{eqpresFitt} (\cite[Corollary-Definition 20.4]{Ei}).
\end{rem}

\begin{rem}
The higher Fitting ideals are compatible with base change 
in the following sense:  
Let $R$ be a commutative ring, and $M$ 
a finitely presented $R$-module. 
Then, for any $R$-algebra $S$ and any $i \in \Z_{\ge 0}$, 
we have 
$\Fitt_{S,i}(S \otimes_R M)=\Fitt_{R,i}(M)S$ 
(\cite[Corollary 20.5]{Ei}).
\end{rem}

\begin{rem}\label{remPhiisom}
Let $R$ be a PID, and 
suppose that  $M$ is a finitely generated $R$-module.
By the structure theorem of 
finitely generated modules over a PID, 
the $R$-module $M$ is 
isomorphic to an elementary $R$-module
$R^{\oplus r} \oplus \bigoplus_{j=1}^s R/d_j R$ 
with a sequence  $\set{ d_j }_{j} 
\subseteq R \smallsetminus R^\times$ satisfying 
$d_j \mid d_{j-1}$ for every $j$.
We have 
\begin{equation}
\label{eq:FittM}	
\Fitt_{R,i}(M)= \begin{cases}
\set{0 } & \text{if $i < r$}, \\
\left( \prod_{j=i-r+1}^s d_j \right)R & \text{if $r \le i < s+r$}, \\
R & \text{if $i \ge  s+r$}.
\end{cases}
\end{equation}
In particular, the higher Fitting ideals 
$\set{\Fitt_{R,i}(M) }_i$ determine 
the  isomorphism class of the $R$-module $M$.  
\end{rem}

\begin{rem}\label{remsubquotFitt}
Let $R$ be a commutative ring, and $M$ 
an $R$-module with the presentation 
(\ref{eqpresFitt}). Let $N$ be an $R$-submodule 
of $M$.
\begin{enumerate}
    \item For any $i \in \mathbb{Z}_{\ge 0}$, 
    we have $\Fitt_{R,i}(M) \subseteq \Fitt_{R,i}(M/N)$. 
    Indeed, we have a presentation of $M/N$ of the form 
    \(
    R^{m+k} \xrightarrow{\ (A \, \vert \, B) \ } R^n 
\longrightarrow M/N \longrightarrow 0 
    \),
    and every $(n-i) \times (n-i)$ minor of $A$ becomes an $(n-i) \times (n-i)$ minor of the augmented matrix 
    $(A\, \vert \, B)$.
\item Suppose that $R=\Zp$, and $M$ is a torsion $\Zp$-module. 
For any finitely generated torsion $\Zp$-module $L$, 
we denote by $L^\vee = \Hom_{\Zp}(L,\Qp/\Zp)$ the Pontrjagin dual of $L$. 
The dual $N^\vee$ is a quotient of $M^\vee$, 
and there are non-canonical isomorphisms 
$M \simeq M^\vee$ and $N \simeq N^\vee$.
By (1), we have 
\[
\Fitt_{R,i}(M)=\Fitt_{R,i}(M^\vee) \subseteq \Fitt_{R,i}(N^\vee)=\Fitt_{R,i}(N)
\] for any $i \in \Z_{\ge 0}$.
\end{enumerate}
\end{rem}

As in \autoref{secintro}, we introduce the following notation: 

\begin{defn}
Let $M$ be a finitely generated torsion  $\Z_p$-module. 
For each $i \in \Z_{\ge 0}$, we define 
\[
\Phi_i(M)=\ord_p (\Fitt_{\Z_p,i}(M)) 
:= \min \set{ m \in  \Z_{\ge 0} | p^m \in \Fitt_{\Z_p,i}(M) }.
\]
\end{defn}

If $M$ is a torsion $\Z_p$-module isomorphic to 
$\bigoplus_{j=1}^s \Z_p/p^{e_j}\Z_p$ 
with a decreasing sequence  $\set{ e_j }_{j} \subseteq \Z_{> 0}$, 
then 
\begin{equation}\label{eqPhiisom}
\Phi_i(M)= \begin{cases}
\displaystyle\sum_{j=i+1}^s e_j & \text{if $0 \le i < s$}, \\
0 & \text{if $i \ge   s$},
\end{cases}
\end{equation}
immediately follows from \eqref{eq:FittM}. 
In particular, we have 
$\Phi_0(M)=\ord_p (\# M)$. 
%
As noted in \autoref{secintro}, 
the isomorphism class of a finitely generated 
torsion $\Z_p$-module $M$ is determined by 
$\set{\Phi_{i}(M) }_i$ by \eqref{eqPhiisom}.

\begin{lem}\label{lemFittsubmod}
Let $M$ be a finitely generated torsion 
$\mathbb{Z}_p$-module. Then, for any 
$i \in \mathbb{Z}_{\ge 0}$, we have 
\[
\Phi_i(M)=\min_{(a_1, \dots , a_i) \in M^i}
\ord_p\left (\# \left(M/ \sum_{j=1}^i
\mathbb{Z}_p a_j  \right) \right).
\]
\end{lem}

\begin{proof}
By the structure theorem, we have 
$M=\bigoplus_{j=1}^s (\mathbb{Z}/p^{e_j}\mathbb{Z})m_j$, 
where the sequence $\set{e_j} \subseteq \mathbb{Z}_{>0}$ is decreasing. 
For any $j \in \mathbb{Z}$ with $1 \le j \le s$, 
the annihilator of $m_j \in M$ is $p^{e_j}\mathbb{Z}_p$.
Fix any $i \in \mathbb{Z}_{\ge 0}$. 
If $i=0$ or $i \ge s$, 
then the assertion of 
\autoref{lemFittsubmod} is clear.
Now, we assume $1 \le i \le s-1$.
Put $N_0:=\sum_{j=1}^i \mathbb{Z}_p m_j$.
We have 
\[
\ord_p (\#(M/N_0))=\ord_p \left(
\# \left(\bigoplus_{j=i+1}^s (\mathbb{Z}/p^{e_j}\mathbb{Z})m_j \right)
\right)
=\sum_{j=i+1}^s e_j
\stackrel{\eqref{eqPhiisom}}{=}\Phi_i(M).
\]
Take any $a_1, \ldots , a_i \in M$, 
and put $N:=\sum_{j=1}^i\Zp a_j$. 
In order to prove \autoref{lemFittsubmod}, 
it suffices to show  
$\Phi_i(M) \le \ord_p(\# (M/N))$. 
Let $\xymatrix{\pi_N \colon \Zp^i 
\ar@{->>}[r] & N}$ 
be the surjection given by the generators 
$a_1, \dots, a_i \in N$, 
and take a presentation 
\[
\xymatrix{
\Zp^k \ar[r]^{A} & \Zp^i 
\ar[r]^{\pi_N} & N \ar[r] & 0}
\]
for some $k\ge 1$.
Since $M/N$ is a torsion $\Zp$-module, 
there is a square presentation 
\[
\xymatrix{
0 \ar[r] & \Zp^t \ar[r]^{B} & \Zp^t 
\ar[r] & M/N \ar[r] & 0}
\]
by the structure theorem. 
This gives a presentation
\[
\xymatrix{
 \Zp^{k+t } \ar[r]^{C} & \Zp^{i+t} 
\ar[r] & M \ar[r] & 0}
\]
with 
$C=\begin{pmatrix}
A & * \\
O & B
\end{pmatrix}$.
We obtain $\#(M/N)=\det B \in \Fitt_{\mathbb{Z}_p,i}(M/N)$. 
This implies  
$\#(M/N) \ge \Phi_i(M)$.
\end{proof}

Let $\set{ a_n}_n$ and $\set{ b_n }_n$ be sequences of real numbers.
We write $a_n \succ b_n$ if it holds that
$\liminf_{n \to \infty} (a_n-b_n) > -\infty$, 
namely, if the sequence $\set{a_n - b_n}_{n}$ is bounded below.
If $a_n \succ b_n$ and $b_n \succ a_n$, then 
we write $a_n \sim b_n$.  

\begin{lem}\label{lemasymPhi}
Let $\set{ M_n }_{n \ge 0}$ be a sequence of 
finitely generated torsion $\Z_p$-modules, and 
suppose that for each $n \in \Z_{\ge 0}$, 
a $\Zp$-submodule $N_n$  of $M_n$ is given. 
Then, the following hold.
\begin{enumerate}
\item If $\set{ (M_n :N_n) }_{n \ge 0}$ is bounded, then 
we have 
$\Phi_i(M_n) \sim \Phi_i(N_n)$ for any $i \in \Z_{\ge 0}$.
\item If $\set{\# N_n }_{n \ge 0}$ is bounded, then 
we have 
$\Phi_i(M_n) \sim \Phi_i(M_n/N_n)$ for any $i \in \Z_{\ge 0}$.
\end{enumerate}
\end{lem}

\begin{proof}
Let us show the assertion (1). 
Suppose that there exists 
some $B \in \mathbb{Z}_{>0}$ such that 
$(M_n :N_n) \le p^B$ for any $n \in \Z_{\ge 0}$.
Since $N_n$ is a submodule of $M_n$, 
by \autoref{remsubquotFitt} (2), 
we have $\Phi_i(M_n)\ge \Phi_i(N_n)$. 
In order to prove the assertion (1),
it suffices to show that 
$\Phi_i(M_n) \le \Phi_i(N_n)+B$.  
By \autoref{lemFittsubmod}, there exist 
$a_1, \dots, a_i \in N_n$ such that 
\[
\ord_p\left(\#\left(N_n/\sum_{j=1}^i \Zp a_j \right) \right)=\Phi_i(N_n).
\]
Since $(M_n:N_n)\le p^B$,  \autoref{lemFittsubmod} 
implies that 
\[
\Phi_i(M_n) \le 
\ord_p\left(\#\left(M_n/\sum_{j=1}^i \Zp a_j \right) \right) \le \Phi_i(N_n)+B.
\]
Accordingly, 
we obtain $\Phi_i(M_n) \sim \Phi_i(N_n)$, and 
the assertion (1) is verified.
By taking the Pontrjagin dual, 
the assertion (2) immediately follows from (1).
\end{proof}

\section{The conditions (C1) and (C2)}
\label{secC}

Until the end of this note, we use the following notation: 
Fix an \emph{odd} prime number $p$. 
Let $E$ be an elliptic curve over $\Q$. 
We denote by $D_E$ the discriminant of 
the minimal Weierstrass model for $E$ over $\Z$. 
We define the $p$-adic Tate module $T_p(E)$ by 
$T_p(E):= \varprojlim_n E[p^n]$, 
and put $V_p(E):=\Q_p \otimes_{\Z_p} T_p(E)$. 
As in \autoref{secintro}, 
for each $n\in \Z_{\ge 0}$, we define $K_n^E = \Q(E[p^n])$, and  
$K_n = \Q(\mu_{p^n})$. 
Put also $K_{\infty}^E = \bigcup_{n\ge 0}K_n^E$ and  
$K_{\infty} = \bigcup_{n\ge 0}K_n$. 

In this section, 
we review some results on the conditions (C1) and (C2) referred in \autoref{thmmain} 
under the additional assumption that 
$E$ has \emph{good reduction} at $p$. 
First, we recall the conditions: 
\begin{itemize}
\item[{\rm (C1)}] The restriction  
\[
\rho^E_1  
\colon G_{K_{\infty}} \longrightarrow \Aut_{\Fp}(E[p])
\simeq GL_2(\F_p)
\]
to $G_{K_{\infty}}$ of the mod $p$ Galois representation $\rho_1^E:G_{\Q}\longrightarrow \Aut_{\Fp}(E[p])$
is absolutely irreducible over $\F_p$.
\item[{\rm (C2)}] 
For any $n \in \Z_{\ge 1}$ and 
any place $v$ of $K_n$ 
with the base change $E_{K_{n,v}}$ of $E$
has potentially multiplicative reduction, 
we have $E(K_{n,v})[p]=0$.
\end{itemize}

\subsection{Remarks on (C1) and (C2)}

In this paragraph, 
we shall show some properties relates to (C1) and (C2)
mentioned in \autoref{secintro}. 
First, let us verify the following 
property, which is noted in \autoref{rem:C1andC1str}.

\begin{prop}
\label{lemC1}	
The condition {\rm (C1)} is satisfied if the Galois representation 
\[
\rho^E \colon G_\Q 
\longrightarrow \Aut_{\Z_p}(T_p (E)) \simeq 
GL_2(\Z_p)
\] 
is surjective. 
\end{prop}
\begin{proof}
	It is enough to show that 
	the image of $\rho_1^E\colon G_{K_{\infty}}\longrightarrow GL_2(\Fp)$ generates $\End_{\Fp}(E[p])\simeq M_2(\F_p)$ over $\Fp$. 
	By using the Weil pairing, $G_{\Q}$ acts on $\bigwedge^2_{\Z_p} T_p(E) \simeq T_p(\boldsymbol{\mu})$ 
	via the cyclotomic character $\chi$, where $T_p(\boldsymbol{\mu}) := \varprojlim_n \mu_{p^n}$ (cf.~\cite[Chapter V, Section 2]{Si1}). 
	We obtain the following commutative diagram with exact rows:
	\[
	\xymatrix{
	0 \ar[r] & G_{K_{\infty}} \ar[r]\ar[d]& G_\Q\ar[d]^{\rho^E} \ar[r] & \Gal(K_{\infty}/\Q) \ar[r]\ar[d]^{\chi}_{\simeq} & 0\, \\
	0\ar[r] & SL_2(\Zp) \ar[r] & GL_2(\Zp)\ar[r]^-{\det} & \Zp^{\times} \ar[r] & 0.
	}
	\]
	The assumption implies that 
	the image of the restriction $\rho^E|_{G_{K_{\infty}}}$ 
	coincides with $SL_2(\Zp)$. 
	By taking the mod $p$ reduction, 
	$SL_2(\Fp) = \rho_1^E(G_{K_{\infty}})$ 
	and this generates $M_2(\Fp)$ over $\Fp$. 
\end{proof}

Next, let us see the following property 
referred in \autoref{rem:C2twist}.

\begin{prop}
\label{lemTwist}
There exists a quadratic twist $E'/\Q$  of $E$ 
which satisfies the condition $\mathrm{(C2)}$. 
\end{prop}
\begin{proof}
For each prime number $\ell$, 
put $L_\ell:=\Q_{\ell}(\mu_{p^{\infty}})$. 
Suppose that $E$ is defined by the Weierstrass equation 
$y^2=x^3+ax+b$ with $a,b \in \Q$, and 
let $S(E)$ be the  set of all the 
prime numbers at which $E$ has potentially multiplicative reduction. 
As $E$ has good reduction at $p$, we have $p\not\in S(E)$.
For each $\ell \in S(E)$, we fix an embedding 
$\iota_\ell \colon \overline{\Q} \hookrightarrow 
\overline{\mathbb{Q}}_\ell$, and regard $\mu_{p^\infty}$ as 
a subgroup of $\overline{\Q}_\ell^\times$. 
Note that under these notations, 
the elliptic curve $E$ satisfies the condition (C2)  
if and only if
$E(L_{\ell})[p] = 0$ for any $\ell \in S(E)$. 
In order to show the assertion of \autoref{lemTwist}, 
we may suppose that 
$E$ does not satisfy the condition (C2). 
In particular, the set $S(E)$ is not empty. 
We define 
\begin{align*}
S_0(E):=& \set{ \ell \in S(E) | 
 E(L_\ell)[p] \ne 0,\  
 2 \nmid [\mathbb{Q}_\ell(\mu_p):\mathbb{Q}_\ell]}, \\
S_1(E):=& \set{ \ell \in S(E) | E(L_\ell)[p] \ne 0,\  
 2 \mid [\mathbb{Q}_\ell(\mu_p):\mathbb{Q}_\ell]}
\end{align*}
and put $N_1^*:=\prod_{\ell' \in S_1(E)}(\ell')^*$, 
where for each odd prime number $\ell$, we write 
$\ell^*:=(-1)^{\frac{\ell-1}{2}} \ell$, 
and put $2^*:=2$. 
For each odd $\ell \in S(E) \smallsetminus S_1(E)$,  we put 
\[
\varepsilon_\ell:=\prod_{\ell' \in S_1(E)} 
\Leg{(\ell')^*}{\ell},
\]
where $\Leg{\bullet}{\ell}$ denotes 
the Legendre symbol modulo $\ell$. 
By Dirichlet's theorem on arithmetic progressions, there exists 
an odd prime number $q$ prime to $p$  such that  
$q \equiv 1\ \mathrm{mod}\, 4$, and 
\begin{equation}\label{eqLegendreq}
\Leg{q}{\ell} =\begin{cases}
-\varepsilon_\ell & \text{if $\ell \in S_0(E)$}, \\
\varepsilon_\ell & \text{if $\ell \notin S_0(E)$},
\end{cases}
\end{equation}
for any odd $\ell \in S(E) \smallsetminus S_1(E)$. 
Moreover, we may suppose that 
\[
qN_1^* \equiv \begin{cases}
1 \bmod 8 & \text{if  $2 \in S(E) \smallsetminus (S_0 (E) \cup S_1(E))$}, \\
5 \bmod 8 & 
\text{if  $2 \in S_0 (E)$}.
\end{cases} 
\]
Take such a prime number $q$, and let 
$E'$ be a quadratic twist of $E$  defined by  
the Weierstrass equation 
$q N_1^* y^2=x^3+ax+b$. 
We have an equality $S(E')=S(E)$ because 
$E$ and $E'$ are isomorphic over the field 
$\Q(\sqrt{qN_1^*})$.

Let us show that $E'$ satisfies (C2).  
In the following, We prove $E'(L_\ell)[p] = 0$ 
for any $\ell \in S(E')$. 
Take any $\ell \in S(E')$. 

\smallskip
\noindent 
\textbf{(The case $\ell \not\in S_0(E)\cup S_1(E)$)} 
First, we suppose that $\ell$ does not belong to $S_0(E) \cup S_1(E)$. 

\setcounter{claim}{0}
\begin{claim}
The prime $\ell$ splits in $\mathbb{Q}(\sqrt{qN_1^*})/\mathbb{Q}$. 
\end{claim} 
\begin{proof}
When $\ell$ is odd, the prime  
$\ell$ is split in 
$\mathbb{Q}(\sqrt{qN_1^*})/\mathbb{Q}$ 
if and only if 
 $\Leg{qN_1^{\ast}}{\ell} = 1$ (\cite[Chapter 1, Proposition 8.5]{Ne}). 
By \eqref{eqLegendreq}, we have 
\[	\Leg{qN_1^*}{\ell} = \Leg{q}{\ell}\Leg{N_1^*}{\ell} =\varepsilon_{\ell}\prod_{\ell'\in S_1(E)}\Leg{(\ell')^{\ast}}{\ell} = 1.
\]
Next,  consider the case $\ell = 2$. 
As $\ell \not\in S_1(E)$ in this case, 
we know that 
$N_1^*$ is odd. 
Since 
$qN_1^* \equiv 1 \bmod 8$, the prime
$\ell = 2$ splits in $\mathbb{Q}(\sqrt{qN_1^*})/\mathbb{Q}$. 
\end{proof}

\noindent 
From the above claim, 
the completion of $\Q(\sqrt{qN_1^*})$ at a place $v$ above $\ell$ 
is $\Q_{\ell}$ and 
the base change to the local field $\Q_{\ell}$,  
we obtain $E'_{\Q_{\ell}} \simeq E_{\Q_{\ell}}$. 
As a result, we have 
\[
E'(L_\ell)[p] \simeq E'_{\Q_{\ell}}(L_\ell)[p] \simeq 
E_{\Q_{\ell}}(L_\ell)[p] \simeq 
E(L_\ell)[p] =0.
\]

\smallskip
\noindent
\textbf{(The case $\ell \in S_0(E)\cup S_1(E)$)} 
Next, we suppose that $\ell$ belongs to $S_0(E) \cup S_1(E)$. 
It holds that $E(L_{\ell})[p]\neq 0$ and fix a non-zero $P \in E(L_{\ell})[p]$.

\begin{claim}
\label{clm:3.1}
	The action of $G_{L_{\ell}}$ on $E[p]$ is unipotent.
\end{claim}
\begin{proof}
Take a basis $\set{P,Q}$ 
of $E[p]$ as an $\Fp$-vector space with 
$Q\in E[p]\smallsetminus \Fp P$. 
Recall that 
the Weil pairing $e: E[p]\times E[p] \longrightarrow \mu_p$ is alternating and $G_{L_\ell}$-equivariant 
(\cite[Chapter III, Section 8]{Si1}). 
As $\mu_p\subseteq L_{\ell}$, 
we have $\sigma \left( e(P,Q) \right) = e(P,Q)$ for any $\sigma \in G_{L_{\ell}}$. 
On the other hand, $\sigma \left( e(P,Q) \right) 
= e(\sigma P, \sigma Q) = e(P,\sigma Q)$ 
implies $e(P,\sigma Q-Q) = 1$. 
Here, the element of the form $\sigma Q-Q$ is in the kernel of $E[p]\longrightarrow \mu_p; T\longmapsto e(P,T)$ 
which is generated by $P$ so that 
$\sigma Q-Q = aP$ for some $a \in \Fp$. 
According to the fixed basis above, 
the action of $\sigma$ is written as $\begin{pmatrix}
	1 & a\\ 0 & 1
\end{pmatrix}$ 
which is unipotent.
\end{proof}

\begin{claim}
The extension 
$L_{\ell}(\sqrt{qN_1^*})/L_\ell$ is 
quadratic. 
\end{claim}

\begin{proof}
Let us show the claim by dividing 
into three cases. 
\begin{enumerate}[label=(\roman*)]
\item Suppose that  $\ell \in S_0(E)$, and 
$\ell$ is odd. 
The equalities
\[
\Leg{qN_1^{\ast}}{\ell} \stackrel{\eqref{eqLegendreq}}{=} 
-\varepsilon_{\ell}\prod_{\ell'\in S_1(E)}\Leg{(\ell')^{\ast}}{\ell} = -1
\]
imply that the prime $\ell$ is inert in the quadratic extension $\Q(\sqrt{qN_1^{\ast}})/\Q$. 
For the prime $2$ does not divid $[\Q_\ell(\mu_p):\Q_{\ell}]$, 
we have $\Q_\ell(\sqrt{qN_1^*}) \not\subseteq L_\ell$. 
Hence, 
the extension $L_{\ell}(\sqrt{qN_1^*})/ L_\ell$ is non-trivial. 
\item Suppose that 
$\ell=2 \in S_0(E)$. The extension 
$L_2 = \Q_2(\mu_{p^{\infty}})/\mathbb{Q}_2$ does not contain 
quadratic extension fields of $\mathbb{Q}_2$. 
Since we have 
$qN_1^* \equiv 5 \ \mathrm{mod}\, 8$, 
the prime $2$ is inert in the extension $\Q(\sqrt{qN_1^{\ast}})/\Q$.
Thus, the extension $L_{2}(\sqrt{qN_1^*})/ L_2$ is non-trivial. 

\item Suppose that $\ell \not\in S_0(E)$. Then 
$\ell \in S_1(E)$ and thus $\ell \mid N_1^{\ast}$. 
This implies that 
the prime $\ell$ is ramified  
in the extension $\Q(\sqrt{qN_1^{\ast}})/\Q$. 
We also have 
$L_\ell(\sqrt{qN_1^*}) \ne L_\ell$ because  $L_{\ell}/\Q_\ell$ is 
unramified. 
\end{enumerate}
In each cases, the extension $L_{\ell}(\sqrt{qN_1^*})/L_\ell$ is 
quadratic. 
\end{proof}

From \autoref{clm:3.1} above, 
there exists a basis $\set{P,Q}$ of $E[p]$ as $\Fp$-vector space such that 
$G_{L_{\ell}}$ acts trivially on 
$\mathbb{F}_p P$, and also 
$E[p]/\F_p P$ 
which is generated by  
the residue class represented by $Q \in E[p] \smallsetminus \Fp P$.
We have an isomorphism $f \colon E[p]\otimes \mathbb{F}_p(\psi) 
\xrightarrow{\ \simeq \ } E'[p]$
of $\mathbb{F}_p[G_{L_\ell}]$-modules, 
where $\psi$ denotes the quadratic character 
attached to $L_\ell(\sqrt{qN_1^*})/L_\ell$.
Take a lift $\sigma\in G_{L_{\ell}}$ 
of the generator of 
$\Gal(L_\ell(\sqrt{qN_1^*})/L_{\ell})$. 
This satisfies $\sigma P=P$ and $\sigma Q-Q \in \Fp P$. 
Thus, 
the element $\sigma$ acts by $\psi (\sigma)=-1$
on both $\Fp f(P \otimes 1) \subseteq E'[p]$ and 
$E'[p]/\Fp f(P \otimes 1)$. 
Therefore, for any $\ell \in S(E')=S(E)$, 
we have $E'(L_{\ell})[p] = 0$.
\end{proof}

\subsection{Equivalent conditions of (C2)}

For later use in the proof of our main results, 
let us study some equivalent conditions of (C2).
We need the following lemma.

\begin{lem}
\label{lem:K1E}
Suppose that $E$ has potentially multiplicative reduction at $\ell$ ($\neq p$). 
	Then, the elliptic curve $E_{K_1^E}$ has split multiplicative reduction 
at every place of $K_1^E = \Q(E[p])$ above $\ell$.  
\end{lem}
\begin{proof}
	We may assume that the $j$-invariant $j(E)$ is not equal to 
$0$ or $1728$ because $E$ has potentially good reduction 
at all primes in such cases (\cite[Chapter VII, Proposition~5.5]{Si1}). 
By \cite[Chapter V, Lemma~5.2]{Si2},
there exist elements $q,\gamma \in \mathbb{Q}_\ell ^\times$ with  
$\ord_\ell(q)>0$ 
such that $E_{\mathbb{Q}_\ell(\sqrt{\gamma})}$ has
split multiplicative reduction, and 
we have a $G_{\mathbb{Q}_\ell}$-equivariant isomorphism  
\[
f \colon
E[p^\infty] 
\xrightarrow{\ \simeq\ } 
(\overline{\mathbb{Q}}_\ell^\times/q^{\mathbb{Z}})[p^\infty]
\otimes_{\mathbb{Z}_p} 
\mathbb{Z}_p(\chi), 
\]
where 
$\chi \colon G_{\mathbb{Q}_\ell} \longrightarrow \mathbb{Z}_p^\times$ 
is the trivial character or 
the quadratic character attached to the extension 
$\mathbb{Q}_\ell(\sqrt{\gamma})/\mathbb{Q}_\ell$. 
In order to prove the assertion, 
it is  sufficient to show that $\sqrt{\gamma} \in 
\mathbb{Q}_\ell(E(\overline{\mathbb{Q}}_\ell)[p])$. 
Since the Weil paring 
\[
E(\overline{\mathbb{Q}}_\ell)[p] 
\times E(\overline{\mathbb{Q}}_\ell)[p] \longrightarrow 
\boldsymbol{\mu}_p(\overline{\mathbb{Q}}_\ell)=\mu_p
\]
preserves the action of 
$G_{\mathbb{Q}_\ell}$, we have $\mu_p \subseteq 
\mathbb{Q}_\ell(E(\overline{\mathbb{Q}}_\ell)[p])$. 
If $\sqrt{\gamma} \in \mu_p$, then  
$\sqrt{\gamma} \in \mathbb{Q}_\ell(E(\overline{\mathbb{Q}}_\ell)[p])$.
Suppose that $\sqrt{\gamma} \notin \mu_p$.
The fields $F_1:=\mathbb{Q}_\ell(\sqrt{\gamma})$ and 
$\mathbb{Q}_\ell(\mu_p)$ are linearly disjoint 
over $\mathbb{Q}_\ell$.
Moreover, as $p$ is odd, 
the fields $F_1$ and 
$F_2:=\mathbb{Q}_\ell(\mu_p,\sqrt[p]{q})$ are linearly disjoint 
over $\mathbb{Q}_\ell$.
Put $\widetilde{F}:=
\mathbb{Q}_\ell(\mu_p, \sqrt[p]{q}, 
\sqrt{\gamma})=F_2F_1$. 
By the isomorphism $f$, we have 
$\mathbb{Q}_\ell(E(\overline{\mathbb{Q}}_\ell)[p], 
\sqrt{\gamma})
= \widetilde{F}$. 
It holds that 
$F_2 \subseteq \mathbb{Q}_\ell(E(\overline{\mathbb{Q}}_\ell)[p])$ 
because $\widetilde{F}/\mathbb{Q}_\ell$ is an abelian extension, 
and $p$ is odd.
Furthermore, by the isomorphism $f$, 
the group $\Gal(\widetilde{F}/F_2) \ (\simeq \Gal(F_1/\mathbb{Q}_\ell))$ 
acts faithfully on $E(\overline{\mathbb{Q}}_\ell)[p]$.
This implies that 
$\mathbb{Q}_\ell(E(\overline{\mathbb{Q}}_\ell)[p])
= \widetilde{F}$, and especially 
$\sqrt{\gamma} \in  \mathbb{Q}_\ell(E(\overline{\mathbb{Q}}_\ell)[p])$. 
Consequently, the elliptic curve $E_{K_1^E}$ has split multiplicative reduction at every place of $K_1^E$.
\end{proof}

The following \autoref{lemC2} 
gives some conditions equivalent to  
(C2).

\begin{lem}\label{lemC2}
Let $\ell$ be a prime number. 
Suppose that $E$ has potentially multiplicative reduction at $\ell$.
Then, the following are equivalent:

\begin{enumerate}[label=$\mathrm{(\alph*)}$]
\item For any $n \in \mathbb{Z}_{\ge 1}$ and 
any place $v$ of $K_n$ above $\ell$, we have 
$E(K_{n,v})[p]=0$.
\item For any $n \in \mathbb{Z}_{\ge 1}$ and 
any place $w$ of $K^E_n$ above $\ell$ 
where the base change $E_{K^E_{n,w}}$ of $E$ 
has split multiplicative reduction, 
we have 
\[
H^0\left( K_{n,v},
E(K_{n,w}^{E,\ur})[p^\infty]_{\div}
\right)=0.
\]
Here, we denote by $v$ the place of 
$K_n$ below $w$.
(Note that the absolute Galois group $G_{K_{n,v}}$ acts on 
$E(K_{n,w}^{E,\ur})[p^\infty]_{\div}$ because the extension 
$K_{n,w}^{E,\ur}/K_{n,v}$ is Galois.)
\item For any $n \in \mathbb{Z}_{\ge 1}$ and 
any place $w$ of $K^E_n$ above $\ell$ at 
where $E_{K^E_n}$ has split multiplicative reduction, 
we have 
\[
H^0\left( K_{n,v}, E[p^\infty]/
E(K_{n,w}^{E,\ur})[p^\infty]_{\div}
\right)=0,
\]
where $v$ denotes the place of 
$K_n$ below $w$.
\item For any $n \in \mathbb{Z}_{\ge 1}$,  
any place $v$ of $K_n$ above $\ell$ and  
any subquotient $\mathbb{Z}_p[G_{K_{n,v}}]$-module $M$ 
of $E[p^\infty]$, 
we have 
\[
H^0\left( K_{n,v}, M 
\right)=0.
\]
\end{enumerate}
\end{lem}

\begin{rem}
Recall that the condition (C2) holds if and only if 
for any prime number $\ell$ with $E$ has 
potentially multiplicative reduction, 
the condition (a) in \autoref{lemC2} holds. 
As we are assuming $E$ has good reduction at $p$, 
the prime number $\ell \neq p$. 
\end{rem}

\begin{proof}[Proof of \autoref{lemC2}]
\noindent 
\textbf{(a) $\Rightarrow$ (b):} 
Suppose that the base change $E_{K_{n,w}^E}$ 
has split multiplicative reduction for $n\ge 1$ and a place $w$ of $K_n^E$ above $\ell$, 
we have 
\[
H^0(K_{n,v},E(K_{n,w}^{E,\ur})[p^{\infty}]_{\div}) = E(K_{n,v})[p^{\infty}]_{\div}\subseteq E(K_{n,v})[p^{\infty}].
\]
The latter group is trivial because of $E(K_{n,v})[p] = 0$.

\smallskip
\noindent 
\textbf{(d) $\Rightarrow$ (a):} 
Take any $n\ge 1$, and any place $v$ of $K_n$ above $\ell$. 
As 
$E[p]$ is a submodule of  $E[p^{\infty}]$, 
the condition (d) implies $E(K_{n,v})[p] = H^0(K_{n,v},E[p]) = 0$. 

\smallskip 
\noindent 
\textbf{(b) $\Leftrightarrow$ (c):} 
Suppose that $w$ is a place of $K^E_n$ where $E_{K^E_n}$ 
has split multiplicative reduction, and let 
$v$ be a place of $K_n$ below $w$. 
The elliptic curve is isomorphic to 
a Tate curve $\mathbb{G}_m/q_w^\mathbb{Z}$ (\cite[Chapter V, Theorem 3.1]{Si2}).
Since $\ell \ne p$ 
and $K_{n,w}$ is an extension of $\Q_{\ell}$, 
the extension $K_{n,w}(\mu_{p^{\infty}})$ 
is unramified (\cite[{{\sc Chapitre}~IV, \S 4, {\sc Proposition}~16}]{Se1}) 
so that 
we have 
$E(K_{n,w}^{E,\ur})[p^\infty]_{\div} \simeq \mu_{p^\infty}$ and 
\[
E[p^\infty]/E(K_{n,w}^{E,\ur})[p^\infty]_{\div} \simeq 
\frac{\mu_{p^\infty} \times (q_w^{\mathbb{Z}} \otimes_{\mathbb{Z}}
\mathbb{Z}[1/p])}{
\mu_{p^\infty} \times q_w^{\mathbb{Z}}}.
\]
By the Weil pairing, we have 
a natural $G_{K_{n,v}}$-equivariant isomorphism
\[
\bigg(
E(K_{n,w}^{E,\ur})[p^\infty]_{\div}
\bigg)[p] \simeq
\Hom_{\mathbb{Z}_p} \left(
\frac{E[p^\infty]}{E(K_{n,w}^{E,\ur})[p^\infty]_{\div}}[p], 
\mu_p \right)
\]
for any $n \in \mathbb{Z}_{\ge 1}$.
As $G_{K_{n,v}}$ acts trivially on $\mu_p$, 
we deduce that (b) and (c) are equivalent.

\smallskip 
\noindent 
\textbf{(b) \& (c) $\Rightarrow$ (a):} 
Take any $n\ge 1$, and 
any place $v$ of $K_n$ above $\ell$. 
By \autoref{lem:K1E}, 
the base change $E_{K_{n,w}^E}$ of $E$ has split multiplicative reduction 
for some place $w$ of $K_n^E$ above $v$. 
The short exact sequence 
\[
0 \longrightarrow E(K_{n,w}^{E,\ur})[p^\infty]_{\div} \longrightarrow E(K_{n,w}^{E,\ur})[p^{\infty}]\longrightarrow 
\frac{E(K_{n,w}^{E,\ur})[p^{\infty}]}{E(K_{n,w}^{E,\ur})[p^{\infty}]_{\div}} \longrightarrow 0
\]
induces the exact sequence 
\begin{equation}
\label{ex:Epinfty}	
H^0\left(K_{n,v},E(K_{n,w}^{E,\ur})[p^\infty]_{\div}\right)\longrightarrow E(K_{n,v})[p^{\infty}]\longrightarrow H^0\left(K_{n,v}, \frac{E(K_{n,w}^{E,\ur})[p^{\infty}]}{E(K_{n,w}^{E,\ur})[p^{\infty}]_{\div}}\right)
\end{equation}
by the equality $H^0(K_{n,v},E(K_{n,w}^{E,\ur})[p^{\infty}]) = E(K_{n,v})[p^{\infty}]$. 
From the condition (b), we have 
$H^0\left(K_{n,v},E(K_{n,w}^{E,\ur})[p^\infty]_{\div}\right) = 0$. 
As the functor $H^0(K_{n,v},-)$ is left exact,  the condition (c) implies 
\[
H^0\left(K_{n,v}, \frac{E(K_{n,w}^{E,\ur})[p^{\infty}]}{E(K_{n,w}^{E,\ur})[p^{\infty}]_{\div}}\right)
\subseteq H^0\left(K_{n,v}, \frac{E[p^{\infty}]}{E(K_{n,w}^{E,\ur})[p^{\infty}]_{\div}}\right) = 0.
\]
From the exact sequence \eqref{ex:Epinfty}, 
we obtain $E(K_{n,v})[p] \subseteq E(K_{n,v})[p^{\infty}] =0$ and this implies the condition (a).

\smallskip 
\noindent 
\textbf{(b) \& (c) $\Rightarrow$ (d):} 
For any $n\ge 1$ and any place $v$ of $K_n$ above $\ell$, 
take any subquotient $\Zp[G_{K_{n,v}}]$-module $M$ of $E[p^{\infty}]$. 
From \autoref{lem:K1E}, the elliptic curve $E_{K_{n,w}^E}$ has split multiplicative reduction 
for some place $w$ of $K_n^E$ above $v$. 
Every Jordan--H\"older constituent 
(that is, composition factors of the Jordan--H\"older series)
of the $\mathbb{Z}_p[G_{K_{n,v}}]$-module
$E[p]$
(and hence every simple subquotient of $E[p^\infty]$) 
is isomorphic to  
$\bigg(E(K_{n,w}^{E,\ur})[p^\infty]_{\div}\bigg)[p] $ or 
$\bigg(
E[p^\infty]/E(K_{n,w}^{E,\ur})[p^\infty]_{\div}
\bigg)[p] $, which are one dimensional representation of 
$G_{K_{n,v}}$ over $\mathbb{F}_p$. 
As $M$ is a subquotient $\Z_p[G_{K_{n,v}}]$-module of $E[p^{\infty}]$, 
every simple subquotient of $M$ is isomorphic to 
a Jordan--H\"older constituents of $E[p]$. 
The conditions (b) and (c) imply (d).
This completes the proof of \autoref{lemC2}
\end{proof}

\subsection{Example of (C2)}\label{ss:C2ex}

It is obvious that if $E$ has potentially good reduction everywhere, 
then $(E,p)$ satisfies the condition (C2). 
Here, we introduce an example of $(E,p)$ satisfying (C2)
such that $E$ has multiplicative reduction at some primes. 
The following proposition is useful to find 
such a pair $(E,p)$. 

\begin{prop}\label{prop:C2suff}
Let $\ell$ be a prime number distinct from $p$. 
Suppose that $E$ has non-split multiplicative reduction at $\ell$. 
We also assume that $p \equiv 3 \ \mathrm{mod}\, 4$, and 
$-p$ is quadratic residue modulo $\ell$. 
Then, it holds that $E(K_{n,v})[p]=0$ for any $n \in \Z_{\ge 0}$, 
and any place $v$ of $K_n$ above $\ell$. 
\end{prop}

\begin{proof}
Fix any embedding $\overline{\Q} 
\hookrightarrow \overline{\Q}_\ell$, and 
regard $\mu_{p^\infty}$ 
as a subgroup of $\overline{\Q}_\ell^\times$.
Let $q,\gamma \in \Q_\ell$ 
be as in the proof of \autoref{lem:K1E}. 
We have $\sqrt{\gamma} \notin \Q_\ell$ 
because $E$ has non-split multiplicative 
reduction at $\ell$. 
Let $\chi \colon G_{\Q_\ell} \longrightarrow \Z_p^\times$ 
be the quadratic character 
attached to $\Q_\ell (\sqrt{\gamma})/\Q_\ell$.
We have a $G_{\Q_\ell}$-equivariant isomorphism 
\[
f\colon E[p^\infty] \longrightarrow 
(\overline{\Q}_\ell^\times/q^{\mathbb{Z}})[p^\infty] \otimes_{\Z_p} 
\Z_p(\chi). 
\]
In order to prove \autoref{prop:C2suff}, 
it suffices to show that 
\begin{equation}\label{eq:forC2ex}
H^0(\Q_\ell(\mu_{p^\infty}), 
(\overline{\Q}_\ell^\times/q^{\mathbb{Z}})[p^\infty] \otimes_{\Z_p} 
\Z_p(\chi))=0. 
\end{equation}
It holds that $\sqrt{-p} \in \Q_\ell$, 
because $p \equiv 3 \ \mathrm{mod}\, 4$, and 
$-p$ is quadratic residue modulo $\ell$. 
For this reason, the extension degree $[\Q_\ell (\mu_p):\Q_\ell]$ is odd. 
This implies that $\Q_\ell(\mu_{p^\infty})$ never contains 
any quadratic extension field of $\Q_\ell$ because $p$ is odd. 
We obtain 
\[
H^0(\Q_\ell(\mu_{p^\infty}), 
\mu_{p^\infty} \otimes_{\Z_p} 
\Z_p(\chi))\simeq 
H^0(\Q_\ell(\mu_{p^\infty}), 
(\Q_p/\Z_p)(\chi))
=0. 
\]
Suppose that \eqref{eq:forC2ex} does not hold.  
There exists an element  
\[
P \in H^0(\Q_\ell(\mu_{p^\infty}), 
(\overline{\Q}_\ell^\times/q^{\mathbb{Z}})[p^\infty] \otimes_{\Z_p} 
\Z_p(\chi))
\]
of order $p$. 
Let $\zeta$ be a primitive $p$-th root of unity, 
and $\sigma \in G_{\Q_\ell(\mu_{p^\infty})}$ an element 
satisfying $\chi(\sigma)=-1$.
Note that in $\mu_{p^\infty} \otimes_{\Z_p} \Z_p(\chi)$, 
we have 
$\sigma(\zeta\otimes 1)=\zeta\otimes (-1)$.
By taking the Weil pairing 
$e\colon E[p] \times E[p] \longrightarrow \mu_p$, we obtain 
\[
e(\zeta\otimes 1, P)
=\sigma(e(\zeta\otimes 1, P))
=e\left( 
\sigma(\zeta\otimes 1), 
\sigma P \right)=e(\zeta \otimes (-1), P )
=e(\zeta\otimes 1, P)^{-1}.
\]
As $p$ is odd, this contradicts to 
the fact that the Weil pairing $e$ is 
$G_{\Q_\ell}$-equivariant. 
Consequently, the assertion \eqref{eq:forC2ex} holds.
\end{proof}

\begin{ex}\label{ex:C2suff}
Let $(E,p)$ be as in \autoref{ex:5077p=7}. 
Then, the elliptic curve $E$ has good reduction outside the prime $5077$, 
and it has non-split multiplicative reduction at $5077$. 
Since $p=7 \equiv 3\ \mathrm{mod}\, 4$, and 
$-7$ is a quadratic residue modulo $5077$, 
\autoref{prop:C2suff} implies that 
$(E,p)$ satisfies the condition (C2).
\end{ex}

\section{Selmer Groups}
\label{secSelmer}

In this section, we shall recall the definition of 
the fine Selmer groups of an elliptic curve, 
and introduce some preliminary results related to Selmer groups. 
In \autoref{ssIwasawa}, 
we shall review preliminary results in the Iwasawa theoretical setting.
We keep the notation and the assumptions in \autoref{secC}.

\subsection{Definition of Selmer groups}\label{ssdefSel}
Let $K$ be a number field, that is, a finite extension field of $\Q$. 
First, let us recall Bloch--Kato's finite local conditions.

\begin{defn}[{\cite[Definition~1.3.4, Remark~1.3.6]{Ru}}]
Let $v$ be any place of $K$. 
We define  
$H^1_f(K_v, V_p(E))$ by 
\[
H^1_f(K_v, V_p(E)):=
\begin{cases}
H^1_{\ur}(K_v,V_p(E)) & \text{if $v \nmid p$}, \\
\Ker \left(
H^1(K_v, V_p(E)) \longrightarrow 
H^1(K_v, B_{\mathrm{cris}} \otimes_{\Q_p} V_p(E)) 
\right) & \text{if $v \mid p$}, \\
0 & \text{if $v\mid \infty$},
\end{cases}
\]
where $B_{\mathrm{cris}}$ is Fontaine's $p$-adic ring 
and $v\mid \infty$ we means that $v$ is an infinite place in $K$. 
We define $H^1_f(K_v, E[p^\infty]) \subseteq H^1(K_v,E[p^{\infty}])$  
and $H^1_f(K_v,T_p(E)) \subseteq H^1(K_v,T_p(E))$ 
to be  
the image and the inverse image, respectively, of $H^1_f(K_v, V_p(E))$ 
under the natural maps
\(
H^1(K_v,T_p(E)) \longrightarrow 
H^1(K_v, V_p(E)) \longrightarrow 
H^1(K_v, E[p^\infty]) 
\). 
For each $n \in \Z_{>0}$, 
we define $H^1_f(K_v, E[p^n])$ to be
the inverse image of $H^1_f(K_v, E[p^{\infty}])$ by the natural map
\begin{equation}
	\label{def:iota}
\iota_{n,v}\colon H^1(K_v, E[p^n]) \longrightarrow 
H^1(K_v, E[p^\infty]). 
\end{equation}
The subgroup $H^1_f(K_v,E[p^n])$ coincides with 
the image of $H^1_f(K_v,T_p(E))$ under the map $H^1(K,T_p(E))\longrightarrow H^1(K_v,E[p^n])$ 
induced by 
$T_p(E) \longrightarrow T_p(E)/p^nT_p(E) 
\simeq E[p^n]$ (\cite[Lemma 1.3.8]{Ru}).
\end{defn}

\begin{rem}
\label{remH^1_f}
Let $v$ be any finite place  of $K$ not above $p$. 
Suppose that $E_{K}$ has good reduction at $v$. 
The $p$-adic Tate module $T_p(E)$ is unramified at $v$ (from the ``easy'' direction of the N\'eron-Ogg-Shafarevich criterion \cite[Chapter~VII, Theorem~7.1]{Si1}) 
so that 
 $H_f^1(K_v, E[p^n])$ 
coincides with $H_{\ur}^1(K_v, E[p^n])$ (cf.\ \cite[Lemma~1.3.8]{Ru}),
for each $n \in \Z_{>0} \cup \set{ \infty }$. 
Furthermore, the inflation-restriction exact sequence (e.g., \cite[Proposition~B.2.5]{Ru}) gives a natural isomorphism 
\[
H^1(K_v^{\ur}/K_v, E[p^n]) \simeq 
H^1_f(K_v, E[p^n]).
\]
\end{rem}

\begin{defn}[the fine Selmer group]\label{defnSel}
For each $n \in \Z_{>0} \cup \set{ \infty }$, 
we define  {\em the fine Selmer group} 
$\Sel_p (K,E[p^n])$ to be the kernel of  
\[
H^1(K, E[p^n]) 
\longrightarrow 
\prod_{u \mid p}
H^1(K_u, E[p^n]) 
\times
\prod_{v \nmid p}
\frac{
H^1(K_v, E[p^n]) 
}{H^1_f(K_v, E[p^n]) 
},
\]
where $u$ runs through all the places of $K$ 
above $p$, and 
$v$ runs through all the places of $K$ 
not above $p$.
\end{defn}

\begin{rem}
	When $v$ is an infinite place of $K$, 
	the cohomology group 
	$H^1(K_v,E[p^n])$ is annihilated by at most $2$ 
	for each $n\in \Z_{\ge 1}\cup \set{\infty}$. 
	Since we are considering the odd prime $p$, 
	we have $H^1(K_v,E[p^n]) = 0$. 
	Because of this, 
	we may not care about infinite places in the following.	
\end{rem}

\begin{rem}\label{remSigma}
We denote by $\Sigma_K$ 
the set of places of $K$ above the prime divisors of $pD_E$ and the all infinite places  
and by $K_{\Sigma}$ the maximal algebraic extension field of $K$ 
unramified outside $\Sigma_K$. 
Then, for each $n \in \Z_{>0} \cup \set{ \infty }$, 
the kernel of the natural map
\[
H^1(K, E[p^n]) 
\longrightarrow 
\prod_{v \notin \Sigma_K}
\frac{
H^1(K_v, E[p^n]) 
}{H^1_f(K_v, E[p^n]) 
}
\] 
coincides with $H^1(K_{\Sigma}/ K, E[p^n])$
(\cite[Lemma 1.5.3]{Ru}).  
The fine Selmer group $\Sel_p (K,E[p^n])$ 
can be regarded as a subgroup of 
$H^1(K_{\Sigma}/ K, E[p^n])$.
\end{rem}

\begin{rem}\label{remSelcl}
Here, we give a remark on the relation between 
$\Sel_p (K,E[p^n])$ and the classical Selmer group.
Take any $n \in \Z_{>0}$. 
Recall that the classical Selmer group 
$\Sel (K,E[p^n])$ is defined by 
\[
\Sel (K,E[p^n]):=
\Ker \left(
H^1(K, E[p^n]) 
\longrightarrow \prod_{v }
\frac{H^1(K_v, E[p^n]) }
{H^1_{\mathrm{cl}}(K_v, E[p^n])}
\right),
\]
where $v$ runs through all the finite places of $K$, 
and $H^1_{\mathrm{cl}}(K_v, E[p^n])$ denotes
the image of the homomorphism  
\[
E(K_v)=
H^0(K_v, E(\overline{K}_v)) \longrightarrow H^1(K_v, E[p^n])
\]
induced by the short exact sequence 
\[
0 \longrightarrow  E[p^n] \xrightarrow{\ \subseteq \ } 
E(\overline{K}_v) \xrightarrow{\ \times p^n \ } 
E(\overline{K}_v) \longrightarrow 0.
\]
For any $n \in \Z_{>0}$, there exists a short exact sequence 
\[
0 \longrightarrow E(K) \otimes_{\Z} \Z/p^n\Z 
\longrightarrow \Sel (K,E[p^n]) \longrightarrow 
\Sha(E_K/K)[p^n]\longrightarrow 0,
\]
where $\Sha(E_K/K)$ denotes the Tate-Shafarevich group of $E_K/K$.
For each finite place $v$ of $K$, we have 
\[
H^1_{\mathrm{cl}}(K_v, E[p^n]) 
=H^1_{f}(K_v, E[p^n]). 
\]
It holds that 
\[
\Sel_p (K,E[p^n])=
\Ker \left(
\Sel (K,E[p^n]) \longrightarrow \prod_{v \mid p} 
H^1(K_v, E[p^n])
\right). 
\]
\end{rem}
%
%

\subsection{Preliminaries of Iwasawa theory}\label{ssIwasawa}

For each place $v$ of $K_1$, we denote by 
$D_v$ the decomposition subgroup of the Galois group
$\Gamma:=\mathcal{G}_{\infty,1}=\Gal (K_{\infty}/K_1)$ at $v$, and 
define
\[
\cA_{v}:=\begin{cases}
\Ann_{\Z_p[\![D_v]\!]}\left(\displaystyle
\frac{E(K_{\infty,w})[p^\infty]}{
E(K_{\infty,w})[p^\infty]_{\mathrm{div}}} \right) 
& \text{if $v \mid p$}, \\
\Ann_{\Z_p[\![D_v]\!]}\left(\displaystyle
H^1 \left(K_{\infty,w}^{\ur}/K_{\infty,w}, 
\frac{E(K_{\infty,w}^{\ur})[p^\infty]}{
E(K_{\infty,w}^{\ur})[p^\infty]_{\mathrm{div}}} 
\right)\right) & \text{if $v \nmid p$}, 
\end{cases}
\]
where $w$ is a place of $K_{\infty}$ above $v$.
We set
\[
\cA_{\mathcal{N}}:= \prod_{v \mid p D_E} \cA_{v} \Z_p[\![\Gamma]\!].
\]

\begin{prop}[Control theorem, {\cite[Proposition 7.4.4]{Ru}}]
\label{propcontrthm}
Suppose that $E$ satisfies the condition {\rm (C1)}.
Let $m, n \in \Z_{\ge 0}$ be any integers. 
Then, the following hold.
\begin{enumerate}
\item The restriction map  $H^1(K_m,E[p^\infty]) \longrightarrow 
H^1(K_\infty,E[p^\infty])$ is injective.
\item The natural map $H^1(K_m,E[p^n]) \longrightarrow 
H^1(K_m,E[p^\infty])[p^n]$ is injective.
\item The cokernel of the restriction map 
\[
\Sel_p (K_m,E[p^\infty]) \longrightarrow 
H^0(K_m,\Sel_p( K_\infty ,E[p^\infty]))
\] 
is finite, and annihilated by $\cA_{\mathcal{N}}$.
\item The cokernel of the natural map 
\[
\Sel_p (K_m,E[p^n]) \longrightarrow 
\Sel_p( K_m ,E[p^\infty])[p^n]
\] 
is finite, and independent of $n$.
\end{enumerate}
\end{prop}

\begin{rem}
In {\cite[Proposition 7.4.4]{Ru}} 
two additional assumptions, namely 
Assumption 7.1.4 and Assumption 7.1.5, are assumed, 
but the arguments in the proof of \cite[Proposition 7.4.4]{Ru}
do not need Assumption 7.1.4. 
In our setting, it follows from Hasse-Weil's theorem that 
the $\Z_p [G_{K_1}]$-module 
$T_p(E)$ satisfies Assumption 7.1.5. 
We also note that (C1) for $E$ implies 
\[
\cA_{\mathrm{glob}}:=
\Ann_{\Z_p[\![\Gamma]\!]}(E(K_{\infty}))
=\Z_p[\![\Gamma]\!].
\]
\end{rem}

As in \autoref{secintro}, we fix a topological generator $\gamma$ 
of  $\Gamma$. 
For each $m,n \in \Z_{\ge 1}$, we put  
\[
\Lambda_{m,n} := 
\Z/p^n\Z [\mathcal{G}_{m,1}]
=\Z/p^n\Z [\Gal(K_m/K_1)] 
\simeq \Lambda/(\gamma^{p^m}-1,p^n).
\]
Recall that for any $m,n \in \Z_{\ge 0}$, we put 
$X_{m,n}:=\Sel_p(K_m,E[p^n])^\vee$. 
By \autoref{propcontrthm}, 
we immediately obtain  the following corollary.

\begin{cor}\label{corpropcontrthm}
There exists an integer $\nu_X$ such that 
for any $m,n \in \Z_{> 0}$, the orders of the kernel and the cokernel of 
\(
X_\infty \otimes_{\Lambda} \Lambda_{m,n} 
\longrightarrow  X_{m,n}
\)
are at most $p^{\nu_X}$.
\end{cor}

\begin{rem}\label{remcontrolthmR}
Recall that 
$\Delta=\Gal (K_1/\Q)$.  
Take any $n \in \Z_{>0}$. 
Since the order of $\Delta$
is prime to $p$, 
we have
\[
H^0(\Delta, \Sel_p(K_1,E[p^n]))
\simeq \Sel_p(\Q,E[p^n]),
\]
and hence 
$(X_{1,n})_{\boldsymbol{1}} \simeq X_{0,n}$, 
where $\boldsymbol{1} \in \widehat{\Delta}$ 
denotes the trivial character. 
By \autoref{corpropcontrthm},  
the orders of the 
kernel and the cokernel of
\[
X_{\infty, \boldsymbol{1}} 
\otimes_\Lambda \Lambda_{1,n}
\longrightarrow 
(X_{1,n})_{\boldsymbol{1}} \simeq X_{0,n}
\]
are at most $p^{\nu_X}$.
\end{rem}

\section{Proof of Main results}\label{secpfmain}

In this section, we shall prove our main results, 
in particular, \autoref{thmmain}.
We keep the notation in \autoref{secC} and 
we suppose that the elliptic curve $E$ over $\Q$ has good reduction at an odd prime number $p$.

\subsection{Boundedness of the order of Galois cohomology}

In this paragraph, let us prove the following \autoref{propkerresbound}, 
which is related to the boundedness of the order of 
the kernel and the cokernel of the restriction map
\[
H^1(K_n, E[p^n]) \longrightarrow 
H^1(K^E_n, E[p^n]).
\] 

\begin{prop}\label{propkerresbound}
Suppose that the elliptic curve $E$ satisfies the conditions $\mathrm{(C1)}$ 
and $\mathrm{(C3)}$. 
Then, for any $i \in \set{ 1,2 }$, the set 
\[
\set{\# H^i(K^E_n/K_n, E[p^n])}_{n \ge 0 }
\]
is bounded. 
\end{prop}

In order to prove \autoref{propkerresbound}, 
we need the following lemmas. 

\begin{lem}
	\label{clm:5.2}
We assume the condition $\mathrm{(C1)}$ and 
also $E$ has complex multiplication by an order $\oo$ of 
an imaginary quadratic field $\Q(\sqrt{-d})$. 
Then, the fields 
$\Q(\sqrt{-d})$ and $K_\infty$
are linearly disjoint over $\Q$. 
\end{lem}
\begin{proof}
Assume $\Q(\sqrt{-d})\subseteq K_\infty$ for the contradiction. 
As $E$ is defined over $\mathbb{Q}$, 
every endomorphism of $E$ is defined over $\mathbb{Q}(\sqrt{-d})$ (\cite[Chapter~II, Theorem~2.2 (b)]{Si2}), 
hence over $K_\infty$. 
Recall that $E[p]$ is a free $\oo/p\oo$-module of rank $1$ (\cite[Chapter~II, Proposition~1.4]{Si2}). 
The two dimensional representation $\rho^E_1\colon G_{K_{\infty}} \longrightarrow \Aut_{\mathbb{F}_p}(E[p])$ is given by a character 
$G_{K_{\infty}} \longrightarrow \Aut_{\oo \otimes_{\mathbb{Z}}\mathbb{Z}_p}(E[p]) \simeq (\oo/p\oo)^{\times}$. 
This contradicts to (C1).
\end{proof}

\begin{lem}\label{lemvanishGalVp}
Suppose that $E$ satisfies $\mathrm{(C1)}$ and $\mathrm{(C3)}$. 
Then, for any $i \in \Z_{\ge 0}$, it holds that 
$H^i(K^E_\infty/K_\infty,V_p(E))=0$.
\end{lem}

\begin{proof}
\textbf{(The case:~non CM)}
First, suppose that $E$ does not have complex multiplication.
Recall that $G_\Q$ acts on $\bigwedge^2_{\Z_p} T_p(E)$ 
via the cyclotomic character (cf.~\cite[Chapter V, Section 2]{Si1}).
By Serre's open image theorem \cite{Se3}, the image $H$ of 
the Galois representation  
\[
\xymatrix{\rho^E \colon 
\Gal(K^E_\infty/K_\infty)\,
\ar@{^{(}->}[r] &\Aut_{\Z_p}(T_p(E)) \simeq GL_2(\Z_p)
}
\]
becomes an open subgroup of $SL_2(\Z_p)$.
There exists an open normal standard pro-$p$ 
subgroup $U$ of $H$ (\cite[8.29 Theorem]{DDMS}), 
because $H$ is a $p$-adic Lie group, 
By \cite[Chapter V, (2.4.9) Th\'eor\`eme]{La}, 
we have 
\[
H^q(U,V_p(E))=H^q(Lie(U),V_p(E))
\]
for any $q \ge 0$.
Since $Lie(U)$ is an open Lie-subalgebra of 
\[
\mathfrak{sl}_2(\Z_p)
:= \set{ A \in p M_2(\Z_p) \mathrel{\vert} \Tr A=0 },
\]
a matrix of the form $\left(
\begin{array}{cc}
1+p^n & 0 \\
0 & -(1+p^n)
\end{array}
\right)$ for some $n$ belongs to $Lie(U)$.
By \cite[TH\'EOR\`EM 1]{Se2}, we obtain 
$H^q(Lie(U),V_p(E))=0$. 
Hence, the Hochschild--Serre spectral sequence 
\[
E^{pq}_2=H^p(H/U,H^q(U,V_p(E))) \Longrightarrow H^{p+q} (H,V_p(E))
\]
implies that 
$H^i(K^E_\infty/K_\infty,V_p(E))=H^i(H,V_p(E))=0$ for any $i \ge 0$.

\smallskip\noindent
\textbf{(The case:~CM)}
Next, let us assume that $E$ has complex multiplication. 
By the assumption (C3), 
the ring $\End(E)$ of 
endomorphisms of $E$ defined over $\overline{\mathbb{Q}}$
is the maximal order $\oo$ of 
an imaginary quadratic field $L := \Q(\sqrt{-d})$. 
Put $L_{\infty}^E = LK_{\infty}^E$.
Since $E$ is defined over $\mathbb{Q}$, 
every element of $\End(E)$ is defined over 
$L$ (\cite[Chapter~II, Theorem~2.2(b)]{Si2}).
Consider the representation $\rho\colon G_L \longrightarrow \Aut(T_p(E))$ 
which is arising from the action of $G_L$ on $T_p(E)$. 
This factors through an injective homomorphism $\Gal(L_{\infty}^E/L)\longrightarrow \Aut(T_p(E))$ 
which is also denoted by $\rho$. 
The Tate module 
$T_p(E) = \varprojlim_n E[p^n]$ 
is a free $\oo \otimes_{\Z} \Z_p$-module of rank $1$ 
because $E[p^n]$ is a free $\oo/p^n\oo$-module of rank $1$ (\cite[Chapter~II, Proposition~1.4]{Si2}),
As we noted above, every endomorphism of $E$ is defined over $L$, 
the action of $\Gal(L_{\infty}^E/L)$ 
commutes with the scalar multiplication by $\oo$, 
and we obtain the commutative diagram 
\begin{equation}
\label{eq:rhoEo}
\vcenter{\xymatrix{
\Gal(L^E_{\infty}/L) \ar@{^{(}->}[r]^{\rho}\ar@{^{(}->}[rd]_{\rho_{\oo}} & \Aut(T_p(E)) \\
& \ar@{^{(}->}[u]\qquad \Aut_{\oo\otimes_{\Z} \Zp}(T_p(E)) \simeq (\oo\otimes_{\Z}\Zp)^{\times}.
}}
\end{equation}
In particular, the extension $L_{\infty}^E/L$ is an abelian extension. 
The short exact sequence 
\[
\xymatrix@C=5mm{
0 \ar[r]& \Gal(L_{\infty}^E/L) \ar[r]& \Gal(L_{\infty}^E/\Q) \ar[r]& \Gal(L/\Q)\ar[r]& 0
}
\]
induces the action of $\Gal(L/\Q)$ to $\Gal(L_{\infty}^E/L)$.  
In fact, let $c$ be the unique generator of $\Gal(L/\Q)$ 
and take $\widetilde{c} \in \Gal(L_{\infty}^E/\Q)$ a lift of $c$. 
The action of $\Gal(L/\Q)$ on $\Gal(L_{\infty}^E/L)$ 
is given by 
$\sigma \longmapsto \widetilde{c}\sigma \widetilde{c}^{-1}$. 
The induced map $\rho_{\oo}$ preserves the action of $\Gal(L/\Q)$. 
Let $\pi_{\oo^\times}\colon 
\Aut_{\oo \otimes_{\Z} \Z_p}(T_p(E))
=(\oo \otimes_{\Z} \Z_p)^\times \longrightarrow 
(\oo \otimes_{\Z} \Z_p)^\times/\oo^\times$
be the natural surjection. 
We denote by $H'$ the image of $\rho^E_{\oo}$, 
and by $\overline{H}'$ that of
$\pi_{\oo^\times} \circ \rho^E_{\oo}$.
Let $L_{\overline{H}'}$ be the maximal subfield 
of $L^E_{\infty}/L$ 
fixed by the kernel of 
$\pi_{\oo^\times} \circ \rho^E_{\oo}$. 
We have 
\begin{equation}
	\label{eq:Gal(L^E/L)}
	\Gal(L^E_\infty/L)\simeq H'\subset(\oo\otimes_{\Z}\Zp)^{\times},\quad \mbox{and} \quad \Gal(L_{\overline{H}'}/L)\simeq \overline{H}'\subseteq (\oo\otimes_{\Z}\Zp)^{\times}/\oo^{\times}.
\end{equation}

\setcounter{claim}{0}
\begin{claim}
\label{claim:5.2.LH}
The extension $L_{\overline{H}'}/L$ 
is the maximal abelian extension
unramified outside $p$. 
\end{claim}
\begin{proof}
The elliptic curve $E$ is defined over $\Q$ so that 
the class number of $L$ is $1$ (\cite[Chapter II, Theorem 4.1]{Si2}). 
We denote by $L_{\overline{H}'_n}$  be the fixed field of $L_n^E := L(E[p^n])$ 
by the kernel of the composition
\[
\xymatrix@C=6mm{
\Gal(L_n^E/L)\,\ar@{^{(}->}[r] & \Aut_{\oo\otimes \Z/p^n\Z}(E[p^n]) \ar@{->>}[r]&  
\dfrac{\Aut_{\oo\otimes \Z/p^n\Z}(E[p^n])}{\Aut(E)}
\simeq (\oo/p^n\oo)^{\times}/\oo^{\times}.
}
\]
By the theory of complex multiplication (\cite[Chapter II, Theorem 5.6]{Si2}), 
$L_{\overline{H}'_n}$ is the ray class field of $L$ modulo $p^n\oo$. 
The claim follows from $L_{\overline{H}'} = \bigcup_n L_{\overline{H}'_n}$.
\end{proof}
By the global class field theory, 
the above claim implies that the group 
$\overline{H}'\simeq\Gal(L_{\overline{H}'}/L)$
has a quotient isomorphic to $\Zp^2$ 
(see, for instance, \cite[Chapter 13, Proposition 13.2 and Theorem 13.4]{Wa}).
The subgroup $H'$ is open in  
$(\oo \otimes_{\Z} \Zp)^\times$, and in particular, 
the complex conjugate $c$ acts non-trivially on $H'$.

\begin{claim}
\label{clm:5.2.2}
	The field $K^E_\infty$ contains $L = \Q(\sqrt{-d})$.
\end{claim}
\begin{proof}
If $K^E_\infty$ and $L$ are 
linearly disjoint over $\Q$, then 
the extension $L^E_{\infty} = LK^E_{\infty}/\Q$ 
becomes abelian.
Therefore, the complex conjugate $c$ acts on $\Gal(L_{\infty}^E/L)$ 
trivially, and it acts on $H'$ via $\rho_{\oo}$. 
This contradicts to 
the fact that $c$ acts on $H'$ non-trivially.
\end{proof}

From the above claim, we have $L_{\infty}^E = LK_\infty^E = K_\infty^E$.
\begin{claim}
\label{clm:5.2.3}
There exists a lift $\widetilde{c} \in \Gal(K^E_{\infty}/\Q)$  
of $c$ whose order is two such that 
\[
\Gal(K^E_{\infty}/\Q)=
\braket{\widetilde{c}} \ltimes \Gal(K_{\infty}^E/L)
\simeq \braket{\widetilde{c}} \ltimes H'.
\]
\end{claim}
\begin{proof}
From \autoref{clm:5.2.2}, we have $K^{E}_{\infty}\supseteq L$. 
Fix an embedding $\iota_{\C} \colon K_{\infty}^E \hookrightarrow \C$.
Consider the following short exact sequence:
\[
\xymatrix@R=5mm{
0\ar[r] & \Gal(K_{\infty}^E/L)\ar[d]_{\rho_{\oo}^E}^{\simeq} \ar[r] & \Gal(K_{\infty}^E/\Q) \ar[r] & \Gal(L/\Q) \ar[r]\ar@{=}[d] & 0\\ 
 & H' &  & \braket{c}& .
}
\]
The embedding $\iota_{\C}$ induces a splitting of this short exact sequence 
which sends $c$ to the restriction  
$\widetilde{c} \in \Gal(K^E_{\infty}/\Q)$ of the complex conjugation 
after regarding $K^E_{\infty}$ as a subfield of $\C$ via $\iota_{\C}$. 
This splitting gives 
$\Gal(K^E_{\infty}/\Q) \simeq 
\braket{\widetilde{c}} \ltimes H'$.
%
%
\end{proof}


\begin{claim}
\label{clm:Linf}
Putting $L_{\infty} = LK_{\infty}$, 
we have $L_{\overline{H}'} \cap \Q^{\ab}=
L_\infty$.
\end{claim}
\begin{proof}
By \autoref{clm:5.2}, the fields 
$K_{\infty}$ and $L = \Q(\sqrt{-d})$ are linearly disjoint. 
The composition field 
$L_{\infty} = K_{\infty}L$ is an abelian extension of $\Q$ 
so that $L_{\infty}\subseteq \Q^{\ab}$. 
The extension $K_{\infty} = \Q(\mu_{p^\infty}) = \bigcup_n \Q(\mu_{p^n})$ of $\Q$ 
is 
unramified outside $p$ 
and hence 
the extension $L_{\infty} = K_{\infty}L/L$ is 
unramified outside $p$.
Let us show that $L_{\overline{H}'} 
\cap \mathbb{Q}^{\mathrm{ab}} = L_{\infty}$. 
\autoref{claim:5.2.LH} implies that 
$L_{\overline{H}'} 
\cap \mathbb{Q}^{\mathrm{ab}} \supseteq L_\infty$ 
because 
the extension $L_\infty /L$ is unramified outside $p$. 
Accordingly, it suffices to show that 
$L_{\overline{H}'} 
\cap \mathbb{Q}^{\mathrm{ab}} \subseteq L_\infty$. 
As $E$ is defined over $\mathbb{Q}$, the class number of $L$ is one.
Put $p^*:=(-1)^{(p-1)/2}p$.
\autoref{clm:5.2} implies that $\mathbb{Q}(\sqrt{p^*})$ 
and $L$ are linearly disjoint over $\mathbb{Q}$ because 
$\mathbb{Q}(\sqrt{p^*})$ is contained in 
$K_1=\mathbb{Q}(\mu_p)$. 
We deduce that $p$ is unramified in $L/\mathbb{Q}$.
In fact, if $p$ were ramified in $L/\mathbb{Q}$, 
the Hilbert class field of $L$ would contain 
the quadratic extension $L(\sqrt{p^*})/L$. 
Since $L$ is the imaginary quadratic field of class number one, 
there exists a unique prime   
$q_L \in \set{2, 3, 7,8,11,19,43,67,163 }$ 
which is ramified in $L/\Q$. 

For each prime $\ell$, we denote 
by $I_\ell$ the inertia subgroup of 
$\Gal ((L_{\overline{H}'} 
\cap \mathbb{Q}^{\mathrm{ab}})/\Q)$ at $\ell$. 
We define $L_1$ to be the subfield of 
$L_{\overline{H}'} 
\cap \mathbb{Q}^{\mathrm{ab}}$ 
fixed by $I_p$, and 
$L_2$ to be that fixed by $I_{q_L}$. 
The extension $L_{\overline{H}'}/L$ is unramified outside $p$,  
and $L$ has class number one. 
The extension $(L_{\overline{H}'} 
\cap \mathbb{Q}^{\mathrm{ab}})/L$ does not contain 
the proper extension field of $L$ 
where every place above $p$ is unramified. 
As $p$ is unramified in $L/\mathbb{Q}$, 
we obtain $L_1=L$.  
The field $L_2$ coincides with 
the maximal intermediate field of 
$(L_{\overline{H}'} 
\cap \mathbb{Q}^{\mathrm{ab}})/\mathbb{Q}$ 
unramified outside $p$ 
because the extension 
$L_{\overline{H}'}/\Q$ is  uramified outside $\set{p,q_L}$.
The inclusion $L_2 \subseteq K_\infty$ holds, 
because 
$K_\infty/\mathbb{Q}$ 
is the maximal abelian extension unramified outside $p$. 
As a result, we obtain $L_1L_2 \subseteq L_\infty$.
Additionally, the extension 
$L_{\overline{H}'} 
\cap \mathbb{Q}^{\mathrm{ab}}$ of $L_1=L$ is unramified outside $p$. 
In particular, the extension  
$(L_{\overline{H}'} 
\cap \mathbb{Q}^{\mathrm{ab}})/L_1$ is unramified at $q_L$.  
Because of this, we have  
\[
I_p \cap I_{q_L}=
\Gal ((L_{\overline{H}'} 
\cap \mathbb{Q}^{\mathrm{ab}})/L_1) \cap 
I_{q_L}=\set{ 1 }.
\]
Consequently, we deduce that $L_{\overline{H}'} 
\cap \mathbb{Q}^{\mathrm{ab}}
=L_1L_2 \subseteq L_\infty$.
\end{proof}

By this \autoref{clm:Linf}, the abelianization  
of the Galois group $\Gal(L_{\overline{H}'}/\Q)$ is 
\begin{equation}
\label{eq:Linf}
	\Gal(L_{\overline{H}'}/\Q)^{\ab} = \Gal(L_{\overline{H}'}\cap \Q^{\ab}/\Q) = \Gal(L_\infty/\Q).
\end{equation}
The abelianization  
$\Gal(L_{\overline{H}'}/\Q)^{\ab}$ 
is the maximal quotient of $\Gal(L_{\overline{H}'}/\Q)$ 
where $c$ acts trivially, and we have 
$\Gal(L_{\overline{H}'}/\Q) \simeq 
\braket{\widetilde{c}}
\ltimes \overline{H}'$ by 
\autoref{clm:5.2.3}.  
Therefore, we obtain 
\begin{align*}
\Gal(L_\infty/\Q)& \stackrel{\eqref{eq:Linf}}{=}
\Gal(L_{\overline{H}'}/\Q)^{\ab}  \\
&\simeq \big(\braket{\widetilde{c}}
\ltimes \overline{H}'\big)/\big(\braket{\widetilde{c}} \ltimes (1-c)\overline{H}'
\big) \\
&\simeq \overline{H}'/(1-c)\overline{H}'.
\end{align*}
(Here, the group operation of  
$\overline{H}'$ is written in additive manner.)
Let $H'_\infty$ be the inverse image of 
$(1-c)\overline{H}'$ by 
$\pi_{\oo^\times}\vert_{H'} \colon H' \subseteq  
(\oo \otimes_{\Z} \Z_p)^\times \longrightarrow 
(\oo \otimes_{\Z} \Z_p)^\times/\oo^\times
$.
By \eqref{eq:Gal(L^E/L)}, 
we have 
\begin{equation}
	\label{eq:H0'}
	\Gal(L_{\overline{H}'}/L_{\infty}) \simeq (1-c)\overline{H}',\quad \mbox{and}\quad \Gal(L_{\infty}^E/L_{\infty}) \stackrel{\mathrm{\autoref{clm:5.2.2}}}{=} \Gal(K_{\infty}^E/L_{\infty}) \simeq H'_\infty.
\end{equation}
By \autoref{clm:5.2}, 
the fields $K_\infty$ and $L$ are 
linearly disjoint over $\Q$. 
We obtain an isomorphism 
$\Gal(L_{\infty}/K_{\infty}) \simeq \Gal(L/\Q)$ 
and an exact sequence
\[
\xymatrix@R=5mm{
0\ar[r] & \Gal(K_{\infty}^E/L_{\infty})\ar[d]_{\rho_{\oo}^E}^{\simeq} \ar[r] & \Gal(K_{\infty}^E/K_{\infty}) \ar[r] & \Gal(L_{\infty}/K_{\infty}) \ar[r]\ar[d]^{\simeq} & 0.\\ 
 & H'_\infty &  & \braket{c}&
}
\]
There exists a lift $\widetilde{c}' \in 
\Gal(K^E_{\infty}/K_\infty)$ of $c$. 
Note that $\widetilde{c}'$ and 
$\Gal(K^E_{\infty}/L_\infty) \simeq H'_\infty$ generate  
$\Gal(K^E_{\infty}/K_\infty)$, and we have 
$(\Gal(K^E_{\infty}/K_\infty):H'_\infty)=2$.
\[
\xymatrix@R=5mm{
 & L_{\infty}^E \ar@{--}@/_8mm/[rrdddd]|(.2){H'}\\
K_{\infty}^E \ar@{=}[ru] & & & L_{\overline{H}'}\ar@{-}[llu]\\ 
& & L_\infty \ar@{-}[ru]\ar@{-}[luu]\ar@{--}@/_4mm/[luu]|{H_\infty'}\\ 
& K_{\infty}\ar@{-}[ru]\ar@{-}[luu]\\ 
 & & & L \ar@{-}[luu]\ar@{-}[uuu]\ar@{--}@/_5mm/[uuu]|{\overline{H}'}\\
& &\Q\ar@{-}[ru]\ar@{-}[luu]
}
\]

\begin{claim}
	We have $H_\infty' \subseteq H'[(1+c)^2]$. 
Here, the $(1+c)^2$-torsion part of a 
$\mathbb{Z}[\Gal(L/\Q)]$-module $M$ is 
denoted by $M[(1+c)^2]$.
\end{claim}
\begin{proof}
Note that $(1-c)\overline{H}'$ is contained in 
$\overline{H}'[1+c]$ 
and $\oo^\times$ is contained in $H'[1+c]$. 
For any $x \in H_\infty' = \pi_{\oo^{\times}}^{-1}((1-c)\overline{H}')$, 
we have
$\pi_{\oo^{\times}}(x) \in (1-c)\overline{H}'\subseteq \overline{H}'[1+c]$. 
For $(1+c)x\in \Ker(\pi_{\oo^{\times}}) = \oo^{\times} \subseteq H'[1+c]$, 
we obtain 
$(1+c)^2x = (1+c)(1+c)x = 0$.
\end{proof}

Put $V:=(\oo \otimes_{\mathbb{Z}} \Z_p)^\times \otimes_{\Z}
\mathbb{Q}_p \simeq \mathbb{Q}_p^2$. 
Since $c$ acts on $V$ non-trivially, and  
$1+p \in V$ is a non-trivial element fixed by $c$, 
the eigenvalues of the action of $c$ on $V$ 
are $1$ and $-1$.
The group 
$(\oo \otimes_{\Z} \Zp)^\times[(1+c)^2]$ 
has a subgroup of finite index which is isomorphic to $\Zp$.
This implies that 
there exists an element 
$x  \in H'_\infty$ of infinite order 
such that the closure $H_\infty$ 
of $\langle x \rangle$ has finite index in $H'_\infty$.
Fix an embedding 
$\iota_p \colon L \hookrightarrow 
\overline{\Q}_p$.
The embedding $\iota_p$ induces  the ring homomorphism 
\(
\tilde{\iota}_p 
\colon \oo \otimes_{\Z} \Z_p \longrightarrow \overline{\Q}_p\) 
sending $a \otimes b$ to $\iota_p (a)b$.
The eigenvalues of the action of $x$ on 
$V_p(E) \otimes_{\Q_p} \overline{\Q}_p$ are  
$\tilde{\iota}_p(x)$ and 
$\tilde{\iota}_p(c(x))=\tilde{\iota}_p(x)^{-1}$. 
We obtain 
$V_p(E)[x-1]=0 $ and $V_p(E)/(1-x)=0$. 
Note that $H_\infty$ is topologically  generated by $x$. 
By \cite[(1.7.7) Proposition] {NSW} combined with  
\cite[(2.2) Corollary and (2.3) Proposition]{Ta}, 
it holds that 
$H^q(H_\infty,V_p(E))=0$ 
for any $q \ge 0$. 
Let us identify $H'_\infty$ with 
$\Gal(K^E_{\infty}/L_\infty)$.
We may regard $H_\infty$ 
as a normal subgroup of 
$\Gal(K^E_{\infty}/K_\infty)$
because $c$ acts on $H_\infty$ by $x \longmapsto x^{-1}$.
Hence,  by the Hochschild--Serre spectral sequence 
\[
E^{pq}_2=H^p(\Gal(K^E_{\infty}/K_\infty)
/H_\infty, 
H^q(H_\infty,V_p(E))) \Longrightarrow 
H^{p+q} (K^E_{\infty}/K_\infty,V_p(E)), 
\]
we deduce that 
$H^i(K^E_\infty/K_\infty,V_p(E))=0$ for any $i \ge 0$.
\end{proof}

In the proof of 
\autoref{propkerresbound}, 
we use a corollary of  
the following well-known lemma called
topological Nakayama's lemma.

\begin{lem}[Topological Nakayama's lemma]\label{lemtopNak}
Let $(R, \mathfrak{m})$ be a Noetherian 
complete local ring whose residue field is finite,  
and $M$ a compact Hausdorff $R$-module. 
Suppose that $\dim_{R/\mathfrak{m}}M/\mathfrak{m}M < \infty$.
Then, the $R$-module $M$ is finitely generated.
\end{lem}

\begin{proof}
Since $M$ is compact, for any neighborhood $U$ of $0 \in M$, 
there exists an integer $n \in \Z_{\ge 0}$ such that 
$\mathfrak{m}^n M \subseteq U$. 
As $M$ is Hausdorff, we obtain $\bigcap_{n \ge 0} \mathfrak{m}^n M=0$. 
By \cite[Exercise 7.2]{Ei}, we deduce that 
$M$ is finitely generated over $R$ 
if $\dim_{R/\mathfrak{m}}M/\mathfrak{m}M < \infty$.
(See \cite[Lemma 13.16]{Wa} for the proof of 
\autoref{lemtopNak} in the case when 
$R=\Z_p[\![T]\!]$.) 
\end{proof}

\begin{cor}\label{cortopNak}
Let $M$ be a torsion $\Z_p$-module satisfying 
$\dim_{\F_p}M[p]< \infty$.
Then, it holds that $M$ is a cofinitely generated $\Z_p$-module. 
\end{cor}

\begin{proof}
We regard $M$ as a topological group equipped with the discrete topology.
By applying \autoref{lemtopNak}
to the Pontrjagin dual of $M$, 
we  obtain \autoref{cortopNak}. 
\end{proof}

\begin{proof}[\textbf{Proof of \autoref{propkerresbound}}]
Take any $i \in \set{1 ,2}$.
Let $j \in \Z$ be any integer satisfying $0 \le j \le i$.
The group $\Gal(K^E_\infty/K_\infty)$ is topologically 
finitely presented because it
is isomorphic to a closed subgroup of  $GL_2(\Z_p)$.
This implies that 
$H^j(K^E_\infty/K_\infty,E[p])$
is of  finite order. 
The long exact sequence arising from the short exact sequence 
$0 \longrightarrow E[p]\longrightarrow E[p^{\infty}]\xrightarrow{\ p\ } E[p^{\infty}]\longrightarrow 0$ 
induces the surjective homomorphism 
$\xymatrix{H^j(K^E_{\infty}/K_{\infty}, E[p]) \ar@{->>}[r] & H^j(K^E_\infty/K_{\infty},E[p^{\infty}])[p]}$.
In particular, we have 
\[
\dim_{\F_p} H^j(K^E_\infty/K_\infty,E[p^\infty])[p]
\le \dim_{\F_p} H^j(K^E_\infty/K_\infty,E[p])<\infty.
\]
By  \autoref{cortopNak}, 
it holds that 
$H^j(K^E_\infty/K_\infty,E[p^\infty])$ 
is cofinitely generated over $\Z_p$. 
Moreover, the short exact sequence
$0 \longrightarrow T_p(E) \longrightarrow V_p(E) \longrightarrow E[p^{\infty}] \longrightarrow 0$ 
induces 
\[
H^j(K^E_\infty/K_\infty,V_p(E))\longrightarrow H^j(K^E_\infty/K_\infty,E[p^{\infty}])
\longrightarrow H^{j+1}(K^E_\infty/K_\infty,V_p(E)).
\]
Since $H^{j+1}(K^E_\infty/K_\infty,T_p(E))$ 
does not have a non-trivial divisible $\Z_p$-submodule 
by \cite[(2.1) Proposition]{Ta},
it follows from \autoref{lemvanishGalVp} 
that  
$\# H^j(K^E_\infty/K_\infty,E[p^\infty])
< \infty$. 
Take any $n \in \Z_{\ge 0}$. 
As $K_\infty/K_n$ is a pro-cyclic extension, 
the Hochschild--Serre spectral sequence 
\[
E^{pq}_2=H^p(K_\infty/K_n,
H^q(K^E_{\infty}/K_\infty,E[p^\infty])) \Longrightarrow 
H^{p+q} (K^E_{\infty}/K_n,E[p^\infty]) 
\]
implies that 
\[
\# H^i(K^E_\infty/K_n, E[p^\infty])
\le \prod_{q \le i} \set{ \# 
H^q(K^E_{\infty}/K_\infty,E[p^\infty]) }
< \infty.
\]
Therefore, the sequence 
$\set{
\# H^i(K^E_\infty/K_n, E[p^\infty])
}_{n \ge 0 }$ 
is bounded. The exact sequence
\[
H^{i-1}(K^E_\infty/K_n, E[p^\infty])/p^n\longrightarrow 
H^i(K^E_\infty/K_n, E[p^n]) \longrightarrow 
 H^i(K^E_\infty/K_n, E[p^\infty]) [p^n] 
\]
%
implies that $\set{ \# H^i(K^E_\infty/K_n, E[p^n]) }_{n\ge 0}$ is 
bounded. 
The inflation map 
\[
H^1(K^E_n/K_n, E[p^n]) \longrightarrow 
H^1(K^E_\infty/K_n, E[p^n])
\] 
is injective (\cite[Proposition B.2.5]{Ru}). 
the assertion of \autoref{propkerresbound} for $i=1$ follows from this.
In order to  prove \autoref{propkerresbound} for $i=2$, 
by considering the inflation-restriction sequence 
\begin{equation}
\label{seq:InfRes}
H^1(K_{\infty}^E/K_n^E,E[p^n])^{\Gal(K_n^E/K_n)} \longrightarrow  H^2(K_n^E/K_n,E[p^n])\longrightarrow H^2(K_{\infty}^E/K_n,E[p^n])
\end{equation}
(\cite[Proposition B.2.5 (ii)]{Ru}), 
it suffices to show that 
the order of 
\[
H^0(K_n, \Hom(\Gal(K^E_\infty/K^E_n),E[p^n])) 
\]
is bounded. 
Put $H_{n,m} := H^0(K_n, \Hom(\Gal(K^E_\infty/K^E_n),E[p^m]))$. 
The short exact sequence $0\longrightarrow E[p] \longrightarrow E[p^{n}] \longrightarrow E[p^{n-1}] \longrightarrow 0$ 
induces  an exact sequence  
\[
0 \longrightarrow H_{n,1} \longrightarrow H_{n,m}\longrightarrow H_{n,m-1}.
\]
The lemma below (\autoref{claimHomvanish}) says 
that  there exists an integer $N$ such that 
\[
H_{n,1} = H^0(K_n, \Hom(\Gal(K^E_\infty/K^E_n),E[p])) = 0 
\]
for all $n \ge N$.
Thus,  
we have a sequence 
of injective homomorphisms 
\[
\xymatrix@C=6mm{
H_{n,m}\ar@{^{(}->}[r] & H_{n,m-1} \ar@{^{(}->}[r] &\ \cdots\  \ar@{^{(}->}[r] & H_{n,1}.
}
\]
The lemma below again implies $H_{n,1} = 0$. 
In particular, we have 
\[
H_{n,n} = H^0(K_n, \Hom(\Gal(K^E_\infty/K^E_n),E[p^n])) = 0
\] 
for all $n\ge N$. 
Therefore, the sequence $\set{\# H^0(K_n, \Hom(\Gal(K^E_\infty/K^E_n),E[p^n]))}_{n\ge 0}$ 
is bounded. 
\end{proof}


\begin{lem}\label{claimHomvanish}
Suppose that $E$ satisfies $\mathrm{(C1)}$ and $\mathrm{(C3)}$. 
There exists an integer $N$ such that 
\[
H^0(K_m, \Hom(\Gal(K^E_\infty/K^E_m),E[p])) 
=0
\]
for any $m \in \Z_{\ge N}$.
\end{lem}
\begin{proof}
\textbf{(The case:~non-CM)}
First, suppose that $E$ does not have complex multiplication.
The representation $\rho^E\colon G_{\Q}\longrightarrow \Aut(T_p(E))\simeq GL_2(\Zp)$ 
factors through $\Gal(K_{\infty}^E/\Q) \longrightarrow GL_2(\Zp)$ 
and is also denoted by $\rho^E$. 
In this non-CM case, Serre's open image theorem \cite{Se3} implies that 
the group
$\rho^E(\Gal(K^E_\infty/\Q))$ is  an open subgroup of $GL_2(\Z_p)$.
We can take an integer $N \in \Z_{\ge 1}$ such that 
$\rho^E(\Gal(K^E_\infty/\Q))$ contains $1+p^NM_2(\Z_p)$.
Take any $m \in \Z_{\ge N}$. 
As we have 
\[
H^0(K_m, \Hom(\Gal(K^E_\infty/K^E_m),E[p]))= \Hom_{\Gal(K_\infty^E/K_m)}(\Gal(K^E_\infty/K^E_m),E[p])),
\]
it is enough to show that 
there is no non-trivial $\Gal(K_{\infty}^E/K_m)$-equivariant homomorphism 
$ \Gal(K_{\infty}^E/K_m^E)\longrightarrow E[p]$. 
The commutative diagram 
\[
\xymatrix{
\Gal(K_{\infty}^E/\Q) \ar@{^{(}->}[r]^-{\rho^E}\ar@{->>}[d] & \Aut(T_p(E))\simeq GL_2(\Zp)\ar@{->>}@<-10mm>[d] \\
\Gal(K_{m}^E/\Q) \ar@{^{(}->}[r] & \Aut(E[p^m])\simeq GL_2(\Z/p^m\Z)
}
\]
indicates that $\rho^E(\Gal(K^E_\infty/K^E_m)) \subseteq 1+ p^mM_2(\Zp)$. 
As we have  
\[
1 + p^mM_2(\Zp) \subseteq 1 + p^NM_2(\Zp) \subseteq \rho^E(\Gal(K_\infty^E/\Q)),
\]
it holds
\[
\rho^E(\Gal(K^E_\infty/K^E_m))=\rho^E(\Gal(K_\infty^E/\Q)) \cap 
\big(1+p^m M_2(\Z_p) \big)=1+p^m M_2(\Z_p).
\]
Hence, every group homomorphism 
$f\colon \Gal(K^E_\infty/K^E_m) \longrightarrow E[p]$
factors through 
\[
\Gal(K^E_{m+1}/K^E_m) \simeq \mathfrak{gl}_2(\F_p) 
=M_2(\F_p).
\] 
%
The group $G:=\rho^E(\Gal(K^E_\infty/K_m)) \subseteq GL_2(\Zp)$
acts on $M_2(\F_p)$ via the conjugate action, 
and we have 
\(
M_2(\F_p)=\F_p \oplus \mathfrak{sl}_2(\F_p) 
\)
as $\F_p[G]$-modules, where we set
\[
\mathfrak{sl}_2(\F_p):=
\set{ A \in  M_2(\F_p) | \Tr A=0 }.
\]
The condition (C1) for $E$ implies that 
there is no non-trivial 
$G$-equivariant homomorphism
$\F_p \longrightarrow E[p]$.
Suppose that 
there is a non-trivial $G$-equivariant homomorphism
$f \colon \mathfrak{sl}_2(\F_p) \longrightarrow E[p]$, 
and show that this assumption leads to a contradiction.
Put $V:=\Ker(f)$. 
By (C1), we have
$\dim_{\F_p} \Im(f)=2$, and $\dim_{\F_p} V=1$. 
Take any non-zero $A \in V$. 
\begin{itemize}
\item 
First, let us suppose that $A$ is nilpotent. 
In this case, there exists a matrix $P \in GL_2(\F_p)$ 
such that $A=P\begin{pmatrix}
0 & 1 \\
0 & 0
\end{pmatrix}P^{-1}$.
Since $G$ acts via the conjugate action 
on the space $V := \Fp A$, for any $B\in G$ 
there exists  $a\in \Fp^{\times}$ such that 
$BAB^{-1} = aA$.
This implies that if $v \in \Fp^2$ is 
an eigenvector  of $A$, then  
$Bv$ is also an eigenvector of $A$. 
As a result, the group 
$G$ is contained in 
the Borel subgroup 
$P \begin{pmatrix}
\F_p^\times & \F_p \\
0 & \F_p^\times
\end{pmatrix}P^{-1}$ of $GL_2(\F_p)$.
This implies that 
$G$ acts on 
the subspace $W \subseteq \mathfrak{sl}_2(\F_p)\smallsetminus V$ generated by 
$P\begin{pmatrix}
1 & 0 \\
0 & -1
\end{pmatrix}P^{-1}$. 
In fact, for 
$Q = P \begin{pmatrix}
	a & b \\ 0 & d
\end{pmatrix}P^{-1} \in G$, 
we have $QP\begin{pmatrix}
	1 & 0 \\ 0 & -1 
\end{pmatrix}
P^{-1}Q^{-1} \in W$.
The image of 
$W$ by $f$ spans 
a proper $G$-stable $\F_p$-subspace of $E[p]$.
This contradicts to (C1). 

\item Next, suppose that 
$A$ is not nilpotent. 
If we assume the matrix $A\in \mathfrak{sl}_2(\F_p)$ has one eigenvalue $\alpha$, 
then $0 = \Tr A = 2\alpha$. 
Since $p$ is odd, we have $\alpha = 0$ and $A$ is nilpotent.
The matrix  
$A$ has two distinct eigenvalues $\alpha, -\alpha$ in 
$\overline{\F}_p$.  
Since $V$ is stable under the conjugate action of $G$, 
for any $B\in G$,  
there exists some $a \in \Fp^{\times}$ such that 
$BAB^{-1} = a A$. 
For each eigenvalue $\beta \in \set{\alpha,-\alpha}$ of $A$, 
we denote by $V_\beta \subseteq \overline{\mathbb{F}_p}^2$ 
the eigenspace associated with the eigenvalue $\beta$. 
Take any non-zero $v \in V_{\alpha}$. 
Note that 
$Bv$ is also an eigenvector of $A$, 
for we have $BA=aAB$. 
Suppose that $Bv \in V_{-\alpha}$. 
The group $G$ acts on 
$\set{ V_\alpha , V_{-\alpha } }$ 
transitively, and 
$G$ has a subgroup of index $2$. 
This contradicts to the fact that 
$G$ is a pro-$p$-group.
Because of this, we obtain $Bv \in V_{\alpha}$.
This implies that $V_\alpha$ is 
$G$-stable. This contradicts to (C1).
\end{itemize}
Hence, there is  no non-trivial 
$G$-equivariant homomorphism
$\mathfrak{sl}_2(\F_p) \longrightarrow E[p]$, 
and the assertion  
for the non-CM case is finished.

\smallskip\noindent
\textbf{(The case:~CM)}
Suppose that $E$ has complex multiplication. 
By the assumption (C3), the ring $\End(E)$ is the maximal order $\oo$ of 
some imaginary quadratic field $L := \Q(\sqrt{-d})$.
As we shall see below, in this case, we can take $N:=1$. 
Take any $m \in \Z_{\ge 1}$,  
and 
put $G:=\Gal(K^E_{\infty}/K_m)$. 
Let $H'_m$ be the subgroup of 
$(\oo \otimes_{\Z} \Z_p)^\times$ 
corresponding to 
$\Gal(K^E_{\infty}/L_m)$
by 
$$
\rho_{\oo}^E \colon \Gal(K^E_\infty/L)
\longrightarrow \Aut_{\oo\otimes_{\mathbb{Z}}\mathbb{Z}_p}
(T_p(E))=(\oo \otimes_{\mathbb{Z}}\mathbb{Z}_p)^\times, 
$$
where $L_m = K_mL$ 
(cf.\ \eqref{eq:rhoEo}). 
Recall that $L = \Q(\sqrt{-d})$ and $K_\infty$ 
are linearly disjoint over $\Q$ (\autoref{clm:5.2}) 
and $LK_{\infty}^E = K_{\infty}^E$ by \autoref{clm:5.2.2} in the proof of \autoref{lemvanishGalVp}. 
There exists a lift 
$\widetilde{c}_m \in G = \Gal(K^E_{\infty}/K_m)$
of the generator $c \in \Gal(L/\Q)$. 
Note that $G$ is generated by 
$\widetilde{c}_m$ and $H'_m$, and $H'_m$ is a normal 
subgroup of 
$G$ of index two.

\setcounter{claim}{0}
\begin{oneclaim}
There exists a 
non-trivial 
element of $H_m'$ whose order is prime to $p$. 	
\end{oneclaim}
\begin{proof}
Suppose that $H_m'$ has no non-trivial element 
whose order is prime to $p$ for the contradiction.
Then $H'_m$ becomes a pro-$p$ group, and hence 
there exists a non-zero element $a \in E[p]$ fixed by $H'_m$ 
(cf.\ \cite[{{\sc Chapitre}~IX, \S 1, {\sc Lemme}~2}]{Se1}).
\begin{itemize}
\item 
If $a$ is an eigenvector of $\widetilde{c}_m$, then $a$ spans 
a proper $G$-stable $\F_p$-subspace of $E[p]$. 
\item 
Let us suppose that $a$ is  
not an eigenvector of $\widetilde{c}_m$. 
Note that $H'_m$ acts trivially on 
both $a$ and 
$\widetilde{c}_m(a)$, 
for $H'_m$ is a normal subgroup of $G$. 
Since $E[p]$ is spanned by $\set{a , \widetilde{c}_m(a) }$ over $\F_p$, 
the action of $H'_m$ on $E[p]$ is trivial.
The action of $G$ on $E[p]$ 
factors through the cyclic group 
$G/H'_m$ of order two, especially prime to $p$, 
generated by the image of $\widetilde{c}_m$.
\end{itemize} 
In any cases, it contradicts to (C1).
As a result, there exists a 
non-trivial 
element $H'_m$ whose order is prime to $p$. 
\end{proof}

Take any 
non-trivial 
element $x \in H'_m$ whose order is prime to $p$.
Since  $[K_m:K_1]$ and 
$[K^E_{m}:K^E_1]$ are powers of $p$, 
we may regard $x$ as an element of $\Gal(K^E_\infty/K_m)$.  
Since the order of $x \in (\oo \otimes_{\Z} \Z_p)^\times$ is prime to $p$, 
we also note that there is no non-trivial element of $E[p]$  fixed by $x$.
However,  
the element $x$ acts trivially on $\Gal(K^E_{m+1}/K^E_m)$ 
because $\Gal(K^E_{m+1}/K^E_m)$ is a subquotient of 
the abelian group $H'_m$ which contains $x$.
This implies that  there is no non-trivial 
$G$-equivariant homomorphism
$\Gal(K^E_{m+1}/K^E_m) \longrightarrow E[p]$.
This completes the proof of  \autoref{claimHomvanish}
\end{proof}

\subsection{The kernel and the cokernel of the restriction maps}

The goal of this subsection is the following proposition  
which is a key of the proof of \autoref{thmmain}.

\begin{prop}\label{proprescyctoKn}
Let $p$ be a prime number 
at which the elliptic curve $E$ has good reduction. 
Suppose that $E$ satisfies the conditions {\rm (C1)}, {\rm (C2)} and {\rm (C3)}. 
Let 
\[
\mathrm{res}_n:=\mathrm{res}_n^{\mathrm{Sel}}
\colon \mathrm{Sel}_p(K_n,E[p^n])
\longrightarrow H^0(K_n, \mathrm{Sel}_p(K^E_n,E[p^n])).
\]
be the restriction map.
Then, the following hold.
\begin{enumerate}
\item There exists a non-negative integer 
$\nu_{\mathrm{res}}^{\Ker}$ such that 
\[
\# \Ker (\mathrm{res}_n) \le p^{\nu_{\mathrm{res}}^{\Ker}}
\]
for any  $n \in \Z_{\ge 0}$. 
\item There exists a non-negative integer 
$\nu_{\mathrm{res}}^{\Coker}$ such that 
\[
\# \Coker (\mathrm{res}_n) \le p^{\nu_{\mathrm{res}}^{\Coker}}
\]
for any  $n \in \Z_{\ge 0}$. 
\end{enumerate}
\end{prop}

%
%

In order to prove \autoref{proprescyctoKn}, 
we need the following theorem: 

\begin{thm}\label{thmboundpotgoodred}
Let $\ell$ be a prime number, and 
$F/\mathbb{Q}_\ell$ a finite extension. 
Fix an embedding 
$\overline{\Q} \hookrightarrow \overline{F}$, and 
regard $\mu_{p^\infty}$ as a subset of $\overline{F}$. 
If $\ell$ is distinct from $p$, 
suppose that $E_{F(\mu_{p^n})}$ has 
additive reduction for any $n\ge 1$. 
Then, the sequence 
\[
\set{ 
\# E(F(\mu_{p^n}))[p^\infty]  
}_{n \ge 0}
\]
is bounded.
\end{thm}
\begin{proof}
For the case $\ell = p$, this follows from Imai's result \cite{Im}. 
Consider the case $\ell \neq p$. 
For every $n\ge 1$, put $F_n := F(\mu_{p^n})$ and we denote by $\kappa_n$ the residue field of $F_n$. 
Following the notation in \cite[Chapter VII, Section 2]{Si1}, 
we denote by $\pi\colon E_{F_n}(F_n)\longrightarrow \widetilde{E}_{F_n}(\kappa_n)$ 
the reduction map. 
We define 
$\widetilde{E}_{F_n,\mathrm{ns}}$ to be the set of non-singular points in the reduction $\widetilde{E}_{F_n}$ 
and put 
$E_{F_n,0}(F_n) := \pi^{-1}(\widetilde{E}_{F_n,\mathrm{ns}}(\kappa_n))$ the group of rational points whose reduction is non-singular. 
The reduction map $\pi$ induces a short exact sequence 
\begin{equation}
\label{seq:EFn}    
0 \longrightarrow E_{F_n,1}(F_n) \longrightarrow E_{F_n,0}(F_n) \longrightarrow \widetilde{E}_{F_n,\mathrm{ns}}(\kappa_n)\longrightarrow 0, 
\end{equation}
where the left term $E_{F_n,1}(F_n)$ is defined by the exactness 
(\cite[Chapter~VII, Proposition~2.1]{Si1}). 
From the assumption that $E_{F_n}$ has additive reduction, 
the order of the quotient $E_{F_n}(F_n)/E_{F_n,0}(F_n)$ is at most $4$ (\cite[Chapter~VII, Theorem~6.1]{Si1}). 
Hence, it is enough to show that  $\set{\# E_{F_n,0}(F_n)[p^{\infty}]}_{n\ge 0}$ is bounded.
The above sequence \eqref{seq:EFn} induces 
\[
 0 \longrightarrow E_{F_n,1}(F_n)[p^m] \longrightarrow E_{F_n,0}[p^m] \longrightarrow \widetilde{E}_{F_n,\mathrm{ns}}(\kappa_n)[p^m] \longrightarrow E_{F_n,1}(F_n)/p^mE_{F_n,1}(F_n) 
\]
for any $m\ge 1$. 
Since $E_{F_n,1}(F_n)$ is written by the group associated to the formal group law and has no non-trivial points of order $p^m$ (\cite[Chapter~VII, Proposition~3.1]{Si1}), 
we obtain $E_{F_n,1}(F_n)[p^m] = E_{F_n,1}(F_n)/p^mE_{F_n,1}(F_n) = 0$.
From the assumption that $E_{F_n}$ has additive reduction again, 
it follows that  
$\widetilde{E}_{F_n,\mathrm{ns}}(\kappa_n)$ is isomorphic to the additive group $\kappa_n$ 
(\cite[Chapter~III, Exercise~3.5]{Si1}) 
so that $\widetilde{E}_{F_n,\mathrm{ns}}(\kappa_n)[p^m] = 0$. 
The assertion follows from this.
\end{proof}


\begin{lem}\label{lemboundmultred}
Suppose that $E$ has potentially multiplicative reduction at 
a prime number $\ell$ (distinct from $p$).
Then, there exists an integer $N_\ell$ such that 
for any $n \in \Z_{\ge N_\ell}$ and 
any place $w$ of $K_n^E$ above $\ell$, we have 
\(
p^n E(K^E_{n,w})[p^\infty]=0
\). 
\end{lem}

\begin{proof}
Suppose that $E$ has potentially multiplicative reduction at a prime $\ell$.  

\begin{oneclaim}
There exists a finite Galois extension field 
$L$ of $\Q$ contained in $K^E_\infty$ 
satisfying the following conditions: 

\begin{enumerate}[label=$(\mathrm{\alph*})$]
\item The elliptic curve $E_L$ has split multiplicative reduction 
at every place of $L$ above $\ell$.  
\item Every place of $L$ above $\ell$ is 
inert in $L_{\infty} := L(\mu_{p^\infty})/L$. 
\item There exists $N \in \Z_{\ge 1}$ such that $K_N  \subseteq L \subseteq K^E_N$.
\end{enumerate}
\end{oneclaim}
\begin{proof}
By \autoref{lem:K1E}, 
the base change $E_{K_1^E}$ has split multiplicative reduction at every place of $K_1^E$ 
above $\ell$. 
Take any integer $N \in \mathbb{Z}_{>0}$ satisfying 
$\mu_{p^N} \not\subseteq \mathbb{Q}_\ell(E(\overline{\mathbb{Q}}_\ell)[p])$, and  
put 
$L:=K^E_1(\mu_{p^N})$.
As $\mu_{p^N} \subset K_N^E$, 
the conditions (a) and (c) are satisfied. 
Note that $L_\infty :=K^E_1(\mu_{p^\infty})/K^E_1$ is 
a (cyclotomic) $\Z_p$-extension, and 
our choice of $N$ implies that  
the group 
$\Gal(L_\infty /L)$ 
becomes a proper subgroup of the decomposition group 
of $\Gal(L_\infty /K^E_1)$ 
at any place $v$ of $K^E_1$ above $\ell$. 
The condition (b) is satisfied.
\end{proof}

In order to prove  \autoref{lemboundmultred},
it suffices to show that there exists an integer $N'\in \Z_{>0}$ 
such that for any $n \in \Z_{\ge N'}$ and 
any place $w$ of 
$K^E_n$ above $\ell$, it holds that 
\(
E(K^E_{n,w})[p^\infty]=E[p^n].
\)
For the field $L$ and $N\in \Z_{\ge 1}$ given in the above claim, 
take any $n \in \Z_{\ge N}$ and any place $w$ of 
$K^E_n$ above $\ell$.
Let $u$ be the place of $L$ below $w$.
Since $E_{L_u}$ has split multiplicative reduction, 
we have a $G_{L_u}$-invariant isomorphism 
\begin{equation}\label{eqTatemodel}
E(\overline{L}_u) \xrightarrow{\ \simeq\ } \overline{L}_u^\times/q^{\Z}
\end{equation}
for some $q \in {L}_u$ with $\ord_u(q)>0$.
Recall that every place of $L$ above $\ell$ is 
inert in $L_\infty/L$.
By the isomorphism  \eqref{eqTatemodel}, 
if $n \ge N_0:=N+ \ord_u(q)$, then we have 
$E(K^E_{n,w})[p^n] =E[p^n]$, and 
$E(K^{E, \rm ur}_{n,w})[p^\infty] 
\simeq \mu_{p^\infty} \times q^{p^{-n}\Z}/q^{\Z}$. 
\end{proof}


\begin{lem}[{\cite[Example 3.11]{BK}}]\label{lemloccond0}
For any prime number $\ell$ distinct from $p$ and 
any finite extension $F/\Q_\ell$, 
it holds that $H^1_f(F, E[p^\infty])=0$.
\end{lem}

For each $n \in \mathbb{Z}_{\ge 0}$, 
we denote by
$\Sigma_{n,p}$ the set of
all the finite places $v$ of $K_n$  
above $p$, and 
$\Sigma_{n,\mathrm{bad}}$ by the set 
of all the finite places $v$ of $K_n$  
where $E_{K_{n,v}}$ has bad reduction.
We put  
$\Sigma_n:=\Sigma_{n,p} \cup 
\Sigma_{n,\mathrm{bad}}$. 
and define 
$\Sigma_n^0$ to be 
the subset of $\Sigma_n$ consisting of 
all the places $v$ which lies below 
some $w \in \Sigma_m$ 
for every $m \in \mathbb{Z}_{\ge n}$.
Namely, we have 
\begin{equation}
	\label{def:Sigma_n^0}
	\Sigma_n^0 = \Set{v \in \Sigma_n\smallsetminus \Sigma_{n,p} | 
\begin{array}{l}
\text{for any $m\ge 0$, the elliptic curve} \\
\text{$E_{K_m}$ has bad reduction for some $w\mid v$}
\end{array}
}.
\end{equation}

\begin{proof}[\textbf{Proof of \autoref{proprescyctoKn}}]
In this proof, once we fix $n\in \Z_{\ge 0}$ 
and simplify the notation $H^i(F'/F,E[p^n]) = H^i(F'/F)$ for an extension $F'/F$. 
We denote by $K_{n,\Sigma_n}$ the maximal unramified extension of 
$K_n$ outside $\Sigma_n$.
As noted in \autoref{remSigma}, 
the fine Selmer group $\Sel_p(K_n,E[p^n])$ 
is a subgroup of $H^1(K_{n,\Sigma_n}/K_n)$. 
The Hochschild--Serre spectral sequence gives the following commutative diagram with exact rows: 
\begin{equation}
	\label{eq:inf-res}
	\vcenter{
\xymatrix@C=6mm{
0 \ar[r] &\Ker(\res_n^{\loc})\ar[r]    & \Coker(\iota_n) \ar[r]^{\res_n^{\loc}} & \Coker(\iota_n^E) \\
0\ar[r] & H^1(K_n^E/K_n) \ar[u]\ar[r]^{\mathrm{inf}_n} & H^1(K_{n, \Sigma_n}/K_n)\ar@{->>}[u] \ar[r]^-{\mathrm{res}_n} & H^1(K_{n,\Sigma_n}/K_n^E)^{G_{K_n}}\ar@{->>}[u] \ar[r] & H^2(K_n^E/K_n)\\
0\ar[r] & \Ker(\mathrm{res}_n^{\mathrm{Sel}})\ar[r]\ar@{^{(}->}[u] & \Sel_p(K_n,E[p^n])\ar@{^{(}->}[u]^{\iota_n} \ar[r]^-{\mathrm{res}_n^{\mathrm{Sel}}} & \Sel_p(K_n^E,E[p^n])^{G_{K_n}}\ar@{^{(}->}[u]^{\iota_n^E} \ar[r] & \Coker(\mathrm{res}_n^{\mathrm{Sel}}).\ar[u]
}}
\end{equation}
The snake lemma induces the exact sequence  
\begin{equation}
\label{eq:snake}
\begin{tikzcd}
  0\rar & \Ker(\res_n^{\Sel})  \rar & H^1(K_n^E/K_n)  \rar  
             \ar[draw=none]{d}[name=X, anchor=center]{}
    & \Ker(\res_n^{\loc}) \ar[rounded corners,
            to path={ -- ([xshift=2ex]\tikztostart.east)
                      |- (X.center) \tikztonodes
                      -| ([xshift=-2ex]\tikztotarget.west)
                      -- (\tikztotarget)}]{dll}[at end]{\delta} \\      
 &  \Coker(\res_n^{\Sel}) \rar & H^2(K_n^E/K_n).
\end{tikzcd}
\end{equation}
By \autoref{propkerresbound}, 
the order of $H^1(K_n^E/K_n) = H^1(K_n^E/K_n,E[p^n])$ is bounded independently of $n$, 
and so is the kernel $\Ker(\res_n^{\Sel})$.

Let us investigate the cokernel of $\mathrm{res}_n^{\Sel}$. 
For each finite place $v$ in $K_n$, 
we define  
restriction maps 
\begin{align*}
	&\res_{n,v}^{\loc}: H^1(K_{n,v},E[p^n])\longrightarrow H^0\left(K_n,\prod_{w\mid v}H^1(K_{n,w}^E,E[p^n])\right), \\
	&\res_{n,v}^f: H^1_f(K_{n,v},E[p^n])\longrightarrow H^0\left(K_n,\prod_{w\mid v}H^1_f(K_{n,w}^E,E[p^n])\right),\quad  \mbox{and}\\
	&\overline{\res}_{n,v}^{\loc}: 
	\frac{H^1(K_{n,v},E[p^n])}{H^1_f(K_{n,v},E[p^n])} \longrightarrow H^0\left(K_n,\prod_{w\mid v}\frac{H^1(K_{n,w}^E,E[p^n])}{H^1_f(K_{n,w}^E,E[p^n])}\right).
\end{align*}
These maps induce the following commutative diagram with exact rows:  
\begin{equation}
\label{eq:H/H_f}
\vcenter{
\xymatrix@C=4mm{
0\ar[r] & H^1_f(K_{n,v}) \ar[d]^-{\res_{n,v}^f} \ar[r] & H^1(K_{n,v}) \ar[r]\ar[d]^-{\res_{n,v}^{\loc}} & \dfrac{H^1(K_{n,v})}{H^1_f(K_{n,v}) }\ar[r]\ar[d]^-{\overline{\res}_{n,v}^{\loc}}  & 0\\
0\ar[r] & \left(\underset{w\mid v}{\prod} H^1_f(K_{n,w}^E))\right)^{G_{K_n}} \ar[r] & \left(\underset{w\mid v}{\prod} H^1(K_{n,w}^E)\right)^{G_{K_n}} \ar[r] & \left(\underset{w\mid v}{\prod}\dfrac{H^1(K_{n,w}^E)}{H^1_f(K_{n,w}^E)}\right)^{G_{K_n}}
}}
\end{equation}
By \autoref{propkerresbound}, 
the group $H^2(K_n^E/K_n) = H^2(K_n^E/K_n,E[p^n])$ is finite and its order is bounded 
independently of $n$. 
From the exact sequence \eqref{eq:snake}, 
it is enough to give a bound for $\set{\#\Ker(\res_n^{\loc})}_{n\ge 0}$. 
By diagram chaise and the definition of the fine Selmer groups 
(\autoref{defnSel}), 
we have an injective homomorphism 
\begin{equation}
	\label{eq:Kers}
	\xymatrix{
\Ker(\res_n^{\loc})\, \ar@{^{(}->}[r] & \displaystyle \prod_{v\mid p}\Ker(\res_{n,v}^{\loc})\times \prod_{v\nmid p} \Ker(\overline{\res}_{n,v}^{\loc}).
}
\end{equation}
When $E_{K_n}$ has good reduction at 
a finite place $v \nmid p$  of $K_n$, then 
the Tate module $T_\ell(E)$ for the prime number $\ell$ with $v \mid \ell$ 
is unramified (\cite[Chapter VII, Theorem 7.1]{Si1}). 
We have $H^1_f(K_{n,v}, E[p^n]) = H^1_{\ur}(K_{n,v},E[p^n])$ 
and $H^1_f(K_{n,w}^E, E[p^n]) = H^1_{\ur}(K_{n,w}^E,E[p^n])$ (\cite[Lemma~1.3.8 (ii)]{Ru}). 
Moreover, the extension $K_{n,w}^E/K_{n,v}$ is unramified for any $w\mid v$ 
as $E[p^n]$ is unramified. 
From the definition of the unramified cohomology (cf.~Notation), we have a commutative diagram
\[
\xymatrix@C=15mm@R=5mm{
\dfrac{H^1(K_{n,v},E[p^n])}{H^1_{\ur}(K_{n,v},E[p^n])} \ar@{^{(}->}[d] \ar[r]^-{\overline{\res}_{n,v}^{\loc}} & 
\ds \prod_{w\mid v} \left(\frac{H^1(K_{n,w}^E,E[p^n])}{H^1_{\ur}(K_{n,w}^E,E[p^n])}\right)^{G_{K_n}}\ar@{^{(}->}[d] \\
H^1(K_{n,v}^{\ur},E[p^n]) \ar[r]^-{\res_{n,v}^{\ur}} & \ds \prod_{w\mid v} H^1(K_{n,w}^{E, \ur},E[p^n])^{G_{K_n}}. 
}
\]
From the inflation-restriction sequence (\cite[Proposition B.2.5 (i)]{Ru}), the kernel of the bottom horizontal map ${\res_{n,v}^{\ur}}$ is 
$\bigcap_{w\mid v}H^1(K_{n,w}^{E,\ur}/K_{n,v}^{\ur},E[p^n]^{G_{K_{n,v}^{\ur}}}) = 0$ 
and the map $\res_{n,v}^{\ur}$ is injective.  
In particular, we have $\Ker(\overline{\res}_{n,v}^{\loc}) = 0$ for any finite place 
$v\notin \Sigma_n$.
This implies that the order of $\Ker(\overline{\res}_{n,v}^{\loc})$ is bounded independently of $n$ 
for the case $v\nmid p$ and $v \not\in \Sigma_n^0$.
From \eqref{eq:Kers}, 
in order to prove 
\autoref{proprescyctoKn} (2), 
it is left to show the following assertions.
\begin{itemize}
	\item For $v\mid p$, the sequence $\set{\#\Ker(\res_{n,v}^{\loc})}_{n\ge 0}$ is  bounded. 
	\item For $v\nmid p$, and $v\in \Sigma^0_n$, the sequence  $\set{\#\Ker(\overline{\res}_{n,v}^{\loc})}_{n\ge 0}$ is bounded.
\end{itemize}
By applying the snake lemma to the diagram \eqref{eq:H/H_f}, 
there is an exact sequence 
\[
0 \longrightarrow \Ker(\res_{n,v}^f) \longrightarrow \Ker(\res_{n,v}^{\loc})\longrightarrow \Ker(\overline{\res}_{n,v}^{\loc}) \longrightarrow \Coker(\res_{n,v}^f).
\]
The assertion (2) in \autoref{proprescyctoKn} 
follows from the lemma below (\autoref{claimlocalres}).
\end{proof}

By definition (cf.~\eqref{def:Sigma_n^0}), 
we have 
\[
\Sigma_0^0 = \Set{\ell\colon \mbox{prime number} 
|
\begin{array}{l}
\text{for any $m\ge 0$, the elliptic curve} \\
\text{$E_{K_m}$ has bad reduction at a place above $\ell$}
\end{array}
}.
\]
We are assuming $E$ has good reduction at $p$, 
so that $p\not\in\Sigma_0^0$. 

\begin{lem}\label{claimlocalres}
\begin{enumerate}
\item 
For any prime number $\ell \in \Sigma^0_0$ (distinct from $p$),  
the set
\[
\set{ \# \Ker (\mathrm{res}_{n,v}^{\mathrm{loc}}) \mathrel{\vert}
n \ge 0,v \mid \ell }
\]
is bounded.

\item 
For the fixed prime $p$, the set 
\[
\set{ \# \Ker (\mathrm{res}_{n,v}^{\mathrm{loc}}) \mathrel{\vert}
n \ge 0,v \mid p }
\]
is bounded.

\item
For any prime number $\ell \in \Sigma^0_0$ 
(distinct from $p$), 
the set
\[
\set{ \# \Coker(\mathrm{res}_{n,v}^{f}) 
\mathrel{\vert}
n \ge 0,v \mid \ell
}
\]
is bounded. 
\end{enumerate}
\end{lem}

\begin{proof}
First, we prove the following claim.

\setcounter{claim}{0}
\begin{claim}
\label{clm:L}
	There exists a finite Galois extension field $L$ of $\Q$ 
	contained in $K^E_{\infty} = \Q(E[p^\infty])$ 
	satisfying the following conditions.
	\begin{enumerate}[label=$(\mathrm{\alph*})$]
\item The elliptic curve $E_L$ has semistable reduction everywhere.
\item The elliptic curve $E_{L}$ has split multiplicative reduction 
at every place $u$ of $L$ above a prime number $q$ 
where $E$ has potentially multiplicative reduction.
\item Every place of $L$ above every $\ell \in \Sigma^0_0$  is inert in $L_{\infty}:= L(\mu_{p^{\infty}})/L$. 
\item There exists an integer $N\in \Z_{\ge 0}$ such that $K_N\subseteq L \subseteq K_N^E$.
\end{enumerate}
\end{claim}
\begin{proof}
Let $q_0$ be a prime number 
where $E$ has potentially good additive reduction. 
Since $q_0$ is distinct from $p$, 
the order of the image of $G_{\mathbb{Q}_{q_0}^{\mathrm{ur}}}$ in 
$\Aut_{\mathbb{Z}_p}(T_p(E))$ is finite (\cite[Chapter VII, Theorem 7.1]{Si1}).
This implies that there exists an intermediate field 
$F^{(q_0)}_{\mathrm{pg}}$ of 
$K^E_\infty/\mathbb{Q}$ such that 
$F^{(q_0)}_{\mathrm{pg}}/\mathbb{Q}$ is a finite Galois extension, and 
$E_{F^{(q_0)}_{\mathrm{pg}}}$ has good reduction 
at every place above $q_0$.
Let $F_{\mathrm{pg}}$ be the composite of the fields 
$F^{(q)}_{\mathrm{pg}}$ where 
$q$ runs all the prime numbers 
where $E$ has potentially good additive reduction. 
By \autoref{lem:K1E}, 
the composite field 
$L:=F_{\mathrm{pg}}K^E_1$ satisfies the conditions 
(a) and (b). 
Moreover, take a sufficiently large $N \in \mathbb{Z}_{>0}$, 
and replace $L$ with $L(\mu_{p^N})$, 
the additional conditions (c) and (d) follow from 
the similar arguments in 
the proof of \autoref{lemboundmultred}.
Note that 
$L=F_{\mathrm{pg}}K^E_{1}(\mu_{p^N})$ is 
a finite Galois extension field of $\Q$ 
contained in $K^E_\infty$.
\end{proof}

Put $L_n := L(\mu_{p^n})$ for each $n\ge 1$. 
Take any prime number $\ell \in \Sigma^0_0$. 
Fix a place 
$w_\infty$ of $K^E_\infty$ above $\ell$.
For any $m \in \Z_{\ge N}$, 
denote by $w_m$ the place of $K^E_m$ below $w_\infty$ and 
by $u_m$  the place of $L_m$ below $w_\infty$ 
respectively.

Let us prove the assertion (1). 
Take any $n \in \Z_{\ge N}$, and let $v=v_n$ be the place 
of $K_n$ below $w_\infty$.
For the fixed place $w_n$, 
by identifying $G_{K_{n,v}}$ with the decomposition subgroup of $G_{K_n}$ at $v$, 
we consider $H^1(K^E_{n,w_n}, E[p^n])$ as an $G_{K_{n,v}}$-module 
and 
$\prod_{w \mid v}H^1(K^E_{n,w}, E[p^n])$ 
is isomorphic to the induced module $\Ind_{G_{K_n}}^{G_{K_{n,v}}}(H^1(K^E_{n,w_n}, E[p^n]))$.  
Shapiro's lemma gives an isomorphism 
\begin{equation}
\label{eq:H0vw}	
H^0\left(
K_n, 
\prod_{w \mid v}H^1(K^E_{n,w}, E[p^n])
\right)
\simeq H^0\left(
K_{n,v}, 
H^1(K^E_{n,w_n}, E[p^n])
\right)
\end{equation}
(cf.\ \cite[(1.6.4) Proposition]{NSW}). 
By the Hochschild-Serre exact sequence (\cite[Proposition B.2.5 (ii)]{Ru}), 
we obtain the following commutative diagram 
whose rows are exact: 
\[
\xymatrix{
0 \ar[r] & \Ker(\res_{n,v}^{\mathrm{loc}}) \ar[d]\ar[r] & H^1(K_{n,v},E[p^n]) \ar[r]^-{\res_{n,v}^{\mathrm{loc}}}\ar@{=}[d] &\displaystyle{ \prod_{w\mid v} H^1(K_{n,w}^E,E[p^n])^{G_{K_n}}}\ar[d]_{\eqref{eq:H0vw}}^{\simeq}\\
0 \ar[r] & H^1(K_{n,w_n},E[p^n]) \ar[r] & H^1(K_{n,v},E[p^n]) \ar[r] & 
H^1(K_{n,w_n}^E,E[p^n])^{G_{K_n}}.
}
\]
It holds that 
\begin{equation}
\label{eq:Kerloc}	
\Ker (\mathrm{res}_{n,v}^{\mathrm{loc}}) \xrightarrow{\ \simeq\ }
H^1(K^E_{n,w_n}/K_{n,v}, E[p^n]). 
\end{equation}
The order of $\Ker( \mathrm{res}_{n,v}^{\mathrm{loc}})$ 
depends only on the prime number $\ell$ and the positive integer $n$
(in particular, it is independent of the choice of 
the place $w_\infty$ of $K^E_\infty$ 
above the fixed prime number $\ell$).
For any intermediate field $M$ 
of $K^E_{n,w_n}/K_{n,v}$ which is 
Galois over $K_{n,v}$, 
we have an exact sequence
\begin{equation}
    \label{seq:YKZ}
0  \longrightarrow Y_{n}(M) \longrightarrow \Ker(\mathrm{res}_{n,v}^{\mathrm{loc}})
\longrightarrow Z_{n}(M),
\end{equation}
where we put 
\begin{align*}
Y_{n}(M):=& 
H^1(M/K_{n,v}, E(M)[p^n]),\ \mbox{and}\\
Z_{n}(M):=& 
H^0(K_{n,v}, 
H^1(K^E_{n,w_n}/M, E[p^n])).
\end{align*}

First, let us study the cases when $\ell \ne p$. 
Recall that 
$E_{L_{n,u_n}}$ has good or split multiplicative reduction 
and $E_{K_{n,v}}$ has additive reduction from the very definition of $\Sigma_0^0$.

\smallskip 
\noindent {\bf (The case: Potentially good reduction with $\ell \ne p$).} 
Suppose that $\ell \ne p$, and 
$E$ has potentially good reduction at $\ell$. 
Let $M_n$ be the maximal subfield of 
$K^E_{n,w_n}$ which is unramified above $K_{n,v}$. 
As the extension $M_n/K_{n,v}$ is cyclic, 
we have 
\begin{align*}
	H^1(M_n/K_{n,v}, E(M_{n})[p^n])&\simeq \widehat{H}^{-1}(M_n/K_{n,v},E(M_{n})[p^n]) \\
&= \frac{\Ker\big(N_{M_n/K_{n,v}} \colon E(M_n)[p^n] \longrightarrow E(K_{n,v})[p^n]\big)}{\braket{\Frob_v-1}},
\end{align*}
where $\widehat{H}^{\ast}$ stands for the Tate cohomology group, 
$N_{M_n/K_{n,v}}$ is the norm map and $\Frob_v$ is the Frobenius automorphism 
at $v$ which is a generator of the cyclic group $\Gal(M_n/K_{n,v})$ 
(cf.~\cite[{{\sc Chapitre}~VIII, \S 4}]{Se1}).
There are (in)equalities below:
\begin{align*}
\#Y_n(M_n)
&=\# H^1(M_n/K_{n,v}, E(M_{n})[p^n]) \\
&=\# \widehat{H}^{-1}(M_n/K_{n,v}, E(M_{n})[p^n]) \\
& \le \# \left(
\frac{E(M_{n})[p^n]}{\braket{\Frob_{v}-1}}
\right) \\
& =\# (E(M_n)[p^n])[\Frob_{v}-1] \\
& =\# E(K_{n,v})[p^n] \\
& \le \# E(K_{n,v})[p^{\infty}].
\end{align*}
From \autoref{thmboundpotgoodred}, 
the sequence $\set{\# Y_n(M_n) }_{n \ge 0}$ 
is bounded. 
Let us study $Z_n(M_n)$. 
Note that $K^E_{n,w_n}/ L_{n,u_n}$ 
is unramified because $E_{L_{n},u_n}$ has good reduction. 
Since $K^E_{n,w_n}/M_n$ is totally ramified, 
we have 
\[
[K_{n,w_n}^E:M_n] = [L_{n,u_n}:L_{n,u_n}\cap M_n] \le [L_{n,u_n}:K_{n,v}] \le [L:K_N].
\]
This implies that 
\[
\sup_{n \ge N} \# Z_{n}(M_n)
\le \sup_{n \ge N} 
\# H^1(K^E_{n,w_n}/M_n, E[p^n])
<  \infty.
\] 
Consequently, 
the set $\set{ \#\Ker(\mathrm{res}_{n,v}^{\mathrm{loc}})|n \ge 0, v \mid \ell}$
is bounded.

\smallskip 
\noindent \textbf{(The case: Potentially multiplicative reduction with $\ell \ne p$).}
Suppose that $\ell \ne p$, and 
$E$ has potentially multiplicative reduction at $\ell$.
Put $Y_n:=Y_n(L_{n,u_n})$ and 
$Z_n:=Z_n(L_{n,u_n})$.
Let $u:=u_N$ be the place of $L$ below $u_n$. 
The elliptic curve $E_{L_u}$ is isomorphic to a Tate curve 
$\mathbb{G}_m/q_u^{\Z}$, and in particular, we have 
a $G_{L_u}$-equivariant isomorphism 
\begin{equation}\label{eqEtorTate}
E[p^\infty] \simeq  \mu_{p^\infty} \times q_u^{\Z[p^{-1}]}/q_u^{\Z}.
\end{equation}
This implies that $K^E_{\infty,w_\infty}/L_{\infty, u_\infty}$ is 
a totally ramified cyclic extension, where 
$u_{\infty}$ is the place of $L_{\infty}$ below $w_{\infty}$. 
Fix a  topological generator
$\tau \in \Gal(K^E_{\infty,w_\infty}/L_{\infty, u_\infty})$. 
Since $L_{\infty, u_{\infty}}/L_u$ is unramified, 
the homomorphism 
\[
\Gal(K^E_{\infty,w_\infty}/L_{\infty,u_\infty}) 
\longrightarrow \Gal(K^E_{n,w_n}/L_{n,u_n});\ 
\sigma \longmapsto \sigma\vert_{K^E_{n,w_n}} 
\]
is surjective. 
Firstly, we show that $\set{ \# Y_{n} }_{n \ge 0}$ is bounded.
We define 
\[
E'_n:= \left(
E(K^{E, {\rm ur}}_{n,w_n})[p^\infty]_{\rm div}
\right)[p^n].
\]
The isomorphism \eqref{eqEtorTate} implies that 
$E'_n$ is isomorphic to $\mu_{p^n}$ 
as a $\Z_p[G_{L_{u}}]$-module.  
Note that $E'_n$ is $G_{K_{n,v}}$-stable 
as $K^{E, {\rm ur}}_{n,w_n}/K_{n,v}$ is a Galois extension.
We obtain an exact sequence 
\begin{equation}
    \label{seq:Y}
Y'_{n} \longrightarrow 
Y_{n} 
\longrightarrow Y''_{n},
\end{equation}
where we put 
\begin{align*}
Y'_{n}:=& H^1(L_{n,u_n}/K_{n,v},E'_n),\quad \mbox{and}\\ 
Y''_{n}:= &H^1(L_{n,u_n}/K_{n,v}, 
E(L_{n,u_n})[p^n]/E'_n).
\end{align*}
Letu us we study $Y''_{n}$.
Note that $\tau$ acts on 
$T_p(E)$ non-trivially and  unipotently. 
Putting 
$\nu_\tau:= \ord_p(\# (T_p(E)/\braket{\tau-1})_{\rm tor})$, 
we have 
\[
\# (E(L_{n,u_n})[p^n]/E'_n) \le
\# (E(L_{\infty,u_\infty})[p^\infty]/E'_\infty)
=p^{\nu_\tau}.
\]
It follows that the sequence 
\(
\set{ \# Y''_{n} }_{n \ge 0}
\)
is bounded, 
because of the inequaility 
$[L_{n,u_n}:K_{n,v}]\le [L:\Q]$.
Let us consider  $Y'_{n}$.
We define $H_n$ to be the maximal subgroup 
of $\Gal(L_{n,u_n}/K_{n,v})$ 
acting trivially on $E'_n$, and 
$L'_n$ the maximal subfield 
of $L_{n,u_n}$ fixed by $H_n$.
Now, we consider an exact sequence
\[
0 \longrightarrow H^1 (L'_n/K_{n,v}, E'_n)
\longrightarrow Y'_{n}
\longrightarrow H^1 (L_{n,u_n}/L'_n, E'_n).
\]
By (C2) for $E$, we know  
$H^0(K_{n,v}, E'_n)=0$ (\autoref{lemC2}). 
Since $L'_n/K_{n,v}$ is cyclic, we have 
$H^1(L'_n/K_{n,v},E'_n) \simeq \widehat{H}^{-1}(L'_n/K_{n,v},E'_n)$ 
(cf.~\cite[{{\sc Chapitre}~VIII, \S 4}]{Se1}). 
For $E'_n$ is finite, its Herbrand quotient is trivial 
so that $\# \widehat{H}^{-1}(L'_n/K_{n,v},E'_n)  = \#\widehat{H}^0(L'_n/K_{n,v},E'_n)$ (\cite[{{\sc Chapitre}~VIII, \S 4, {\sc Proposition}~8}]{Se1}).
Therefore,  
we have 
\[
\# H^1 (L'_n/K_{n,v}, E'_n)
= \# \widehat{H}^0 (L'_n/K_{n,v}, E'_n) \le \# H^0(K_{n,v},E_n') =1.
\]
Since 
$L_{n,u_n}/L'_n$ is a cyclic extension 
whose order is at most $[L:\Q]$, we have
\[
\# H^1 (L_{n,u_n}/L'_n, E'_n)
= 
\# \Hom_{\Z_p}(\Gal(L_{n,u_n}/L'_n), \Z/p^n\Z)
\le [L:\Q].
\]
The sequence $\set{ \# Y'_{n} }_{n \ge 0}$ is bounded.
This implies that $\set{ \# Y_{n} }_{n \ge 0}$ is bounded by \eqref{seq:Y}.
Secondly, let us show that $\set{ \# Z_{n} }_{n \ge 0}$ is bounded.
We have an exact sequence 
\begin{equation}
\begin{tikzcd}
 H^0(L_{n,u_n}, E[p^n]/E'_n)\rar{\delta_n} 
             \ar[draw=none]{d}[name=X, anchor=center]{}
    &  H^1(K^E_{n,w_n}/L_{n,u_n}, E'_n) 
    \ar[rounded corners,
            to path={ -- ([xshift=2ex]\tikztostart.east)
                      |- (X.center) \tikztonodes
                      -| ([xshift=-2ex]\tikztotarget.west)
                      -- (\tikztotarget)}]{dl} \\      
  H^1(K^E_{n,w_n}/L_{n,u_n}, E[p^n]) \rar & H^1(K^E_{n,w_n}/L_{n,u_n}, E[p^n]/E'_n).
\end{tikzcd}
\end{equation}
We put 
\begin{align*}
Z'_n:=& H^0(K_{n,v}, \Coker (\delta_n) ), \ \mbox{and}\\
Z''_n:=& H^0(K_{n,v}, H^1(K^E_{n,w_n}/L_{n,u_n}, E[p^n]/E'_n)).
\end{align*}
In order to prove that 
$\set{ \# Z_{n} }_{n \ge 0}$ is bounded, 
it suffices to show that 
$\set{ \# Z'_n }_{n \ge 0}$ and 
$\set{ \# Z''_n }_{n \ge 0}$ are bounded. 
Let us show that  $\set{ \# Z'_n }$ is bounded.
The generator $\tau_n:=\tau\vert_{K^E_{n,w_n}}$ of 
the Galois group 
$\Gal(K^E_{n,w_n}/L_{n,u_n})$ acts trivially on 
$E[p^n]/E'_n$. It holds that 
\[
H^0(L_{n,u_n}, E[p^n]/E'_n)=
E[p^n]/E'_n \simeq \Z/p^n.
\]
We also  have an isomorphism 
\begin{align*}
H^1(K^E_{n,w_n}/L_{n,u_n}, E'_n) =
\Hom (\Gal(K^E_{n,w_n}/L_{n,u_n}),E'_n) 
\xrightarrow{\ \simeq \ } 
E'_n[p^{M_n}]
\end{align*}
given by the evaluation at $\tau_n$, 
where $M_n:=\ord_p([K^E_{n,w_n}:L_{n,u_n}])$.
We denote by $\overline{E}'_n$ the image of $E'_n$ in 
$E[p^n]/\braket{\tau-1}$. 
Its order is bounded as 
\(
\# \overline{E}'_n \le p^{\nu_\tau}
\).
By definition, the coboundary map  $\delta_n$ is given by
\[
\delta_n \colon E[p^n]/E'_n \longrightarrow E'_n[p^{M_n}];\ 
(P \ \mathrm{mod}\, E'_n) \longmapsto 
(\tau -1) P.
\]
We obtain 
\[
\# Z'_n \le
\# \Coker (\delta_n) \le 
\# \overline{E}'_n \le p^{\nu_\tau}.
\]
Finally, let us show that $\set{ \# Z''_n}_{n \ge 0}$
is bounded.
Note that we have an injective homomorphism
\[
\xymatrix{
H^1(K^E_{n,w_n}/L_{n,u_n}, E[p^n]/E'_n)
\ar@{^{(}->}[r] & 
\dfrac{E[p^n]/E'_n}{\braket{\tau-1}}
=E[p^n]/E'_n.
}
\]
By (C2) for $E$ and \autoref{lemC2}, 
it holds that 
$Z''_n=H^0(K_{n,v}, E[p^n]/E'_n)=0$.
This implies that $\set{ \# Z_{n} }_{n \ge 0}$ is bounded.
Hence, we deduce that 
$\set{\#\Ker( \mathrm{res}_{n,v}^{\mathrm{loc}}) | n \ge 0, v \mid \ell}$
is bounded. 
%

\smallskip 
Now, suppose that $\ell \ne p$, and 
let us show the assertion (3) of 
\autoref{claimlocalres}. 
Again, take any $n \in \Z_{\ge N}$, and let $v=v_n$ be the place 
of $K_n$ below $w_\infty$. 
The order of $\Coker(\mathrm{res}_{n,v}^{f})$ 
depends only on $\ell$ and $n$. 
By the short exact sequence
$0 \longrightarrow  
E[p^n]
\longrightarrow{}
E[p^\infty]
\xrightarrow{\times p^n}
E[p^\infty]
\longrightarrow 0$,
there is a short exact sequence 
\[
0 \longrightarrow A_n^0  
\xrightarrow{\ \delta\ } H^1(K_{n,v}, E[p^n]) \xrightarrow{\iota_{n,v}} H^1(K_{n,v},E[p^{\infty}])[p^n], 
\]
where 
$A_n^0 := E(K_{n,v})[p^\infty]\otimes_{\Z_p}(\Z/p^n\Z)$. 
Recall that 
$H^1_f(K_{n,v},E[p^n])$ is defined to be the inverse image 
of $H^1_f(K_{n,v},E[p^{\infty}])$ by $\iota_{n,v}$ (cf.\ \eqref{def:iota}). 
Thus, 
the map $\iota_{n,v}$ induces the short exact sequence 
$0 \longrightarrow A_n^0 \longrightarrow B_n^0 \longrightarrow C_n^0$, 
where 
\[
B^0_n:= H^1_f(K_{n,v}, E[p^n]), \ \mbox{and}\ 
C^0_n:= H^1_f(K_{n,v}, E[p^\infty])[p^n].
\]
%
Furthermore, we obtain a commutative diagram
\[
\xymatrix{
0 \ar[r] & 
A^0_n \ar[r] \ar[d]^{a_n} & 
B^0_n \ar[r] \ar[d]^{b_n} & 
C^0_n  \ar[d]^{c_n}  \\
0 \ar[r] & A^1_n \ar[r] & B^1_n 
\ar[r] & C^1_n   
} 
\]
whose rows are exact, where 
\begin{align*}
A^1_n:=& H^0\left(K_n,
\prod_{w \mid v}E(K^E_{n,w})[p^\infty]\otimes_{\Z_p}(\Z/p^n\Z) \right), \\
B^1_n:=& H^0\left(K_n,
\prod_{w \mid v}H^1_f(K^E_{n,w}, E[p^n]) \right), \\
C^1_n:=& H^0\left(K_n,
\prod_{w \mid v}H^1_f(K^E_{n,w}, E[p^\infty])[p^n] \right), 
\end{align*}
and the arrows $a_n$ and 
$c_n$ are restriction maps, and 
$b_n = \res_{n,v}^{f}$.
By \autoref{lemloccond0}, 
we have $C^0_n=C^1_n=0$.
In order to prove \autoref{claimlocalres} (3), 
it suffices to show that 
the sequence $\set{ \# \Coker(a_n)}_{n \ge N}$ is bounded.
By the exact sequence 
\[
0 \longrightarrow 
p^nE(K^E_{n,w})[p^\infty]
\longrightarrow
E(K^E_{n,w})[p^\infty]
\longrightarrow
E(K^E_{n,w})[p^\infty]\otimes_{\Z_p}(\Z/p^n\Z)
\longrightarrow 0,
\] 
using Shapiro's lemma as in \eqref{eq:H0vw}, 
we obtain an exact sequence 
\begin{align*}
E(K_{n,v})[p^\infty]
 \xrightarrow{\ a_n\ } A^1_n
 \longrightarrow 
H^1\left(K^E_n/K_n, 
\prod_{w \mid v}p^nE(K^E_{n,w})[p^\infty]\right)=:\Xi_n.
\end{align*}
In order to prove that 
$\set{ \# \Coker (a_n) \mathrel{\vert} n \in \Z_{\ge N} }$ is bounded, 
it suffices to show that $\set{\# \Xi_n }_n$ is bounded.
Fix a place $w_n$ of $K^E_n$. 
We have 
\[
\Xi_n=H^1(K^E_{n,w_n}/K_{n,v}, p^nE(K^E_{n,w_n})[p^\infty]).
\]
For any intermediate field $M$ of 
$K^E_{n,w_n}/K_{n,v}$ which is Galois over $K_{n,v}$, 
we have the inflation-restriction exact sequence
\[
0 \longrightarrow 
\Xi'_n(M) \longrightarrow 
\Xi_n
\longrightarrow 
\Xi''_n(M),
\]
where we put 
\begin{align*}
\Xi'_n(M) := &
H^1\left(
M/K_{n,v}, 
H^0(
M, 
p^nE(K^E_{n,w_n})
[p^\infty]
)
\right),\ \mbox{and} \\
\Xi''_n(M)
:=& H^0\left(
K_{n,v}, H^1\left(K^E_{n,w_n}/M, p^nE(K^E_{n,w_n})[p^\infty]
\right)\right).
\end{align*}
Recall that $K^E_n$ contains $L$, the elliptic curve 
$E_{K^E_n}$ has semistable reduction everywhere. 
Let $u_n$ be the place of 
$L_n:=L(\mu_{p^n})$ below $w_n$.

\smallskip
\noindent 
{\bf (The case:~Good reduction)} 
Suppose that $E_{L_{n,u_n}}$ 
has  good reduction. 
Let $M_n$ be the maximal subfield of $K^E_{n,w_n}$ which is 
unramified over $K_{n,v}$. 
By similar arguments of 
the boundedness of $\set{\# Y_n(M_n)}_{n \ge 0}$
for {\bf (The case:~Potentially good reduction with $\ell \ne p$)} in 
the proof of (1), 
we have 
\begin{align*}
	 \# \Xi'_n(M_n) &= \# H^1(M_n/K_{n,v}, p^nE(K_{n,w_n}^E)[p^{\infty}]) \\
	 &\le \# \left( \frac{p^nE(M_n)[p^{\infty}]}{\braket{\mathrm{Frob}_v-1}} \right)\\
	 &= \# p^nE(K_{n,v})[p^{\infty}]\\
	 &\le \# E(K_{n,v})[p^{\infty}].
\end{align*}
\autoref{thmboundpotgoodred} implies that  
the sequence $\set{ \# \Xi'_n(M_n) }_{n \ge 0}$ is bounded.
Moreover, as noted in the proof of 
the boundedness of $\set{\# Z_n(M_n)}_{n \ge 0}$ 
in {\bf (The case:~Potentially good reduction with 
$\ell \ne p$)}, 
the sequence 
$\set{[K^E_n:M_n]}_{n \ge N}$ is bounded, 
and hence $\set{\# \Xi''_n(M_n) }_{n \ge N}$ is bounded.

\smallskip
\noindent
{\bf (The case:~Multiplicative reduction)}
Suppose that 
$E_{L_{n,u_n}}$ has multiplicative reduction. 
Put $\Xi'_n:=\Xi'_n(L_{n,u_n})$ and 
$\Xi''_n:=\Xi''_n(L_{n,u_n})$.
In this case, \autoref{lemboundmultred} implies that  
$\Xi'_n=0$ and 
$\Xi''_n=0$ for sufficiently large $n$, 
and in particular, the sequences 
$\set{ \# \Xi''_n }_{n \ge N}$ and 
$\set{ \# \Xi''_n }_{n \ge N}$ is bounded. 

By the above arguments, we deduce that in any case,  the set
$\set{ \# \Xi''_n }_{n \ge N}$ is bounded 
and so is 
$\set{ \# \Xi_n }_{n \ge N}$.
Accordingly, the assertion \autoref{claimlocalres} (3) is proved. 

\smallskip
Let us show the assertion (2).
Here, we study the case when $\ell = p$. 
Recall that by our assumption, 
the elliptic curve $E$ has good reduction at $p$.

\smallskip\noindent 
{\bf (The case: Good ordinary reduction)}
Suppose that 
the elliptic curve $E$ has good ordinary reduction at  $p$. 
In this case, 
there exists a $G_{\Q_p}$-stable $\Z/p^n\Z$-submodule 
$\mathrm{Fil} E[p^n]$ of $E[p^n]$ of rank one 
such that the inertia group 
$I_{\Q_p}$ of  $G_{\Q_p}$  
acts via the cyclotomic character on $\mathrm{Fil} E[p^n]$, 
and  trivially on $E[p^n]/\mathrm{Fil} E[p^n]$.
Fix a generator $P_n$ of 
the cyclic $\mathbb{Z}_p$-module $\mathrm{Fil} E[p^n]$ and 
a lift $Q_n \in E[p^n]$ of a generator of 
the cyclic $\mathbb{Z}_p$-module 
$\overline{Q}_n \in E[p^n]/\mathrm{Fil} E[p^n]$.
The pair $(P_n,Q_n)$ becomes a basis of 
the free $\mathbb{Z}/p^n\mathbb{Z}$-module
of rank two.
Let $M_n$ be the maximal subfield of $K^E_{n,w_n}$ which is 
unramified over $K_{n,v}$, and 
put $I_n:=\Gal(K^E_{n,w_n}/M_n)$. 
Since $I_n$ acts trivially on 
$\mathrm{Fil} E[p^n]$ and 
$E[p^n]/\mathrm{Fil} E[p^n]$, 
the group $I_n$ is a cyclic group which is 
generated by an element acting on 
$E[p^n]$ via a unipotent matrix 
$$
U=\begin{pmatrix}
1 & x_n \\
0 & 1
\end{pmatrix} \in M_2(\mathbb{Z}/p^n\mathbb{Z})
$$ 
under the basis $(P_n,Q_n)$. 
Fix a lift $\tau \in \Gal(K^E_{n,w_n}/K_{n,v})$ 
of the Frobenius  $\Frob_v \in \Gal(M_n/K_{n,v})$.
The filtration 
$\mathrm{Fil} E[p^n]$ is stable under the action of 
$\Gal(K^E_{n,w_n}/K_{n,v})$, 
and the Weil pairing 
$e\colon E[p^n] \times E[p^n] \longrightarrow \mu_{p^n}$
is an alternative pairing preserving the action of 
$\Gal(K^E_{n,w_n}/K_{n,v})$ (\cite[Chapter III, Section 8]{Si1}).
Accordingly, 
the fixed lift $\tau$ acts on $E[p^n]$ by a matrix 
\begin{equation}\label{eq:matpresoftau}
A=\begin{pmatrix}
a_n & b_n \\
0 & a_n^{-1}
\end{pmatrix} \in M_2(\mathbb{Z}/p^n\mathbb{Z})
\end{equation}
 for some $a_n \in (\mathbb{Z}/p^n\mathbb{Z})^\times$
and $b_n \in \mathbb{Z}/p^n\mathbb{Z}$.
We can define $a:=(a_n)_n \in \varprojlim_n 
(\mathbb{Z}/p^n\mathbb{Z})^\times=\mathbb{Z}_p^\times$. 
Since $E$ has good reduction at $p$, 
\autoref{thmboundpotgoodred} 
implies that $a^k \ne 1$ 
for any $k \in \mathbb{Z}_{>0}$. 
In fact, if $a^k = 1$, then 
for any $m \in \mathbb{Z}_{\ge 0}$, 
the group $\mathrm{Fil} E[p^m]$ 
of order $p^m$ is contained in 
$E(\mathbb{Q}_{p^k}(\mu_{p^m}))$, 
and contradicts to \autoref{thmboundpotgoodred}.
Here, we denote by $\mathbb{Q}_{p^k}$ 
the unramified  
extension field of $\mathbb{Q}_p$ 
of degree $k$. 
It holds that 
\begin{equation}\label{eq:actionofAonU}
AUA^{-1}=\begin{pmatrix}
1 & a_n^2x_n \\
0 & 1
\end{pmatrix} 
=U^{a_n^2}.
\end{equation}
By the short exact sequence 
$0 \longrightarrow 
\mathrm{Fil} E[p^n] \longrightarrow 
E[p^n]\longrightarrow 
E[p^n]/\mathrm{Fil} E[p^n] 
\longrightarrow 0$ 
and \eqref{eq:Kerloc}, 
we obtain an exact sequence 
\begin{equation}
    \label{seq:YKZp}
Y_n \longrightarrow 
\Ker( \mathrm{res}_{n,v}^{\mathrm{loc}})
\longrightarrow 
Z_n, 
\end{equation}
where 
\begin{align*}
Y_n &:= H^1(K^E_{n, w_n}/K_{n,v_n}, \mathrm{Fil} E[p^n]), \ \mbox{and}\\
Z_n &:= H^1(K^E_{n, w_n}/K_{n,v_n}, 
E[p^n]/\mathrm{Fil} E[p^n] ).
\end{align*}
In order to show that 
$\set{ \# \Ker( \mathrm{res}_{n,v}^{\mathrm{loc}})}_{n \ge 0}$ is bounded, 
it is sufficient to prove that both
$\set{ \# Y_n }_{n \ge 0}$
and 
$\set{ \# Z_n }_{n \ge 0}$ 
are bounded. 

First, let us study the order of 
$Z_n$.
Since $I_n$ acts trivially on 
$E[p^n]/\mathrm{Fil} E[p^n]$, 
we have an exact sequence 
\begin{equation}
	\label{seq:ZZZ}
0 \longrightarrow 
Z_n'
\longrightarrow 
Z_n \longrightarrow 
Z_n'',
\end{equation}
where 
\begin{align*}
	Z_n' &:= H^1(M_n/K_{n,v_n}, E[p^n]/\mathrm{Fil} E[p^n]),\ \mbox{and}\\ 
	Z_n''&:= H^0(K_{n,v}, 
H^1(K^E_{n,w_n}/M_n, E[p^n]/\mathrm{Fil} E[p^n]) )\\
&=H^0(K_{n,v}, 
\Hom(I_n,E[p^n]/\mathrm{Fil} E[p^n])).
\end{align*}
Since $a \ne 1$,  we have 
$
\#(E[p^\infty]/\mathrm{Fil} E[p^\infty])[a^{-1}-1]<\infty
$,
 and 
\begin{align*}
\# Z_n' 
& \le \# \left(
\frac{E[p^n]/\mathrm{Fil} E[p^n]}{\braket{\tau -1}}
\right) \\
&= \# \left(
\frac{E[p^n]/ \mathrm{Fil} E[p^n])}{\braket{a^{-1}-1}}
\right) \\
&=\# (E[p^n]/ \mathrm{Fil} E[p^n])[a^{-1}-1] \\
&\le  
\# (E[p^\infty]/ \mathrm{Fil} E[p^\infty])[a^{-1}-1].
\end{align*}
The sequence 
$\set{ Z_n'}_{n\ge 0}$ is bounded. 
Let us consider the order of 
$Z_n''$.
The matrix presentation \eqref{eq:matpresoftau} 
implies that the Galois group 
$\Gal(M_n/K_{n,v})=\langle \Frob_v \rangle$ 
acts on 
$E[p^n]/\mathrm{Fil} E[p^n]$ via the character 
$\Frob_v \longmapsto a_n^{-1}$, and 
\eqref{eq:actionofAonU} implies that 
$\Frob_v \in \Gal(M_n/K_{n,v})$ acts on 
$I_n$ via the character 
$\Frob_v \longmapsto a_n$. 
Since $a^3 \ne 1$, namely 
$a^2 \ne a^{-1}$,  there exists an integer 
$m_0 \in \mathbb{Z}_{>0}$ such that 
$a^2_{m_0} \ne a^{-1}_{m_0}$.
We have 
\[
Z_n''\subseteq \Hom(I_n,E[p^{m_0-1}]/\mathrm{Fil} E[p^{m_0-1}]).
\]
Since $I_n$ is cyclic, the sequence 
$\set{ \# \Hom(I_n,E[p^{m_0-1}]/
\mathrm{Fil} E[p^{m_0-1}]) }_{n \ge 0}$
is bounded 
and so is $\set{ Z_n'' }_{n \ge 0}$. 
As a result, the sequence 
$\set{ \# Z_n }_{n\ge 0}$ is bounded 
from \eqref{seq:ZZZ}. 

The boundedness of 
$\set{ \# Y_n }$ 
follows from the arguments in the 
previous paragraph just by 
replacing $E[p^n]/\mathrm{Fil} E[p^n]$
with $\mathrm{Fil} E[p^n]$, where 
the Galois group 
$\Gal(M_n/K_{n,v})=\langle \Frob_v \rangle$ 
acts via the character 
$\Frob_v \longmapsto a_n$. 
By the short exact sequence \eqref{seq:YKZp}  
we deduce that 
$\set{ \# \Ker( \mathrm{res}_{n,v}^{\mathrm{loc}}) }_{n \ge 0}$ is bounded.

\smallskip
\noindent {\bf (The case:~Good supersingular reduction)}
Suppose that $E$ 
has good supersingular reduction at  $p$.
In order to prove that 
the sequence 
$\set{\# \Ker(\mathrm{res}_{n,v}^{\mathrm{loc}})| n \ge 0, v \mid p}$
is bounded, 
by  \eqref{eq:Kerloc}
it suffices to show that 
$H^1(K^E_{n,w_n}/K_{n,v}, E[p^n])=0$ for any $n\ge 0$. 
The short exact sequence $0\longrightarrow E[p] \longrightarrow E[p^{m+1}]\xrightarrow{\,\times p\,} E[p^m]\longrightarrow 0$ 
induces the exact sequence 
\[
H^1(K^E_{n,w_n}/K_{n,v}, E[p])\longrightarrow H^1(K^E_{n,w_n}/K_{n,v}, E[p^{m+1}])\longrightarrow H^1(K^E_{n,w_n}/K_{n,v}, E[p^m]).
\]
By induction on $m$, 
it is enough to show that 
$H^1(K^E_{n,w_n}/K_{n,v}, E[p])=0$.
We denote  the inertia subgroup of $G_{\Qp}$ by 
$I_{\Qp} := \Gal(\overline{\Qp}/\Qp^{\ur})$, and  
the wild inertia subgroup by 
$I_{\Qp}^w := \Gal(\overline{\Qp}/\Qp^{\mathrm{tame}}) \subseteq I_{\Qp}$, 
where $\Qp^{\mathrm{tame}}$ is the maximal tamely ramified extension of $\Qp$.
Let $I_{\Q_p}^t:= I_{\Qp}/I_{\Qp}^w \simeq \varprojlim_n \F_{p^n}^\times$ 
be the tame inertia group of $G_{\Q_p}$ (cf.\ \cite[1.3, Proposition 2]{Se3}), 
and $\psi \colon I_{\Q_p}^t  
\longrightarrow \F_{p^2}^\times$ 
the  character  
induced by 
the natural projection 
$\varprojlim_n \F_{p^n}^\times 
\longrightarrow \F_{p^2}^\times$. 
The characters 
$\psi$ and  $\psi^p$ 
form the fundamental characters of level 2 (cf.~\cite[1.7]{Se3}).
By \cite[1.11, Proposition~12]{Se3}, the following hold.
\begin{itemize}
	\item The action of the wild inertia  subgroup $I_{\Qp}^w$ on $E[p]$ is trivial, so that 
	the action of the inertia group $I_{\Q_p}$
of $G_{\Q_p}$ on $E[p]$ 
factors through $I_{\Q_p}^t$. 
	\item The group $E[p]$ has a structure of $\F_{p^2}$-vector space of dimension $1$. 
	\item The image of $I_{\Qp}$ in $\Aut(E[p])$ is a cyclic group of order $p^2-1$.
	\item The action of $I_{\Qp}^t$ on $E[p]$ is given by the fundamental character $\psi$ of level $2$. 
\end{itemize}
Let us regard $E[p]$ as an $\mathbb{F}_{p}$-vector space, 
and consider the 
$\mathbb{F}_{p^2}$-vector space $E[p]\otimes_{\Fp} \mathbb{F}_{p^2}$, 
which is the extension of scalar of $E[p]$.
By the properties of $E[p]$ noted above, 
the action of $I_{\Qp}^t$ on 
$E[p]\otimes_{\Fp} \mathbb{F}_{p^2}$ 
is given by the matrix
\begin{equation}
\label{eq:psi-psip}
\begin{pmatrix}
\psi & 0 \\
0 & \psi^p
\end{pmatrix}
\end{equation}
after taking a suitable $\F_{p^2}$-basis 
$E[p] \otimes_{\F_p} \F_{p^2}$ 
(cf.\ \cite[1.9, Corollaire~3]{Se3}, see also \cite[2.6 Theorem]{Ed} which is a result on 
modulo $p$ Galois representations attached to 
modular forms with coefficients in $\overline{\F}_p$). 
Let $F$ be the maximal unramified extension field of 
$\Q_p$ contained in $K^E_{1,w_1}$.  
Put $F_n := F(\mu_{p^n})$. 
We have the following inflation-restriction exact sequences:  
\begin{equation}
\label{seq:KnE/Kn}
\begin{tikzcd}
H^1(F_n/\Q_p(\mu_{p^n}), H^0(F_n, E[p])) \rar 
             \ar[draw=none]{d}[name=X, anchor=center]{}
    &  H^1(K^E_{n,w_n}/K_{n,v}, E[p])
    \ar[rounded corners,
            to path={ -- ([xshift=2ex]\tikztostart.east)
                      |- (X.center) \tikztonodes
                      -| ([xshift=-2ex]\tikztotarget.west)
                      -- (\tikztotarget)}]{dl} \\      
  H^1(K^E_{n,w_n}/F_n, E[p])
\end{tikzcd}
\end{equation}
and 
\begin{equation}
\label{seq:KnE/Fn}
\begin{tikzcd}[column sep=0.3cm]
H^1(K^E_{1,w_1}(\mu_{p^n})/F_n, E[p])
\rar 
    \ar[draw=none]{d}[name=X, anchor=center]{}
    &  H^1(K^E_{n,w_n}/F_n, E[p])
    \ar[rounded corners,
            to path={ -- ([xshift=2ex]\tikztostart.east)
                      |- (X.center) \tikztonodes
                      -| ([xshift=-2ex]\tikztotarget.west)
                      -- (\tikztotarget)}]{dl} \\      
   H^1(K^E_{n,w_n}/K^E_{1,w_1}(\mu_{p^n}), E[p])^{G_{F_n}}\arrow[r,equal] 
 & \Hom_{\Z[G_{F_n}]}(\Gal(F'_n/K^E_{1,w_1}(\mu_{p^n})), E[p]),
\end{tikzcd}
\end{equation}
where $F'_n$ is the maximal abelian extension field of $K^E_{1,w_1}(\mu_{p^n})$
contained in $K^E_{n,w_n}$. 

\begin{claim}
\label{clm:H0(Fn,Ep)}
We have 
$H^0(F_n, E[p]) = 0$ 
and 
$H^1(K^E_{1,w_1}(\mu_{p^n})/F_n, E[p]) = 0$. 
\end{claim}
\begin{proof}
We may assume $n \ge 1$.
Since $F/\Q_p$ is unramified, 
the ramification index of $F_n/\Q_p$ is $(p-1)p^{n-1}$, 
which is not divisible by 
$[K^E_{1,w_1}:F]=p^2-1$. 
This implies that 
the restrictions of $\psi$ and $\psi^p$ on $I_{\Q_p} \cap G_{F_n}$ 
are non-trivial, and 
by \eqref{eq:psi-psip}, we have 
\[
H^0(F_n, E[p]) \subseteq 
H^0(F_n, E[p]\otimes_{\Fp} \mathbb{F}_{p^2})=0.
\]

Furthermore, the extension  
$K^E_{1,w_1}(\mu_{p^n})/F_n$ is finite cyclic. 
By using the Herbrand quotient of 
the Tate cohomology groups (\cite[{{\sc Chapitre}~VIII, \S 4, {\sc Proposition}~8}]{Se1}), 
we have 
\begin{align*}
    \# H^1(K^E_{1,w_1}(\mu_{p^n})/F_n, E[p]) &= \# \widehat{H}^1(K^E_{1,w_1}(\mu_{p^n})/F_n, E[p]) \\
    &= \#\widehat{H}^0(K^E_{1,w_1}(\mu_{p^n})/F_n, E[p]) \\
    &\le \# H^0(F_n,E[p]) = 1.
\end{align*}
Because of this, we obtain the claim. 
\end{proof}
Applying  \autoref{clm:H0(Fn,Ep)}, 
the exact sequences \eqref{seq:KnE/Kn} and \eqref{seq:KnE/Fn} 
give 
\begin{align*}
    \# H^1(K^E_{n,w_n}/K_{n,v}, E[p]) &\le 
\# H^1(K_{n,w_n}^E/F_n,E[p]) \\
&\le \#\Hom_{\Z[G_{F_n}]}(\Gal(F'_n/K^E_{1,w_1}(\mu_{p^n})), E[p]).
\end{align*}
Now, we shall show that 
\begin{equation}\label{eq:HomF'nvanishing}
    \Hom_{\Z[G_{F_n}]}(\Gal(F'_n/K^E_{1,w_1}(\mu_{p^n})), E[p])=0.
\end{equation}
For each $m \in \Z$ with $1 \le m \le n$, we define 
the subgroup $\mathrm{Fil}^m$ of $\Gal(F'_n/K^E_{1,w_1}(\mu_{p^n}))$ 
to be the image of $\Gal(K^E_{n,w_n}/K^E_{m,w_m}(\mu_{p^n}))$ 
by the natural map 
\[
\Gal(K_{n,w_n}^E/K_{1,w_1}^E(\mu_{p^n})) 
\longrightarrow \Gal(F'_n/K_{1,w_1}^E(\mu_{p^n})). 
\]
The family $\set{ \mathrm{Fil}^m}_m$ becomes 
a  $G_{F_n}$-stable 
descending filtration  of 
$\Gal(F'_n/K^E_{1,w_1}(\mu_{p^n}))$. 
In order to show \eqref{eq:HomF'nvanishing}, 
it suffices to show that 
\[
\Hom_{\Z[G_{F_n}]}(\mathrm{Fil}^m/\mathrm{Fil}^{m+1}, E[p] 
\otimes_{\Zp} \Z_{p^2})=0.
\]
Take an $\mathbb{F}_{p^2}$-basis $B_1$ of 
$E[p]\otimes_{\Fp}\mathbb{F}_{p^2}$ which gives 
the presentation of 
the action of $I_{\Qp}^t$ by 
the matrix \eqref{eq:psi-psip}, and 
for each $m \in \Z$ with $2 \le m \le n$, 
fix a basis $B_m$ of 
$E[p^m]\otimes_{\Zp}\mathbb{Z}_{p^2}$
which is a lift of $B_1$. 
Since $\Gal(K_{m+1,w_{m+1}}^E/K_{m,w_{m}}^E)$ is 
a normal subgroup of 
$\Gal(K_{m+1,w_{m+1}}^E/\Qp)$, 
it is stable under the conjugate action of $G_{F_n}$.
Recall that we have a $G_{F_n}$-stable injection 
\[
\xymatrix@R=3mm{
\Gal(K_{m+1,w_{m+1}}^E/K_{m,w_{m}}^E) \ar@{^{(}->}[r] & \Ker\bigg(\Aut(E[p^{m+1}]\otimes_{\Zp}\Z_{p^2}) \twoheadrightarrow \Aut(E[p^{m}]\otimes_{\Zp}\Z_{p^2})\bigg)\ar@{=}[d]\\
&  1 + p^{m}M_2(\Z_{p^2}/p^{m+1}\Z_{p^2}) 
\simeq M_2(\mathbb{F}_{p^2}),
}
\]
where the action of $\sigma \in G_{F_n}$ 
on $M_2(\F_{p^2})$ is defined by the conjugate action of the matrix 
\[
\begin{pmatrix}
\psi(\sigma) & 0 \\
0 & \psi^p(\sigma)
\end{pmatrix}.
\]
Since 
$\mathrm{Fil}^m/\mathrm{Fil}^{m+1}$ is 
a quotient of $\Gal(K_{m+1,w_{m+1}}^E(\mu_{p^n})/
K_{m,w_{m}}^E(\mu_{p^n}))$
by definition, and the restriction 
\[\Gal(K_{m+1,w_{m+1}}^E(\mu_{p^n})/
K_{m,w_{m}}^E(\mu_{p^n})) 
\longrightarrow 
\Gal(K_{m+1,w_{m+1}}^E/
K_{m,w_{m}}^E)
\]
is an injective homomorphism, 
we can regard 
$\mathrm{Fil}^m/\mathrm{Fil}^{m+1}$ as 
a $G_{F_n}$-stable subquotient of $M_{2}(\mathbb{F}_{p^2})$. 
Let us study the $\mathbb{F}_{p^2}[G_F]$-module structure of 
$M_{2}(\mathbb{F}_{p^2})$. 
Take any $\sigma \in G_{F_n}$. 
It holds that 
\begin{align*}
\begin{pmatrix}
\psi(\sigma) & 0 \\
0 & \psi^p(\sigma)
\end{pmatrix}
\begin{pmatrix}
a & 0 \\
0 & b
\end{pmatrix}
\begin{pmatrix}
\psi(\sigma) & 0 \\
0 & \psi^p(\sigma)
\end{pmatrix}^{-1}
&=\begin{pmatrix}
a & 0 \\
0 & b
\end{pmatrix} \ 
\text{for any $a,b \in \F_{p^2}$},  \\
\begin{pmatrix}
\psi(\sigma) & 0 \\
0 & \psi^p(\sigma)
\end{pmatrix}
\begin{pmatrix}
0 & 1 \\
0 & 0
\end{pmatrix}
\begin{pmatrix}
\psi(\sigma) & 0 \\
0 & \psi^p(\sigma)
\end{pmatrix}^{-1}
&=\psi^{1-p}(\sigma)\begin{pmatrix}
0 & 1 \\
0 & 0
\end{pmatrix},  \\ \intertext{and}
\begin{pmatrix}
\psi(\sigma) & 0 \\
0 & \psi^p(\sigma)
\end{pmatrix}
\begin{pmatrix}
0 & 0 \\
1 & 0
\end{pmatrix}
\begin{pmatrix}
\psi(\sigma) & 0 \\
0 & \psi^p(\sigma)
\end{pmatrix}^{-1}
&={\psi}^{p-1}(\sigma)\begin{pmatrix}
0 & 0 \\
1 & 0
\end{pmatrix}.
\end{align*}
Note that $\psi \ne \psi^{p-1}$, and $\psi \ne \psi^{1-p}$. 
It holds that $M_2(\mathbb{F}_{p^2})$ is 
a semisimple $\mathbb{F}_{p^2}[G_{F_n}]$-module, and 
there is no simple $\mathbb{F}_{p^2}[G_{F_n}]$-submodule of $M_2(\mathbb{F}_{p^2})$ 
which is isomorphic to an $\mathbb{F}_{p^2}[G_{F_n}]$-submodule of 
$E[p]\otimes_{\Zp} \Z_{p^2}$.
This implies 
\[
\Hom_{\Z[G_{F_n}]}(\mathrm{Fil}^m/\mathrm{Fil}^{m+1}, E[p] 
\otimes_{\Zp} \Z_{p^2})=0, 
\]
and 
we obtain \eqref{eq:HomF'nvanishing}. 
Consequently, we have 
\[
\Ker(\mathrm{res}_{n,v}^{\mathrm{loc}})=
H^1(K^E_{n,w_n}/K_{n,v}(\mu_{p^n}), E[p^n])=0.
\]
By the above arguments, 
we deduce that 
$\set{\# \Ker (\mathrm{res}_{n,v}^{\mathrm{loc}}) }_{n \ge 0, v \mid p}$
is bounded. 
This completes the proof of 
\autoref{claimlocalres} (2).
\end{proof}

\subsection{Proof of \autoref{thmmain}}\label{sspfmain}

%
%
%
%
%
%

Recall that $\Sigma_{0,\mathrm{bad}}$ denotes the set of 
prime numbers where $E$ has bad reduction. 
As $E$ has good reduction at $p$ 
the prime $p$ does not belong to $\Sigma_{0,\mathrm{bad}}$.

\begin{lem}\label{lemfur}
Suppose that $E$ satisfies {\rm (C2)}.
Let $\ell \in \Sigma_{0,\mathrm{bad}}$. 
For each $n \in \Z_{\ge 0}$ 
and $i \in \set{ 0,1,2 }$, 
we put 
\begin{align*}
\cH^i_f(\ell, n):=&
H^i\left(
K_n, 
\prod_{w \mid \ell}
\frac{H^1_{ f}(K^E_{n,w},E[p^n])}{
H^1_{\rm ur}(K^E_{n,w},E[p^n]) \cap H^1_{f}(K^E_{n,w},E[p^n])}
\right),\ \mbox{and} \\
\cH^i_{\rm ur}(\ell, n):=&
H^i\left(
K_n, 
\prod_{w \mid \ell}
\frac{H^1_{\rm ur}(K^E_{n,w},E[p^n])}{
H^1_{\rm ur}(K^E_{n,w},E[p^n]) \cap H^1_{f}(K^E_{n,w},E[p^n])}
\right).
\end{align*}
Then, there exists an integer $N'_\ell \in \Z_{\ge 1}$ 
such that for any $n \in \Z_{\ge N'_\ell}$ 
and $i \in \set{ 0,1,2 }$, it holds that  
$\cH^i_f(\ell, n)=0$ and 
$\cH^i_{\rm ur}(\ell, n)=0$.
\end{lem}

\begin{proof}
\noindent {\bf (The case:~Potentially good reduction at $\ell$)}
First, suppose that $E$ has potentially good reduction at $\ell$.
There exists an integer $n_0\in\Z_{\ge 1}$ such that 
$E_{K^E_{n_0}}$ has good reduction at every place above $\ell$ 
(\cite[Chapter IV, Proposition 10.3]{Si2}). 
For any $n \in \Z_{\ge n_0}$ and any place $w$ of $K^E_{n}$, 
we have $H^1_f(K^E_{n,w},E[p^n])=H^1_{\rm ur}(K^E_{n,w},E[p^n])$
(cf.\ \autoref{remH^1_f}).
We obtain $\cH^i_f(\ell, n)=0$ and 
$\cH^i_{\rm ur}(\ell, n)=0$
for any $n \in \Z_{\ge n_0}$ and $i \in \set{ 0,1,2 }$.

\noindent 
{\bf (The case:~Potentially multiplicative reduction at $\ell$)}
Next, suppose that 
$E$ has potentially multiplicative reduction at $\ell$.
Let $N_\ell\in \Z_{\ge 1}$ be as in  \autoref{lemboundmultred}.
By \autoref{lem:K1E}, the base change $E_{K^E_{N_\ell}}$ has  
split multiplicative reduction at every $w \mid \ell$.
Take any $n \in \Z_{\ge N_{\ell}}$, and 
let $v$ be any place of $K_n$ above $\ell$.
For any place $w$ of $K^E_n$ above $v$, 
$E_{K^E_{n,w}}$ is isomorphic to 
a Tate curve $\mathbb{G}_m/q_w^{\Z}$.
By Shapiro's lemma as in \eqref{eq:H0vw}, 
we have 
\[
\cH^i_{\mathcal{F}}(\ell, n)\simeq 
H^i\left(
K_{n,v}, 
\frac{H^1_{\mathcal{F}}(K^E_{n,w},E[p^n])}{
H^1_{\rm ur}(K^E_{n,w},E[p^n]) \cap H^1_{f}(K^E_{n,w},E[p^n])}
\right)
\]
for each $\mathcal{F} \in \set{ f, \textrm{ur} }$ and 
$i \in \set{ 0,1,2 }$.

Let us show that 
$\cH^i_{f}(\ell, n)=0$ for each $i$.
The natural surjective homomorphism 
$T_p(E) \longrightarrow T_p(E)/p^nT_p(E) 
\simeq E[p^n]$ induces 
a map 
\[
\pi_{n,w}\colon 
H^1(K^E_{n,w},T_p(E))
\longrightarrow 
H^1(K^E_{n,w},E[p^n]).
\]
We note that $H^1_{\rm ur}(K^E_{n,w},T_p(E))$ is 
contained in the inverse image 
of $H^1_{\rm ur}(K^E_{n,w},E[p^n])$ by 
 the map $\pi_{n,w}$.
By \cite[Lemma~1.3.8]{Ru}, the image of $H^1_f(K^E_{n,w},T_p(E))$
by $\pi_{n,w}$ coincides with $H^1_f(K^E_{n,w},E[p^n])$. 
Here, we have $H^1_{\ur}(K^E_{n,w},T_p(E)) \subseteq H^1_{f}(K^E_{n,w},T_p(E))$ 
with finite index
(\cite[Lemma~1.3.5 (ii)]{Ru}). 
The map $\pi_{n,w}$ induces a surjection
\[
\pi_{n,w}^f \colon
\frac{H^1_f(K^E_{n,w},T_p(E))}{H^1_{\rm ur}(K^E_{n,w},T_p(E))}
\longrightarrow 
\frac{H^1_f(K^E_{n,w},E[p^n])}{H^1_{\ur}(K^E_{n,w},E[p^n]) \cap 
H^1_{f}(K^E_{n,w},E[p^n])}.
\]
By \cite[Lemma~1.3.5 (iii)]{Ru}, we have 
\[
\frac{
H^1_f(K^E_{n,w},T_p(E))}{
H^1_{\rm ur}(K^E_{n,w},T_p(E))}
=\left( \frac{E(K_{n,w}^{E,\ur})[p^\infty]}{
E(K_{n,w}^{E,\ur})[p^\infty]_{\div}}
\right)^{\Frob_w =1},
\]
where $\Frob_w \in \Gal(K_{n,w}^{E,\ur}/K_{n,w}^E)$ 
is the Frobenius automorphism.
Since $\pi_{n,w}^f$ is surjective, 
all 
the Jordan--H\"older constituents $J_i/J_{i-1}$ 
of the composition series 
\[
0 = J_0 \subseteq J_1\subseteq \cdots \subseteq J_t := \frac{H^1_{ f}(K^E_{n,w},E[p^n])}{
H^1_{\rm ur}(K^E_{n,w},E[p^n]) \cap H^1_{f}(K^E_{n,w},E[p^n])}
\]
as $\mathbb{Z}_p[G_{K_{n,v}}]$-modules 
are isomorphic to 
\[
\frac{E(K_{n,w}^{E,\ur})[p^\infty]}{E(K_{n,w}^{E,\ur})[p^\infty]_{\div}}
[p]
=
\frac{E[p^\infty]}{E(K_{n,w}^{E,\ur})[p^\infty]_{\div}}
[p]=(\mu_p \times q_w^{p^{-1}\Z})/(\mu_p \times q_w^{\Z}).
\] 
It follows from 
the condition (C2) for $E$ and \autoref{lemC2} that 
\[
H^0\left(K_{n,v},
\frac{E(K_{n,w}^{E,\ur})[p^\infty]}{E(K_{n,w}^{E,\ur})[p^\infty]_{\div}}
[p]
\right)
\subseteq  H^0\left(K_{n,v},
\frac{E(K_{n,w}^{E,{\rm ur}})[p^\infty]}{
E(K_{n,w}^{E,{\rm ur}})[p^\infty]_{\rm div}}
\right)= 0.
\]
By induction on $i$, we have 
$H^0(K_{n,v}, J_i) = 0$. 
In particular, we obtain 
\[
H^0(K_{n,v}, J_t) = \cH^0_f(\ell, n)=0.
\]
The condition (C2) for $E$ and \autoref{lemC2}  also imply the equality 
\[
H^0\left(K_{n,v},
\Hom_{\mathbb{Z}_p} \left(\frac{H^1_{ f}(K^E_{n,w},E[p^n])}{
H^1_{\rm ur}(K^E_{n,w},E[p^n]) 
\cap H^1_{f}(K^E_{n,w},E[p^n])},
\mu_{p^n}
\right)
\right)= 0.
\]
By the local duality of the Galois cohomology (\cite[(7.2.6) Theorem]{NSW}), 
we also have 
\(
\cH^2_f(\ell, n)
= 0
\). 
Moreover, as we have $\ell \ne p$, 
the local Euler--Poincar\'e characteristic 
\[
\frac{\# \cH^0_f(\ell, n)\#\cH^2_f(\ell, n)}{\# 
\cH^1_f(\ell, n)}
\] 
is equal to $1$ 
(\cite[(7.3.1) Theorem]{NSW}).
We obtain \(
\cH^1_f(\ell, n)
= 0
\).

Next, let us show that 
$\cH^i_{\rm ur}(\ell, n)=0$ for each $i$.
The inclusion $E[p^n] \subseteq E[p^\infty]$ 
induces a homomorphism 
\[
\iota_{n,w}:H^1(K^E_{n,w},E[p^n])\rightarrow H^1(K^E_{n,w},E[p^\infty]).
\] 
Recall that 
$H^1_f(K^E_{n,w},E[p^n])$
is defined to be the inverse image of $H^1_f(K^E_{n,w},E[p^{\infty}])$ 
by the natural map $\iota_{n,w}$
(cf.~\cite[Remark 1.3.9]{Ru}). 
From \autoref{lemloccond0}, 
we have 
\begin{equation}
\label{eq:Ker_iota}
H^1_f(K^E_{n,w},E[p^n])=\Ker( \iota_{n,w}).
\end{equation}
By \cite[Lemma 1.3.2 (i)]{Ru}, 
we have 
\[
 H^1_{\ur}(K^E_{n,w},E[p^n])
 \simeq \frac{E(K_{n,w}^{E,{\rm ur}})[p^n]}{\braket{\Frob_w -1}}.
\]
The latter group is isomorphic to 
$E[p^n] = E(K_{n,w}^E)[p^n]$ because of $K_n^E = \Q(E[p^n])$.
The image of 
$H^1_{\ur}(K^E_{n,w},E[p^n])$ 
by $\iota_{n,w}$ 
is contained in 
\[
H^1_{\rm ur}(K^E_{n,w},E[p^\infty])
\simeq 
\frac{E(K_{n,w}^{E,{\rm ur}})[p^\infty]}{\braket{\Frob_w -1}},
\] 
and  we have 
\[
\iota_{n,w}(
H^1_{\rm ur}(K^E_{n,w},E[p^n])
) 
= 
\left(
\frac{E(K_{n,w}^{E,{\rm ur}})[p^\infty]}{\braket{\Frob_w -1}}
\right)[p^n].
\]
By \eqref{eq:Ker_iota}, the map $\iota_{n,w}$ induces 
\[
\frac{H^1_{\rm ur}(K^E_{n,w},E[p^n])}{
H^1_{\rm ur}(K^E_{n,w},E[p^n]) \cap H^1_{f}(K^E_{n,w},E[p^n])}
\xrightarrow{\ \simeq\ }
\left(
\frac{E(K_{n,w}^{E,{\rm ur}})[p^\infty]}{\braket{\Frob_w -1}}
\right)[p^n].
\]
Therefore, by (C2) and \autoref{lemC2}, 
we have $\cH^0_{\rm ur}(\ell, n)=0$.
Moreover, similar to the proof of $\cH^i_f(\ell, n) = 0$, 
by using the local duality theorem and the local Euler--Poincar\'e characteristic formula, 
we deduce that $\cH^1_{\rm ur}(\ell, n)=0$ and 
$\cH^2_{\rm ur}(\ell, n)=0$.
This completes the proof of \autoref{lemfur}.
\end{proof}

\begin{cor}\label{corfur}
Suppose that $E$ satisfies {\rm (C2)}.
Let $\ell$ be a prime number (distinct from $p$)
at which $E$ has bad reduction. 
Then, 
there exists an integer $N_{\ell}' \in \Z_{\ge 1}$ 
such that for any $n \in \mathbb{Z}_{\ge N'_\ell}$  
and any $\cF \in \set{f, \mathrm{ur}}$, 
the natural map
\[
\left(  
\prod_{w \mid \ell}  
\frac{H^1(K_{n,w}^E,E[p^n])}{H^1_{f}(K_{n,w}^E,E[p^n]) 
\cap H^1_{\rm ur}(K_{n,w}^E,E[p^n])}
\right)^{G_{K_n}} \longrightarrow 
\left(
\prod_{w \mid \ell}  
\frac{H^1(K_{n,w}^E,E[p^n])}{H^1_{\mathcal{F}}(K_{n,w}^E,E[p^n]) }
\right)^{G_{K_n}}
\]
is an isomorphism.
\end{cor}

\begin{proof}
Take $N_{\ell}'\in \Z_{\ge 1}$ as in \autoref{lemfur}. 
For $n\ge N_{\ell}'$, 
to simplify the notation, we put $H^i_{\cF'}(K_{n,w}^E):=H^i_{\cF'}(K_{n,w}^E,E[p^n])$ 
 $(\cF' \in \set{\emptyset, \ur, f})$. 
The short exact sequences 
\[
0 \longrightarrow 
\frac{H^1_{\cF}(K_{n,w}^E)}{H^1_{f}(K_{n,w}^E) \cap H^1_{\rm ur}(K_{n,w}^E)} 
\longrightarrow  
\displaystyle  \frac{H^1(K_{n,w}^E)}{H^1_{f}(K_{n,w}^E) \cap H^1_{\rm ur}(K_{n,w}^E)}  \longrightarrow
\frac{H^1(K_{n,w}^E) }{H^1_{\cF}(K_{n,w}^E)} \longrightarrow 0
\]
for all place $w$ above $\ell$ 
%
induce the cohomological long exact sequence
\[
\xymatrix@C=5mm{
\cH^0_{\mathcal{F}}(\ell, n) \ar[r] 
& \displaystyle \Big(  
\prod_{w \mid \ell}  \frac{H^1(K_{n,w}^E)}{H^1_{f}(K_{n,w}^E) \cap H^1_{\rm ur}(K_{n,w}^E)}\Big)^{G_{K_n}} 
 \ar[r]^-{h} 
 & \displaystyle  \Big(
\prod_{w \mid \ell}  
\frac{H^1(K_{n,w}^E)}{H^1_{\mathcal{F}}(K_{n,w}^E) }
\Big)^{G_{K_n}}
\ar[r] & \cH^1_{\mathcal{F}}(\ell, n).
}
\]
\autoref{lemfur} implies that the map $h$ is an isomorphism. 
\end{proof}

Recalling from \autoref{secintro}, 
let 
\begin{align*}
	\rho_n^E &\colon \Gal(K_n^E/\Q) \longrightarrow \Aut_{\Zp}(E[p^n])\simeq GL_2(\Z/p^n),\ \mbox{and}\\ 
	(\rho_n^E)^{\vee}&\colon \Gal(K_n^E/\Q)^{\mathrm{op}} \longrightarrow \Aut_{\Zp}(E[p^n]^{\vee}) = GL_2(\Z/p^n)
\end{align*}
be the Galois representations arising from the action on $E[p^n]$ 
and the right action of the Pontrjagin dual $E[p^n]^{\vee} = \Hom_{\Zp}(E[p^n],\Z/p^n)$
of $E[p^n]$.
For each $n\in \Z_{\ge 0} \cup \set{\infty}$, we define 
$R_n = \Z/p^n \Z[\Gal(K_n/\Q)]$-modules 
\begin{align*}
S_{n} &:=
\Hom_{\Zp [\Gal(K^E_n/K_n)]}
(\Cl(\cO_{K^E_n}[1/p]) \otimes_{\Z} {\Z_p}, E[p^n]  ) \\ \intertext{and}
A_n^E &:= (M_2(\Z/p^n\Z), (\rho_n^E)^\vee)
\otimes_{\Z[\Gal(K^E_n/K_n)]} 
\Cl(\cO_{K^E_n}[1/p]).
\end{align*}

\begin{lem}
\label{lemAS}
	For each $n\in \Z_{\ge 1}$, there exists a $\Gal(K_n/\Q)$-equivariant isomorphism
\[
(A_n^E)^{\vee} \stackrel{\simeq}{\longrightarrow} S_n^{\oplus 2}.
\] 
\end{lem}
\begin{proof}
	There is a $\Gal(K^E_n/\mathbb{Q})$-equivariant 
isomorphism 
\[
(M_2(\Z/p^n\Z), \rho_n^E)
\simeq E[p^n]^{\oplus 2}
\]
so that we have $\Gal(K_n/\mathbb{Q})$-equivariant  isomorphisms 
\begin{align*}
\Hom_{\Z_p}(A_{n}^E, \Z/p^n\Z) 
& \xrightarrow{\ \simeq \ } \Hom_{\Z_p[\Gal(K^E_n/K_n)]}
\bigg( \Cl(\cO_{K^E_n}[1/p]) \otimes_{\Z} \Z_p, 
(M_2(\Z/p^n\Z), \rho_n^E) \bigg) \\
& \xrightarrow{\ \simeq \ } 
S_{n}^{\oplus 2}.
\end{align*}
This shows the assertion. 
\end{proof}

\begin{lem}
\label{lemSnSel}
There exists an integer $N\in\Z_{\ge 1}$ 
such that, for any $n \in \Z_{\ge N}$,  we have an isomorphism 
\[
S_n\simeq H^0(K_n,\Sel_p(K^E_n, E[p^n])).
\]
\end{lem}
\begin{proof}
Let $H_n^E$ be the maximal subextension of the $p$-Hilbert class field 
of $K_{n}^E$ which is completely split at primes above $p$. 
From the global class field theory, 
the ideal class group $\Cl(\cO_{K^E_n}[1/p]) \otimes_{\Z}\Zp$ 
is isomorphic to the Galois group 
$\Gal(H_n^E/K_n^E)$.
We have 
\begin{align*}
	&\Hom (\Cl(\cO_{K^E_n}[1/p]) \otimes_{\Z}\Zp,E[p^n])\\
	&\simeq \Hom (\Gal(H_n^E/K_n^E),E[p^n])\\
	&= \Ker\big(\Hom(G_{K_n^E},E[p^n]) \longrightarrow
	\prod_{w \mid p} \Hom(G_{K_{n,w}^E},E[p^n]) \times 
	\prod_{w \nmid p} \Hom(G_{K_{n,w}^{E,{\rm ur}}},E[p^n]) \Big)\\
	&\simeq \Ker\Big(H^1(G_{K_n^E},E[p^n]) \longrightarrow
	\prod_{w \mid p} H^1(G_{K_{n,w}^E},E[p^n]) \times 
	\prod_{w \nmid p} H^1(G_{K_{n,w}^{E,{\rm ur}}},E[p^n]) \Big)
\end{align*}
Therefore, we obtain
\begin{align*}
S_n &\simeq  \Hom_{\Zp[\Gal(K_{n}^E/K_n)]}(\Gal(H_n^E/K_n^E),E[p^n]) \\
 &\simeq  \Hom(\Gal(H_n^E/K_n^E),E[p^n])^{G_{K_n}} \\
&= \Ker\Big(H^1(K^E_n,E[p^n])^{G_{K_n}} \longrightarrow 
	\big(\prod_{w \mid p} H^1(K_{n,w}^E,E[p^n]) \times 
	\prod_{w \nmid p} H^1(K_{n,w}^{E,{\rm ur}},E[p^n]) \big)^{G_{K_n}}\Big).
\end{align*}
By the very definition of $H^1_{\ur}$, there exists an injective homomorphism
\begin{equation}
	\label{eq:H1/H1ur}
	H^1(K_{n,w}^E, E[p^n])/H^1_{\ur}(K_{n,w}^E,E[p^n]) \hookrightarrow H^1(K_{n,w}^{E,\ur},E[p^n])
\end{equation}
and hence $S_n$ is isomorphic to the kernel of 
\begin{equation}
	\label{eq:SnKer}
	H^1(K^E_n,E[p^n])^{G_{K_n}} \longrightarrow 
	\Big(\prod_{w \mid p} H^1(K_{n,w}^E,E[p^n]) \times 
	\prod_{w \nmid p} \frac{H^1(K_{n,w}^E,E[p^n])}{
	H^1_{\ur}(K_{n,w}^E,E[p^n])}\Big)^{G_{K_n}}.
\end{equation}
It follows from \autoref{corfur} 
that, for each prime $\ell \in \Sigma_{0,\mathrm{bad}}$, 
there exists 
an integer $N'_{\ell}\in\Z_{\ge 1}$ such that 
\begin{equation}
\label{eq:isom}
\left(\prod_{w \mid \ell}  
\frac{H^1(K_{n,w}^E,E[p^n])}{H^1_{f}(K_{n,w}^E,E[p^n]) }
\right)^{G_{K_n}}\simeq 
\left(\prod_{w \mid \ell}  
\frac{H^1(K_{n,w}^E,E[p^n])}{H^1_{\ur}(K_{n,w}^E,E[p^n]) }
\right)^{G_{K_n}}
\end{equation}
for any $n\ge N_{\ell}'$.
%
Now, we put $N := \max\set{ N'_{\ell} | \ell \in \Sigma_{0,\mathrm{bad}}}$. 
For any 
$n \ge N$, we have 
\begin{align*}
&H^0(K_n,\Sel_p(K^E_n, E[p^n])) \\ 
&=\Ker\Big(H^1(K^E_n,E[p^n])^{G_{K_n}} \longrightarrow 
\Big( \prod_{w \mid p} H^1(K_{n,w}^E,E[p^n]) 
\times  \prod_{w \nmid p} 
\frac{H^1(K_{n,w}^E,E[p^n])}{H^1_{f}(K_{n,w}^E,E[p^n])}\Big)^{G_{K_n}}\Big)\\
&\overset{(\diamondsuit)}{=} \Ker\Big(H^1(K^E_n,E[p^n])^{G_{K_n}} \longrightarrow 
	\Big(\prod_{w \mid p} H^1(K_{n,w}^E,E[p^n]) \times 
	\prod_{w \nmid p} \frac{H^1(K_{n,w}^E,E[p^n])}{
	H^1_{\ur}(K_{n,w}^E,E[p^n])}\Big)^{G_{K_n}}\Big)\\
&\overset{\eqref{eq:SnKer}}{\simeq}S_n.
\end{align*}
Here, the second equality ($\diamondsuit$) follows 
from \eqref{eq:isom} for a bad prime $\ell \neq p$ and 
\autoref{remH^1_f} for a good prime $\ell \neq p$. 
\end{proof}

\begin{thm}
\label{thmmainbody}
Suppose that $E$ satisfies the conditions 
{\rm (C1)}, {\rm (C2)} and {\rm (C3)}. 
Then, there exists a family of $R_n$-homomorphisms 
\[
r_n\colon \Sel(K_n,E[p^n])^{\oplus 2}\longrightarrow (A_n^E)^{\vee}
\]
such that 
the kernel $\Ker(r_n)$ and the cokernel $\Coker(r_n)$ are finite 
with order bounded independently of $n$. 
\end{thm}
\begin{proof}
By \autoref{proprescyctoKn} and \autoref{lemSnSel}, there exists 
$N \in \Z_{\ge 1}$, 
the order of the kernel and that of the cokernel of the map 
\[
\Sel_p(K_n, E[p^n])^{\oplus 2}
\xrightarrow{\left( \mathrm{res}_n^{\mathrm{Sel}} \right)^{\oplus 2} } H^0(K_n,\Sel_p(K_n^E,E[p^n]))^{\oplus 2} \simeq S_n^{\oplus 2}
\]
are at most $p^{2 \nu_{\mathrm{res}}^{\Ker}}$ and $p^{2 \nu_{\mathrm{res}}^{\Coker}}$ 
respectively for all $n\ge N$. 
By \autoref{lemAS}, there is an isomorphism $S_n^{\oplus 2} \simeq (A_n^E)^{\vee}$.  
Since $\Sel_p(K_n,E[p^n])^{\oplus 2}$ and $(A^E_n)^\vee$ 
are finite for any $n<N$, 
this completes the proof of \autoref{thmmainbody}. 
\end{proof}

%

\providecommand{\bysame}{\leavevmode\hbox to3em{\hrulefill}\thinspace}

\end{document}